\documentclass[10pt]{amsart}
\usepackage{custom_anna}
\hypersetup{
 pdfauthor = {Anna Fokma, Alvaro del Pino Gomez, Lauran Toussaint},
 pdftitle= {Wrinkling and Haefliger structures},
 pdfsubject = {Differential geometry, h-principle, jet spaces, groupoids, Haefliger structures}}

\usepackage[depth=3]{bookmark}

\newcommand{\gtop}{\mathrm{PGroup}} 

\newcommand{\Etale}{{J^\textrm{germs}}}
\newcommand{\TSol}{{\operatorname{EtSol}}} 
\newcommand{\EtSol}[2][\SR]{{{\TSol}}_{#1}(#2)}
\newcommand{\GammaRM}[1]{\GammaR^{#1}}
\newcommand{\GammaM}[1]{\Gamma^{#1}}
\newcommand{\taut}{\mathrm{taut}}
\newcommand{\tautcoij}{\gamma_{ij}^\taut}

\newcommand{\tautcofij}{\gamma_{ij}^{f,\taut}}

\newcommand{\tautcofliftij}{\tilde\gamma_{ij}^{f,\taut}}

\newcommand{\tautMbF}[1]{\SM_\SR^f(#1)}
\newcommand{\tautPbF}[1]{\SP_\SR^f(#1)}
\newcommand{\SM}{{\mathcal{M}}}
\newcommand{\nh}{\mathrm{nHS}}

\newcommand{\aorder}{{\tilde{r}}}

\DeclareMathOperator{\hofib}{hofib}

\newcommand{\corank}{\operatorname{corank}}

\newcommand{\Fib}{\catname{Fib}}
\newcommand{\Eucl}{\catname{Eucl}}

\newcommand{\GammaR}[1][\SR]{{\operatorname{\Gamma}}_{#1}}
\newcommand{\Whitney}{{\operatorname{Wh}}}

\begin{document}

\title{Wrinkling and Haefliger structures}

\date{\today}


\author{Anna Fokma}
\address{Utrecht University, Department of Mathematics, Budapestlaan 6, 3584~CD Utrecht, The~Netherlands}
\email{s.j.fokma@uu.nl}

\author{\'Alvaro del Pino}
\address{Utrecht University, Department of Mathematics, Budapestlaan 6, 3584~CD Utrecht, The~Netherlands}
\email{a.delpinogomez@uu.nl}

\author{Lauran Toussaint}
\address{VU Amsterdam, Department of Mathematics, De Boelelaan 1111, 1081~HV Amsterdam, The~Netherlands}
\email{l.e.toussaint@vu.nl}

\begin{abstract} 
Wrinkling techniques, introduced by Eliashberg and Mishachev, are typically used to prove h-principles of the form: ``formal solutions of a partial differential relation $\SR$ can be deformed to singular/wrinkled solutions''. What a wrinkled solution is depends on the context, but the overall idea is that it should be an object that fails to be a solution only due to the presence of mild/controlled singularities. \\

\noindent Much earlier, Haefliger structures were introduced by Haefliger as singular analogues of foliations. Much like a foliation is locally modeled on a submersion, a Haefliger structure is modeled on an arbitrary map. This implies that Haefliger structures have better formal properties than foliations. For instance, they can be pulled back by arbitrary maps and admit a classifying space.\\

\noindent In \cite{PT1}, the second and third authors generalized the \emph{wrinkled embeddings} of Eliashberg and Mishachev to arbitrary order. This paper can be regarded as a sequel in which we deal instead with generalizations of \emph{wrinkled submersions}. The main messages are that:
\begin{itemize}
\item Haefliger structures provide a nice conceptual framework in which general wrinkling statements can be made.
\item Wrinkling can be interpreted as holonomic approximation into the \'etale space of solutions of the relation $\SR$.
\end{itemize}
These statements imply connectivity statements relating (1) $\SR$ to its \'etale space of solutions and (2) the classifying space for foliations with transverse $\SR$-geometry to its formal counterpart.
\end{abstract}
\maketitle

\section{Introduction} \label{sec:introduction}

In \cite{LM}, Laudenbach and Meigniez present the h-principle for geometric structures as a two-step process. The first step consists of producing a Haefliger microbundle for the given geometry, out of a given formal geometric structure. The second step is the so-called \emph{regularization}: using homotopies/surgeries one makes the base manifold transverse to the foliation on the Haefliger microbundle, yielding a genuine geometric structure. Depending on the geometric problem, difficulties appear in each of the steps (or in both). Classical work of Gromov \cite{Gr86} shows that the second step can always be achieved if the manifold is open, a particular case being precisely the h-principle for foliations of Haefliger~\cite{Haef1,Haef2}.

In \cite{LM}, the authors carry out the first step for symplectic and contact structures. However, their proof uses Moser stability and, as such, cannot be adapted to other geometric structures, since most of them have a much larger moduli space.

The observation that motivates the present article is that wrinkling techniques can be used to implement the first step for any open (or, more generally: microflexible and locally integrable) $\Diff$-invariant partial differential relation $\SR$. That is, every formal solution of $\SR$ can be suitably homotoped to produce a Haefliger microbundle endowed transversely with a solution of $\SR$. This is stated as \cref{thm:wrinklingMicro} below. To get to this statement we first prove \cref{thm:wrinklingEtale}, saying that every formal solution can be approximated by a wrinkled submersion into the \'etale space of solutions of $\SR$. We think of these as singular, and in particular wrinkled, solutions of $\SR$.

The key ingredient behind both theorems is a parametric application of holonomic approximation. This approach is very closely related to the main construction in Eliashberg's and Mishachev's work on \emph{wrinkled submersions} \cite{ElMiWrinI}, which has later been used in other h-principles \cite{BEM}. That Haefliger microbundles and wrinkling turn out to be closely related is not surprising: Much like in the two-step process of Laudenbach and Meigniez, the idea behind wrinkling is to produce singular solutions of differential relations, which may be regularized in certain applications. In fact, one of the first applications of wrinkling \cite{ElMiWrinIII} was reproving Thurston's regularization of Haefliger microbundles in corank at least $2$.

We dedicate the remainder of the introduction to stating our main theorems.

\subsection{Holonomic approximation in \'etale space}

We fix a fiber bundle $p: \Psi \rightarrow M$ over an $m$-dimensional manifold $M$, and write $p_b: J^r\Psi \rightarrow M$ for its associated bundle of $r$-jets. We then consider the \'etale space $\Etale{\Psi}$ of sections of $\Psi$, which consists of germs of sections of $\Psi$; its projection to the base is denoted by $p_b: \Etale{\Psi} \rightarrow M$ as well. We also introduce for each $r \in \N\cup\{\infty\}$, the maps $p_r: \Etale{\Psi} \rightarrow J^r\Psi$ that amount to taking the $r$-jet of the given germ. Endowed with the \'etale topology, $\Etale{\Psi}$ is a smooth $m$-dimensional non-Hausdorff manifold that submerses surjectively onto $M$ via $p_b$.

In \cite{ElMiWrinI}, Eliashberg and Mishachev prove that the space of formal submersions is homotopy equivalent to the space of wrinkled submersions, which are submersions up to a specific type of singularity: the wrinkle. We review the fundamentals of wrinkling in \cref{sec:singularities}. 
Our main result generalizes the key construction in \cite{ElMiWrinI} to show the following. 

\begin{restatable*}{thm}{wrinklingEtale} \label{thm:wrinklingEtale}
Let $\Psi \rightarrow M$ be a fiber bundle. Given any section $F : M \rightarrow J^r\Psi$, there exists a wrinkled submersion $G: M \rightarrow \Etale{\Psi}$ such that:
\begin{itemize}
\item $p_r \circ G$ and $F$ are $C^0$-close, 
\item the wrinkled submersion $p_b \circ G: M \rightarrow M$ is $C^0$-close to the identity, and
\item the regularized differential of $p_b \circ G$ is homotopic to $j^1(\id)$ via formal submersions.
\end{itemize}
\end{restatable*} 
\Cref{thm:wrinklingEtale} is the basis for all the statements that follow and is proven in \cref{sec:wrinklingEtale}. A reader not (yet) familiar with the wrinkling philosophy can think of the third item as a variation of the second item, telling us that $p_b \circ G$ is close to being a section.

We recall that elementary examples show that holonomic approximation is generally impossible for sections over a closed domain. One may interpret \cref{thm:wrinklingEtale} as saying that it is possible to achieve holonomic approximation if multi-valued holonomic sections are allowed: we observe that any map $G$ into $\Etale{\Psi}$ yields a map $p_i \circ G$ that is tangent to the Cartan distribution on $J^i\Psi$, for every $i$. Such maps are non-graphical generalizations of holonomic sections, which is why they were dubbed \emph{holonomic multi-valued sections} in \cite{PT1}. It follows that:
\begin{corollary}
    Let $\Psi \rightarrow M$ be a fiber bundle. Any section $F : M \rightarrow J^r\Psi$ can be $C^0$-approximated by a holonomic multi-valued section $p_r \circ G: M \rightarrow J^r\Psi$.
\end{corollary}

\begin{remark}
    Results similar to \cref{thm:wrinklingEtale} and its upcoming parametric and relative version from \cref{thm:wrinklingEtaleParametric} were obtained simultaneously by Eliashberg, Mishachev and Cieliebak in \cite{CiElMi}, although the frameworks that are used to phrase the results are different.
\end{remark}

\subsubsection{Parametric and relative version} \label{sec:WhitneyVSEtaleTop}
We now state a version of \cref{thm:wrinklingEtale} that is parametric, and relative in both parameter and domain. The parametric nature of the statement requires a topology on $\Etale\Psi$ that is coarser than the \'etale topology: the Whitney topology. We define the \textbf{Whitney topology} on $\Etale{\Psi}$ as the pullback of the topology in $J^\infty\Psi$ along the projection $p_\infty: \Etale\Psi \to J^\infty\Psi$. We observe that the maps $p_r:\Etale\Psi \to J^r\Psi$ are continuous for both topologies. We still make it explicit in the statement below, but from now on, unless indicated otherwise, we always consider individual maps into $\Etale{\Psi}$ that are continuous in the \'etale sense.

\begin{restatable*}{thm}{wrinklingEtaleParametric} \label{thm:wrinklingEtaleParametric}
Let $\Psi \rightarrow M$ be a fiber bundle, let $K$ be compact manifold serving as parameter space, and let $F_k : M \rightarrow J^r\Psi$ be a $K$-family of sections. Suppose that the $F_k$ are holonomic on $\Op(M')$ and whenever $k$ belongs to $\Op(K')$, where $M' \subset M$ and $K' \subset K$ are closed subsets. Then, there exists a $K$-family of wrinkled submersions $G_k: M \rightarrow \Etale{\Psi}$ such that:
\begin{itemize}
    \item for each $k \in K$ the map $G_k$ is continuous for the \'etale topology,
    \item the map $G: M \times K \to \Etale(\Psi)$ defined by $G(\cdot,k)=G_k(\cdot)$ is continuous for the Whitney topology,
    \item the maps $p_r \circ G_k$ and $F_k$ are $C^0$-close, and agree on $M'$ and if $k \in K'$,
    \item the maps $p_b \circ G_k: M \rightarrow M$ form a $K$-family of wrinkled submersions that are $C^0$-close to the identity, and
    \item the $K$-family of regularized differentials of $p_b \circ G_k$ is homotopic to the constant family $j^1(\id)$ via families of formal submersions, relative to $M'$ and $K'$.
\end{itemize}
\end{restatable*}

We will prove a slightly stronger statement in \cref{pro:wrinklingEtaleMfdFib}, phrased instead in terms of a single fibered wrinkled submersion $M \times K \rightarrow \Etale(\Psi \times K)$ that is continuous for the \'etale topology. This reduces \cref{thm:wrinklingEtaleParametric} to the non-parametric case.

\subsection{Wrinkled solutions of differential relations}

Consider a partial differential relation $\SR \subset J^r\Psi$ of order $r$. The \'etale space $\EtSol{M}$ of solutions of $\SR$ is an $m$-dimensional submanifold of $\Etale \Psi$. By construction, $p_r: \Etale\Psi \rightarrow J^r\Psi$ takes $\EtSol{M}$ to $\SR$, and this is a surjection if $\SR$ is open.

Our wrinkled holonomic approximation from \cref{thm:wrinklingEtale} immediately implies that:
\begin{corollary} \label{cor:wrinklingEtale}
    Let $M$ be an $m$-dimensional manifold and let $\SR \subset J^r\Psi$ be an open differential relation. Let $F : M \rightarrow \SR$ be a formal solution. Then, there exists a wrinkled submersion $G: M \rightarrow \EtSol{M}$ satisfying the conclusions of \cref{thm:wrinklingEtale} and additionally satisfying that $p_r \circ G$ is homotopic to $F$ via maps $M \to \SR$.
\end{corollary}

We also state the parametric and relative version:
\begin{corollary} \label{cor:wrinklingEtaleParametric}
    Let $M$ be an $m$-dimensional manifold, let $K$ be a compact manifold serving as parameter space and let $\SR \subset J^r\Psi$ be an open differential relation. Let $F_k : M \rightarrow \SR \subset J^r\Psi$ be a $K$-family of formal solutions of $\SR$. Suppose that the $F_k$ are holonomic over a neighborhood of a closed subset $M' \subset M$ and whenever $k$ belongs to a neighborhood of a closed subset $K' \subset K$.

    Then, there exists a $K$-family $G_k: M \rightarrow \EtSol{M}$ satisfying the conclusions of \cref{thm:wrinklingEtaleParametric} and additionally satisfying that each $p_r \circ G_k$ is homotopic to $F_k$, as maps $M \to \SR$, relative to $M'$ and $K'$.
\end{corollary}
It will be apparent from the proof of \cref{thm:wrinklingEtale} that \cref{cor:wrinklingEtale} also holds when $\SR$ is microflexible and locally integrable. The same applies to the parametric versions. We leave this to the reader.

\subsubsection{Connectivity of \'etale space}
\Cref{cor:wrinklingEtale} states that we can upgrade formal solutions (i.e. sections) of $\SR$ to wrinkled submersions into $\EtSol{M}$. If we give up on sections altogether, we can consider instead arbitrary maps into $\SR$ (with domain not necessarily $M$) and ask whether these lift to \'etale space. We prove the following result in \cref{sec:proofConnectivityEtale}:
\begin{restatable*}{pro}{connectivityEtale} \label{prop:connectivityEtale}
Let $M$ be an $m$-dimensional manifold and let $\SR$ be an open relation. The map $p_r: \EtSol{M} \to \SR$ is $m$-connected.
\end{restatable*}

\subsection{\texorpdfstring{$\SR$}{R}-microbundles} \label{sec:IntroRMb}
Throughout this section we assume that $\SR$ is a $\Diff$-invariant relation of order $r$ and dimension $n$. This means (see \cref{ssec:natural} for details) that over each $n$-dimensional manifold $N$ we have a bundle $\Psi \rightarrow N$, and a subset of $J^r\Psi$, also denoted by $\SR$. The bundle $\Psi$ is endowed with an action of the pseudogroup $\Diff_\loc(N)$ of locally defined diffeomorphisms of $N$, and we require $\SR$ to be invariant under this action. In this case $\Etale\Psi$ also inherits an action of $\Diff_\loc(N)$ which leaves $\EtSol{N}$ invariant.

Given a microbundle $E\to M$ of rank $n$ (see \cref{def:MB}), we consider solutions of $\SR$ on the fibers of $E$. This defines a family of solutions of $\SR$, parametrized by $M$. As discussed in \cref{sec:WhitneyVSEtaleTop}, such a family can be continuous for either the \'etale or the Whitney topology. In the former case, we obtain a canonical foliation on the microbundle. This motivates us to define $\SR$-microbundles, which are thus generalizations of solutions of $\SR$ on $M$:

\begin{definition} \label{def:HaefligerMicro}
    Let $M$ be an $m$-dimensional manifold and let $\SR$ be a $\Diff$-invariant differential relation of dimension $n$. An \textbf{$\SR$-microbundle} on $M$ is a triple $(E,\SF,F)$, where:
    \begin{itemize}
        \item $E \to M$ is a microbundle of rank $n$,
	   \item $\SF$ is a germ of a codimension-$n$ foliation on $E$ along $M$, transverse to the fibers of $E$, and 
	   \item $F$ is a family of germs of solutions on the fibers of $E$ along $M$, that is continuous for the \'etale topology and invariant under the holonomy of $\SF$.
    \end{itemize}
\end{definition}

If we drop $F$, we obtain a foliated microbundle $(E,\SF)$ that has nothing to do with $\SR$ and is classically known as a Haefliger microbundle \cite{Haef1,Haef2}. The geometry is introduced via $F$, which we can think of as a solution of $\SR$ over the leaf space of $\SF$. This only makes sense due to $\Diff$-invariance as we will explain in \cref{sec:prelimFolTransvStr}, when we discuss transverse $\SR$-structures.

\begin{remark}
Two names that we have actively tried to avoid are Haefliger structure and $\Gamma$-structure. The reason is that these terms are used to mean slightly different things across the literature. In particular, what Haefliger originally \cite{Haef2} called a $\Gamma$-structure, we just refer to as an equivalence class of $\Gamma^n$-cocycles. In turn, this is in one-to-one correspondence with an isomorphism class of principal $\Gamma^n$-bundles (\cref{ssec:PGBundles}). 
\end{remark}

We are particularly interested in the case where $E=TM$ and hence $m=n$, in which case we speak of a \textbf{tangential} microbundle. The simplest examples arise as follows: Fixing a metric on $M$, and thus an exponential map, allows us to assign an $\SR$-microbundle 
\[ \exp(f) := (TM,\ker(d\exp),f \circ \exp) \]
to each solution $f$ of $\SR$ on $M$. The leaf space of (the germ of) $\ker(d\exp)$ is simply $M$. It follows that leaf spaces of other tangential microbundles may be regarded as singular replacements of $M$.

For the formal analogue of $\SR$-microbundles (see ~\cref{ssec:formalAnalogue} for more details), we ask the family of solutions to instead be a family of \emph{formal} solutions. Since we now consider formal solutions, we ask the family to be smooth for the Whitney topology. We note that hence we also do not obtain a canonical foliation. 
\begin{definition}
Let $M$ be an $m$-dimensional manifold and let $\SR$ be a $\Diff$-invariant differential relation of dimension $n$. A \textbf{formal $\SR$-microbundle} on $M$ is a pair $(E,F)$, where:
\begin{itemize}
	\item $E \to M$ is a microbundle of rank $n$, and
	\item $F$ is a smooth family of formal solutions on the fibers of $E$ along $M$.
\end{itemize}
\end{definition}

\subsubsection{The tangential case}
Every formal solution $F$ of $\SR$ over $M$ can be lifted to a tangential formal $\SR$-microbundle $\exp(F)$. The following is proven in \cref{sec:wrinklingHaefliger}. Here homotopic should be interpreted (as argued in \cref{sec:prelim2classspace}) as a \emph{concordance} of formal $\SR$-microbundles over $M$, which is a formal $\SR$-microbundle over $M\times I$.
\begin{restatable*}{thm}{wrinklingMicro} \label{thm:wrinklingMicro}
    Let $M$ be an $m$-dimensional manifold and let $\SR$ be an open and $\Diff$-invariant relation of dimension $n=m$. Let $F : M \rightarrow \SR \subset J^r\Psi$ be a formal solution of $\SR$. Then, the formal $\SR$-microbundle is homotopic via formal $\SR$-microbundles to an $\SR$-microbundle $(TM,\SF,G)$ where the foliation has only wrinkle singularities with respect to $M$.
\end{restatable*}
Let us explain briefly the relation between \cref{thm:wrinklingEtale,thm:wrinklingMicro}. Roughly speaking, the second one follows from the first by pulling back the tautological $\SR$-microbundle in the \'etale space $\EtSol{M}$ using $G$.

For \'etale spaces we emphasized the distinction between the \'etale and Whitney topologies, which was important to discuss parametric statements. For $\SR$-microbundles we similarly speak of \textit{concordances} and \textit{Whitney continuous families}. The former is an $\SR$-microbundle over $M \times I$. The latter consists of a microbundle $E \to M \times I$, with a homotopy of foliations $\SF_t$ on $M \times \{t\}$ each endowed with a transverse $\SR$-structure $F_t$ inducing an $M \times I$-family of solutions of $\SR$ that varies continuously in $t$ for the Whitney topology.

It is the notion of a Whitney continuous family that we need for the parametric analogue of \cref{thm:wrinklingMicro}. We point out that in the statement below, if the $F_k$ are holonomic over $\Op(M') $ and whenever $k \in Op(K') $, this implies that the microbundles $\exp(F_k)$ are a Whitney continuous family of $\SR$-microbundles when restricted to $\Op(M')$ or whenever $k \in Op(K')$.

\begin{restatable*}{thm}{wrinklingMicroParametric} \label{thm:wrinklingMicroParametric}
Let $M$ be an $m$-dimensional manifold and let $\SR$ be an open and $\Diff$-invariant relation of dimension $n=m$. Let $K$ be a compact manifold serving as parameter space. Let $F_k : M \rightarrow \SR \subset J^r\Psi$ be a $K$-family of formal solutions. Suppose that they are holonomic over a neighborhood of a closed subset $M' \subset M$ and whenever $k$ belongs to a neighborhood of a closed subset $K' \subset K$.

Then, there exists a Whitney continuous family of tangential $\SR$-microbundles $(TM,\SF_k,G_k)$ that is homotopic to $\exp(F_k)$ via formal $\SR$-microbundles, relative to $M'$ and $K'$. 
\end{restatable*}

\subsubsection{General \texorpdfstring{$\SR$}{R}-microbundles}
In the not necessarily tangential case, we generalize \cref{thm:wrinklingMicro} by starting with a formal $\SR$-microbundle, instead of a formal solution. In this setting it also makes sense to allow for the relation $\SR\subset J^r\Psi$ to be of a different dimension $n$ than the dimension $m$ of the manifold $M$. We obtain the following.

\begin{restatable*}{pro}{HaefligerNonTangential} \label{prop:HaefligerNonTangential} 
    Let $M$ be an $m$-dimensional manifold and let $\SR$ be an open and $\Diff$-invariant relation of dimension $n \geq m$. Let $(E,F)$ be a formal $\SR$-microbundle on $M$, then $(E,F)$ is homotopic to an $\SR$-microbundle via formal $\SR$-microbundles.
\end{restatable*}
We also state the parametric and relative version. Results similar to Propositions \ref{prop:HaefligerNonTangential} and \ref{prop:HaefligerParametricNonTangential} are obtained in \cite{LM} for contact and symplectic structures.
\begin{restatable*}{pro}{HaefligerParametricNonTangential}\label{prop:HaefligerParametricNonTangential}
    Let $M$ be a manifold of dimension $m$ and let $\SR$ be an open and $\Diff$-invariant relation of dimension $n \geq m$. Let $K$ be compact manifold serving as parameter space. Let $(E,\SF_k,F_k)$ be a $K$-family of formal $\SR$-microbundles over $M$, varying continuously in $k$ for the Whitney topology, such that $F$ is holonomic over a neighborhood of a closed subset $M' \subset M$ and whenever $k$ belongs to a neighborhood of a closed subset $K' \subset K$.

    Then, there exists a Whitney continuous family of genuine $\SR$-microbundles $(E,\SF_k',F_k')$ that is homotopic to $(E,\SF_k,F_k)$ via formal $\SR$-microbundles, relative to given lift over $M'$ and $K'$.
\end{restatable*}

\subsection{Classifying spaces of geometric structures} \label{ssec:introGamma}
In \cref{sec:Rgroupoids} we associate to every $\Diff$-invariant relation $\SR$ a (possibly non-Hausdorff and non second-countable) \'etale Lie groupoid $\GammaR$, whose base encodes all germs of solutions of $\SR$ and whose arrow space encodes all symmetries of $\SR$. This was first done by Haefliger for various well-studied geometries (symplectic, contact, complex) \cite{Haef2}, but the authors did not know of such a procedure in the literature of the required generality. There is a one-to-one correspondence (up to concordance) between the $\SR$-microbundles from \cref{def:HaefligerMicro} and principal $\GammaR$-bundles. We will study the latter via the former and vice versa, which is a standard idea in the foliation literature, also going back to Haefliger \cite{Haef2}.

The groupoid $\GammaR$ has a classifying space $B\GammaR$ such that homotopy classes of maps $M \to B\GammaR$ correspond to concordance classes of principal $\GammaR$-bundles. One can construct formal counterparts $\GammaR^f$ and $B\GammaR^f$ as well by considering jets instead of germs of solutions, and they come with natural scanning maps $\GammaR \rightarrow \GammaR^f$ and $B\GammaR \rightarrow B\GammaR^f$. Hence we think of the domain in the next result as the space of \textbf{principal $\GammaR$-bundles} on $M$, and of the target as its formal analogue. 

\begin{restatable*}{thm}{cocycleConnectivityEtale} \label{cor:cocycleConnectivityEtale}
    Let $M$ be an $m$-dimensional manifold and let $\SR$ be an open and $\Diff$-invariant relation of dimension $n \geq m$. The scanning map 
    \[ \Maps(M,B\GammaR) \rightarrow \Maps(M,B\GammaR^f) \]
    is $(n-m)$-connected.
\end{restatable*}

By taking $M = \{*\}$ we obtain the following result. It is a more functorial incarnation of \cref{thm:wrinklingMicro}.
\begin{restatable}{cor}{classifyingSpaceConnectivity} \label{thm:classifyingSpaceConnectivity}
Let $\SR$ be an open and $\Diff$-invariant relation of dimension $n$. Then, the map $B\GammaR \rightarrow B\GammaR^f$ is $n$-connected.
\end{restatable}

\begin{remark}
In the contact and symplectic settings better connectivity statements are known: McDuff proved in \cite{McDuff87} that one obtains $(n+1)$-connectivity in \cref{thm:classifyingSpaceConnectivity} when $\SR$ is the relation defining either contact or symplectic structures. Recently, this was improved to $n+2$ by Nariman~\cite{Nar} in the contact case. 
This depends on both geometries exhibiting Moser stability and having large automorphism groups. 

Our expectation, which we aim to tackle in future work, is that the connectivity of \cref{thm:classifyingSpaceConnectivity} is sharp for most relations $\SR$, using the fact that solutions have local invariants. In fact, for many $\SR$, we expect the $n$th homotopy group of $B\GammaR$ to be very large (uncountable).
\end{remark}

\subsection{Structure of the paper} \label{ssec:contents}

We discuss partial differential relations, Diff-invariance, and \'etale spaces of solutions in \cref{sec:preliminaries1}. In \cref{sec:singularities} we discuss the required theory on singularities of maps. Our theorems about wrinkling into \'etale space are proven in \cref{sec:wrinklingEtale}. \Cref{prop:connectivityEtale}, concerning the connectivity of \'etale space, is proven in \cref{sec:applWrinkling}, where also applications of \cref{thm:wrinklingEtale} to folded symplectic structures and horizontal homotopy groups are discussed. In \cref{sec:prelimGroupoids} we discuss Haefliger's viewpoint on groupoids, principal bundles and microbundles. \Cref{sec:TechResultsPb} contains some technical results related to the smoothing of principal bundles, and the lifting of principal bundles to microbundles. In \cref{sec:Rgroupoids} we develop the language needed to deal with groupoids encoding arbitrary $\Diff$-invariant relations, and we discuss the formal analogue in \cref{ssec:formalAnalogue}. We combine this with wrinkling in \cref{sec:wrinklingHaefliger} to prove our results about Haefliger microbundles and connectivity of the associated classifying space.

\subsection{Acknowledgments}

The authors want to thank Luca Accornero for providing insightful comments on a preliminary version of this article.

The third author is funded by the Dutch Research Council (NWO) via the project ``proper Fredholm homotopy theory'' (project number OCENW.M20.195) of the research program Open Competition ENW M20-3. In the very early stages of this project the second author was funded by the NWO grant 016.Veni.192.013 ``Topology of bracket-generating distributions''. In the very late stages the second author was funded by the NWO grant VI.Vidi.223.118 ``Flexibility and rigidity of tangent distributions''.
\section{Preliminaries: \texorpdfstring{$\Diff$}{Diff}-invariant relations} \label{sec:preliminaries1}

In this section we recall the definition and basic properties of partial differential relations encoded as subsets of jet spaces. We refer the reader to \cite{Gr86,ElMi} as the standard introductions to this viewpoint. We put particular emphasis in discussing:
\begin{itemize}
    \item \'Etale spaces of solutions and the different topologies that they admit (\cref{ssec:etaleSpaces}),
    \item Natural bundles and the nature of $\Diff$-invariance (\cref{ssec:natural}).
\end{itemize}

\subsection{Solutions and formal solutions} \label{ssec:solutions}

Let $M$ be a manifold and $\SR \subset J^r\Psi \rightarrow M$ a partial differential relation. There exists a canonical smooth structure on $J^r \Psi$ induced by the smooth structure on $\Psi$. We say that $\SR$ is an \textbf{open} differential relation, if it is open as a subset of $J^r\Psi$. We write $\Gamma(M,J^r\Psi)$, or sometimes $\Gamma(J^r\Psi)$, for the space of smooth sections of $J^r\Psi$. There are two interesting topologies we can endow $\Gamma(J^r \Psi)$ with:
\begin{enumerate}
    \item the weak topology, for which a basis consists of the sets $\{s \in \Gamma( J^r \Psi) \mid j^r s(K) \subset U\}$ indexed by compact subsets $K \subset M$ and opens $U \subset J^r\Psi$, and
    \item the strong topology, for which a basis consists of the sets $\{s \in \Gamma( J^r \Psi) \mid  j^r s(M) \subset U \} $ indexed by opens $U \subset J^r\Psi$.
\end{enumerate}
If $M$ is closed, these topologies agree, but if $M$ is open they do not.

We define the (weak/strong) \textbf{Whitney $C^r$-topology} as the smallest topology on $\Gamma(\Psi)$ such that the map $j^r : \Gamma(\Psi) \to \Gamma( J^r \Psi)$ is continuous with respect to the (weak/strong) topology on $\Gamma(J^r \Psi)$. The (weak/strong) \textbf{Whitney $C^\infty$-topology} is the union of the (weak/strong) Whitney $C^r$-topologies on $\Gamma(\Psi)$ over all  $r\in \N$. 

Sections of $J^r\Psi$ with image in $\SR$ are called \textbf{formal solutions} and we denote the space of such sections by $\SolFFun$. A \textbf{(genuine) solution} of $\SR$ is a section $f: M \rightarrow \Psi$ whose $r$-jet extension $j^r f: M \rightarrow J^r\Psi$ takes values in $\SR$. We write $\SolFun \subset \Gamma(\Psi)$ for the subspace of solutions of $\SR$. If $\SR$ is open, then $\SolFun$ is open as a subset of $\Gamma(\Psi)$ in the strong topology (but not necessarily in the weak). We endow both $\SolFun$ and $\SolFFun$ with the weak Whitney topology. Further explanations of these choices can be found in \cref{rem:topInf,rem:topWeakStrong}.

With this terminology we define the \textbf{scanning map}
\[ \scanR: \SolFun \rightarrow \SolFFun \]
and the purpose of the h-principle is to understand its connectivity. We say that the \textbf{h-principle} holds for a relation $\SR$ if the map $\scanR$ is a weak homotopy equivalence. 

\begin{remark}\label{rem:topWeakStrong}
    The strong topology forces continuous $K$-families $M\to N$ to be constant outside of a compact in $M \times K$, where $M$ and $N$ are manifolds and $K$ is a compact manifold serving as parameter space. However, h-principles on open manifolds generally depend crucially on the non-compactness of the domain $M$ and therefore require that families are not constant outside a compact. Hence, when we discuss families of solutions, we let them be continuous for the weak topology.
\end{remark}
\begin{remark} \label{rem:topInf}
It is common in the literature to use the Whitney $C^r$-topology in $\SolFun$ and the $C^0$-topology in $\SolFFun$ instead of the Whitney $C^\infty$-topology in both. The benefit of using the Whitney $C^\infty$-topology is that all continuous families are in fact smooth in the parameter so that we can take arbitrarily many derivatives. However, due to smoothing, the resulting spaces are homotopy equivalent. 
\end{remark}

\subsection{\'Etale spaces of solutions} \label{ssec:etaleSpaces}
We observe that we can interpret $\Gamma(\cdot,\Psi)$ as a functor from opens in $M$ to $\Top$ which makes it into a $\Top$-sheaf. By composing the sheaf with the forgetful functor $\Top \to \Set$ we obtain a $\Set$-sheaf. Its associated \textbf{\'etale space} is denoted by $ \Etale\Psi$. 

Explicitly, as a set, $\Etale\Psi$ consists of germs of local sections of $\Psi$ on $M$, that is: 
\[ \Etale\Psi = \{[s]_x \mid s \in \Gamma(U,\Psi), x\in U \subset M \text{ open}\}, \]
where $[s]_x$ denotes the germ of $s$ at $x$. 
A basis for its \textbf{\'etale topology} consists of all subsets $ \{[s]_x \mid x\in U\} $ indexed by all opens $U \subset M$ and sections $s \in \Gamma(U,\Psi)$. This makes the projection $p_b : \Etale \Psi \to M$ into an \textbf{\'etale map}, that is, a local homeomorphism. We moreover endow $\Etale\Psi$ with the canonical smooth structure making $p_b$ into a local diffeomorphism. 

Since we are interested in solutions of a relation $\SR$, we define the subset $\EtSol{M}$ of $\Etale\Psi$ as the \textbf{\'etale space  of (germs of) solutions} of $\SR$.

\subsubsection{The Whitney topologies} \label{sec:EtSolWhitneyTop}
It is also possible to endow $\Etale\Psi$, and hence $\EtSol{M}$ in particular, with the Whitney $C^r$-topologies. They are each defined as the smallest topology such that the projection $p_r : \Etale \Psi \to J^r \Psi$ is continuous. Again, the Whitney $C^\infty$-topology is the union of all these topologies, and we will henceforth refer to it as the \textbf{Whitney topology}. Our convention is to endow  $\Etale\Psi$ and hence $\EtSol{M}$ with the \'etale topology. If we instead endow a subset $A$ of $\Etale\Psi$ with the Whitney topology, we indicate this by writing $A^\Whitney$. 

To distinguish the two topologies it is helpful to have the following observation in mind: a map into $\EtSol{M}$ that is continuous for the \'etale topology is (locally) the lift of a solution, whereas a map that is Whitney continuous does not (necessarily) come from such a lift. Hence maps that are Whitney continuous can be thought of as the germ analogue of a formal solution.

The Whitney topology is coarser than the \' etale one. This entails in particular that every continuous map into $\Etale\Psi$ endowed with the \'etale topology will also be continuous with respect to the Whitney topology. The converse is not true, as illustrated by the following example:
\begin{example} \label{ex:WhitneyEtaleTop}
    Let $\chi: \R\to\R$ be a function that is flat in $0$ (i.e. $j^\infty_0\chi = 0$) and is non-vanishing outside of $0$. Define a path $\gamma: I \to \Etale\R$ by 
    \[
        \gamma(t) = \begin{cases}
            [0]_0 &\text{ if } t\in \Q \\
            [\chi]_0 &\text{ otherwise.}
        \end{cases}
    \]
    Then $\gamma$ is continuous if we endow $\Etale\R$ with the Whitney topology, but not if we use the \'etale topology.
\end{example}

    \subsubsection{The tautological solution} \label{sec:tautsolPullback}
A useful feature of $\EtSol{M}$ is that it carries a canonical solution of $p_b^*\SR$. We observe that we can use $p_b: \EtSol{M} \rightarrow M$ to lift $\Psi \rightarrow M$ to a bundle $p_b^*\Psi$ over $\EtSol{M}$. We then take jets and, since $p_b$ is \'etale, there is a canonical isomorphism between $J^r(p_b^*\Psi)$ and $p_b^*(J^r\Psi)$. From this it follows that there is a well-defined lift $p_b^*\SR \subset J^r(p_b^*\Psi)$ of $\SR$. In this setup we now define:
\begin{definition} \label{def:tautologicalSolutionPullback}
The \textbf{tautological solution} $\tau : \EtSol{M} \to p_b^*\SR$ is defined as $\tau([s]_x) = p_b^*(j^r_x s) $ where $s \in \SolFun(U)$ is a local solution and $x\in U \subset M$. 
\end{definition}

\subsection{Natural bundles and \texorpdfstring{$\Diff$}{Diff}-invariance} \label{ssec:natural}
In \cref{ssec:solutions} we defined a partial differential relation as a subset of $J^r\Psi \to M$. However, we can also intrinsically formulate them, such that they do not depend on the particular manifold over which they live. To this end, we use the notion of a natural fiber bundle~ \cite{nijenhuis1972natural,PalTe}:
\begin{definition} \label{def:natural}
A \textbf{natural fiber bundle} of dimension $n$ is a functor $\Psi$ from the category $\Man_n$ of $n$-manifolds (with embeddings as morphisms) to the category of fiber bundles $\Fib_n$ over $n$-manifolds (with fibered maps as morphisms), such that:
\begin{itemize}
    \item $\Psi(M)$ is a fiber bundle over $M$ for every manifold $M$, and
    \item $\Psi(f) : \Psi(M) \rightarrow \Psi(N)$ covers $f$ for every embedding $f:M \rightarrow N$ between manifolds.
\end{itemize}
\end{definition}
For notational convenience we will often just write $\Psi \rightarrow M$ to mean $\Psi(M)$. 

An example of a natural bundle is the tangent bundle, from which other examples can be derived such as the frame bundle, the cotangent bundle, and their wedge and symmetric products. Another important operation that preserves the naturality of a fiber bundle is taking $r$-jets. Any morphism $\Psi(f): \Psi(M) \to \Psi(N)$ lifts to a morphism $j^r\Psi(f): J^r\Psi(M) \to J^r\Psi(N)$ by considering its action on $r$-jets.

We write $\Gamma(-,\Psi):\Man_n \to \Top$ for the sheaf associating to every $n$-manifold $M$ the space $\Gamma(M,\Psi)$ of sections of $\Psi(M)$ endowed with the Whitney topology.

\subsubsection{Actions on natural bundles}
We recall that a pseudogroup is defined as follows.
\begin{definition} \label{def:pseudogroup}
    Let $M$ be a smooth manifold and let $\Diff_\loc(M)$ be the set of all locally defined diffeomorphisms of $M$. A \textbf{pseudogroup} $G$ over $M$ is a subset $G \subset \Diff_\loc(M)$ such that 
    \begin{itemize}
        \item $\id|_U \in G$ for all open subsets $U$ of $M$,
        \item if $f \in G$ then $f^{-1} \in G$,
        \item if $f,g \in G$ then $f \circ g|_{g^{-1} (U)} \in G$ where $U$ is the domain of definition of $f$, and
        \item if $U$ is an open in $M$ with cover $\{U_i\}_{i \in I}$ by opens in $M$, and if $f\in \Diff_\loc(M)$ such that $f|_{U_i} \in G$ for all $i \in I$, then $f\in G$.
    \end{itemize}
\end{definition}
A trivial example of a pseudogroup is the set $\Diff_\loc(M)$ itself. It hence follows from the definition of a natural bundle that the pseudogroup of local diffeomorphisms $\Diff_\loc(M)$ acts on $\Psi(M)$.

In the same spirit, Thurston and Epstein \cite{EpsteinThurston79} proved that for any natural bundle $\Psi$ there exists a minimal number $\aorder \in \N$ such that $\Psi(f)$ depends only on $j^\aorder f$. Hence when we take $\aorder$-jets of the elements in $\Diff_\loc(M)$, we obtain the groupoid $J^\aorder \Gamma^M$ which acts on $\Psi(M)$.\footnote{We will discuss the groupoids $\Gamma^M$ and $J^\aorder \Gamma^M$ in greater detail in \cref{sec:Gamma}, as they are not needed prior.} The number $\aorder$ is often referred to as the order of the bundle $\Psi$, but we will refer to it as the \textbf{action order} to distinguish it from the order of a differential relation (see \cref{ssec:diffInvariant}).

\subsubsection{\texorpdfstring{$\Diff$}{Diff}-invariant relations} \label{ssec:diffInvariant}

The notion of a natural fiber bundle allows us to abstract the relation from the particular manifold on which it lives.
\begin{definition} \label{def:diffInvariantRelations}
A \textbf{$\Diff$-invariant partial differential relation} of order $r$ and dimension $n$ is a triple $(\SR,\Psi,i)$ in which:
	\begin{itemize}
		\item $\SR$ and $\Psi$ are natural fiber bundles of dimension $n$,
		\item $i: \SR \to J^r\Psi$ is a natural transformation,
	\end{itemize}
	such that $i: \SR(M) \to J^r\Psi(M)$ is an embedding for all $n$-manifolds $M$.
\end{definition}
We then obtain the sheaf $\Gamma(-,\Psi)$ that sends every $n$-manifold $M$ to the space of sections $\Gamma(M,\Psi)$. Two related sheaves are $\SolFun$ and $\SolFFun$, which respectively send $M$ to its space $\SolFun(M)$ of solutions and to the space $\Gamma(-,\SR)$ of sections of $\SR$. Moreover, we have a natural transformation:
\[ \scanR: \SolFun \rightarrow \SolFFun. \]

    \subsubsection{Revisiting the tautological solution} \label{sec:tautsol}
Recall the \' etale space $\EtSol{M}$ of germs of solutions of $\SR$ over $M$. In this section we claim that, due to $\Diff$-invariance, we can speak of $\SR$ and $\Psi$ being bundles over $\EtSol{M}$. This implies that we can write the tautological solution of \cref{def:tautologicalSolutionPullback} as a section of $\Psi \to \EtSol{M}$, without having to refer to pullbacks from $M$. We hence obtain:
\begin{definition} \label{def:tautologicalSolution}
    The \textbf{tautological solution} $\tau : \EtSol{M} \to \SR$ is defined as $\tau([s]_x) = s(x)$ where $s \in \SolFun(U)$ is a local solution and $x\in U \subset M$. 
\end{definition}

In the remainder of this section we discuss the claim that we can speak of $\Psi$ as a bundle over $\EtSol{M}$. The reasoning for $\SR$ follows in an analogous manner. Note that a priori we cannot apply $\Psi$ to $\EtSol{M}$, since natural bundles are defined on the category $\Man_n$ of Hausdorff and second-countable manifolds which does not include $\EtSol{M}$. We now claim that there exists a natural extension of a natural bundle to the category $\Man^\nh_n$ of $n$-dimensional manifolds that are not necessarily Hausdorff or second-countable, with embeddings as morphisms.

We note that $\Psi$ is completely determined by its restriction $\Psi_\Eucl$ to the category $\Eucl_n$ consisting of opens in $\R^n$ and embeddings. That is, if we fix an $n$-manifold $M$ with atlas $\{ \phi_i: U_i \to \R^n \}_{i \in I}$ and transition functions $\{\phi_{ij} : \phi_j(U_i \cap U_j ) \to \phi_i(U_i \cap U_j ) \}_{i,j \in I}$, we see by naturality of $\Psi$ that we can define
\[ \Psi(M) = \bigsqcup_{i \in I} \Psi_\Eucl(\phi(U_i)) / \sim,
\]
where $(i,x) \sim (j,y)$ if $(\Psi_\Eucl(\phi_{ij}))(y)=x$. In more categorical terms, this corresponds to the claim that $\Psi$ can be obtained from $\Psi_\Eucl$ by left Kan extension. We refer to  \cite{mac2013categories} for more details on left Kan extension. 

Let us introduce some notation. We define the category $\Fib_n^\nh$ of fiber bundles over objects of $\Man^\nh_n$ whose fibers are Hausdorff and second-countable manifolds, with fibered maps as morphisms. We also define the inclusion functor $J: \Fib_n \to \Fib_n^\nh$ and note that $J$ preserves colimits.

Every manifold, possibly non-Hausdorff or not second-countable,  can be defined as a colimit generated by $\Eucl_n$ using its maximal atlas. Hence left Kan extension allows us to extend the functor $J \circ \Psi_\Eucl$ to the functor $\tilde \Psi:\Man^\nh_n \to \Fib_n^\nh$. Since left Kan extension is unique up to natural transformation and $J$ preserves colimits, the functors $J \circ \Psi$ and $\tilde \Psi$ agree up to natural transformation, when restricted to $\Man_n$.

In particular we can apply $\tilde \Psi$ to $\EtSol{M}$ to obtain a fiber bundle over it. Therefore, when we speak of $\Psi(\EtSol{M})$ as a bundle over $\EtSol{M}$, we actually refer to $\tilde\Psi(\EtSol{M})$. However, by a slight abuse of notation, we will not make this distinction. We note that the obtained bundle is isomorphic to the pullback bundle $p_b^*\Psi \to \EtSol{M}$ as constructed in \cref{sec:tautsolPullback}, since $\Psi$ is a natural bundle and $p_b$ is a local diffeomorphism.

\section{Preliminaries: singularity theory} \label{sec:singularities}

Our main theorem (\cref{thm:wrinklingEtale}) tells us that any section $F:M\to J^r\Psi$ can be holonomically approximated by allowing wrinkles. To do so, we will construct a map $G : M \to \Etale(\Psi)$ and project down to $J^r\Psi$. In general we cannot expect the map $G$ to be without singularities, as this would imply that any formal solution can be approximated by a solution. However, it turns out that we can make sure that the map $G$ is a submersion with only (mild) singularities of a specific type, known as wrinkles. 

In these preliminaries we first, very briefly, introduce the basics of singularity theory in \cref{sec:basicsSing}. For more details on this topic we refer to \cite{arnold1985singularities,Th55,Bo67}. A more modern reference, in the context of h-principles is \cite[Part 3]{CiElMi}. Next, we discuss some specific singularities of maps $\R^m \to \R^n$ in \cref{sec:singMapsEucl}, working our way up towards the standard wrinkle. We generalize these singularities to maps between manifolds in \cref{sec:singMaps}. In \cref{sec:singFib} we recall the notion of fibered maps, and observe that the previous singularities are examples. We end in \cref{sec:wrinklSubm} by discussing wrinkled submersions. We refer to the original wrinkling papers \cite{ElMiWrinI,ElMiWrinII,ElMiWrinIII,ElMiWrinEmb} for more details.

\subsection{Singularities}\label{sec:basicsSing}

By a singularity of a map we understand the following.
\begin{definition}
    Let $f: M \to N$ be a smooth map between manifolds of dimension $m$ and $n$, respectively. We say that $f$ is \textbf{singular} at a point $p\in M$ if 
    \[
        \dim(d_p f) < \min{(m,n)}.
    \]
\end{definition}

To understand the set of singular points, we introduce the loci of singularities of a given corank. These loci of singularities tell us about where a singularity occurs, but also how severe a singularity is.
\begin{definition} \label{def:locusSing}
    Let $f: M \to N$ be a smooth map between manifolds of respectively dimension $m$ and $n$. The \textbf{locus of singularities} of corank $k$ of $f$ is the set 
    \[
        \Sigma^k(f) =\left\{x\in M \mid \min(m,n) - \dim(df (T_x M))  = k \right\}.
    \]
    In the case that $\Sigma^k(f)$ is a submanifold of $M$, we recursively define $\Sigma^{k,l}(f) = \Sigma^l(f|_{\Sigma^k(f)})$.
\end{definition}
We will not need loci of higher order singularities, but their definitions are straightforward extensions. 

One aspect of singularity theory is the study of what singularities are exhibited by a generic map. Hence we recall the notion of genericity: A subset of a topological space is \textbf{residual} if it is the countable intersection of dense open subsets. Then, a property $\SP$ on the space $C^\infty(M,N)$ of smooth maps $M \to N$ is \textbf{generic} if the set of maps in $C^\infty(M,N)$ satisfying this property is residual. We say that a generic map satisfies a property $\SP$ if $\SP$ is a generic property. An an example, we mention that $\Sigma^k(f)$ from \cref{def:locusSing} is generically a submanifold of $M$.

\subsection{Standard singularities} \label{sec:singMapsEucl}
We now introduce specific types of singularities for maps $\R^m \to \R^n$. 
Throughout this section we assume $m \geq n \geq 1$.

\subsubsection{Fold}

A generic map $M \to \R$ only has Morse singularities. However, when we consider maps $\R^m \to \R^n$ where $\dim(M) \geq \dim(N)$, more generic singularities arise. The easiest of these is the fold.

\begin{definition}
    The \textbf{standard fold} $F: \R^m \to \R^n$ of index $k$ is the map given by
    \begin{align*}
        F = F(m,n,s) : \R^{n-1} \times \R^{m-n+1} &\to \R^{n-1} \times \R, \\
        (t,x) &\mapsto \left(t, - \sum_{i=1}^k x_i^2 + \sum_{j=k+1} ^{m-n+1} x_j^2 \right).
    \end{align*}
\end{definition}
Its locus of singularities is given by 
\[\Sigma^1(F) = \Sigma^{1,0}(F)= \{x=0\}. \]

\subsubsection{Double fold} \label{sec:DoubleFold}
A double fold consists of two folds, which are moreover in a formally canceling position, as we explain after the definition.

\begin{definition}
    The \textbf{double fold} $D: \R^m \to \R^n$ of index $k$ is the map given by 
    \begin{align*}
        D = D(m,n,s) : \R^{n-1} \times \R^{m-n} \times \R^1 &\to \R^{n-1} \times \R, \\
        (t,x,y) &\mapsto \left(t, y^3-3y - \sum_{i=1}^k x_i^2 + \sum_{j=k+1} ^{m-n} x_j^2 \right).
    \end{align*}
\end{definition}

Its locus of singularities is given by 
\[\Sigma^1(D) = \Sigma^{1,0}(F)= \{y = \pm1, x=0\}. \]

When we compute the Jacobian of the double fold, we obtain the matrix
\begin{equation} \label{eq:DFoldJacob}
    \left(\begin{array}{@{}c|c@{}}
I_{n-1} & 
  0 \\
\hline
  \begin{matrix}
  6t_1 y & \dots & 6 t_{n-1} y 
  \end{matrix}  &
  \begin{matrix}
  -2x_1 & \dots & 
            2x_{m-n} & 3(y^2 -1)
  \end{matrix}
\end{array}\right).
\end{equation}

It is the bottom right submatrix of the Jacobian that causes the differential of $D$ to not be of full rank. However, we can modify the term in the bottom right corner over the set $\Op(\{y\in[-1,1],x=0\}) $ such that it becomes positive everywhere. This defines (up to homotopy) the \textbf{regularized differential} $\RD_F(D): T\R^m \to T\R^n$ of $D$. Hence we say that the double fold formally cancels.

\subsubsection{Cusp} 
When we consider maps between Euclidean spaces of higher dimension, we obtain more generic singularities than just folds. In particular we obtain a singularity known as the cusp, which is already generic for maps $\R^2\to\R^2$. We have illustrated the cusp in \cref{fig:cusp}.

\begin{definition}\label{def:strCusp}
    The \textbf{standard cusp} $C: \R^m \to \R^n$ of index $k$, where $n>1$, is the map given by 
    \begin{align*}
        C = C(m,n,s) : \R^{n-1} \times \R^{m-n} \times \R &\to \R^{n-1} \times \R, \\
        (t,x,y) &\mapsto \left( t,y^3 - 3 t_1 y - \sum_{i=1}^k x_i^2 + \sum_{j=k+1} ^{m-n} x_j^2 \right).
    \end{align*}
\end{definition}
For the locus of singularities we obtain 
\[ \Sigma^1(C) = \{y^2=t_1, x=0\} \]
which decomposes as $\Sigma^1(C) = \Sigma^{1,0}(C) \cup \Sigma^{1,1}(C)$ where
\[
    \Sigma^{1,0} = \{y^2=t_1>0, x=0\} \qquad \text{and} \qquad \Sigma^{1,1} = \{t_1=x=y=0\}. 
\] 

\begin{figure}[h]
    \centering
    \begin{subfigure}[t]{0.4\linewidth}
        \centering
        \includegraphics[width=0.7\linewidth,page=2,clip=true,trim = 14cm 0 2cm 0]{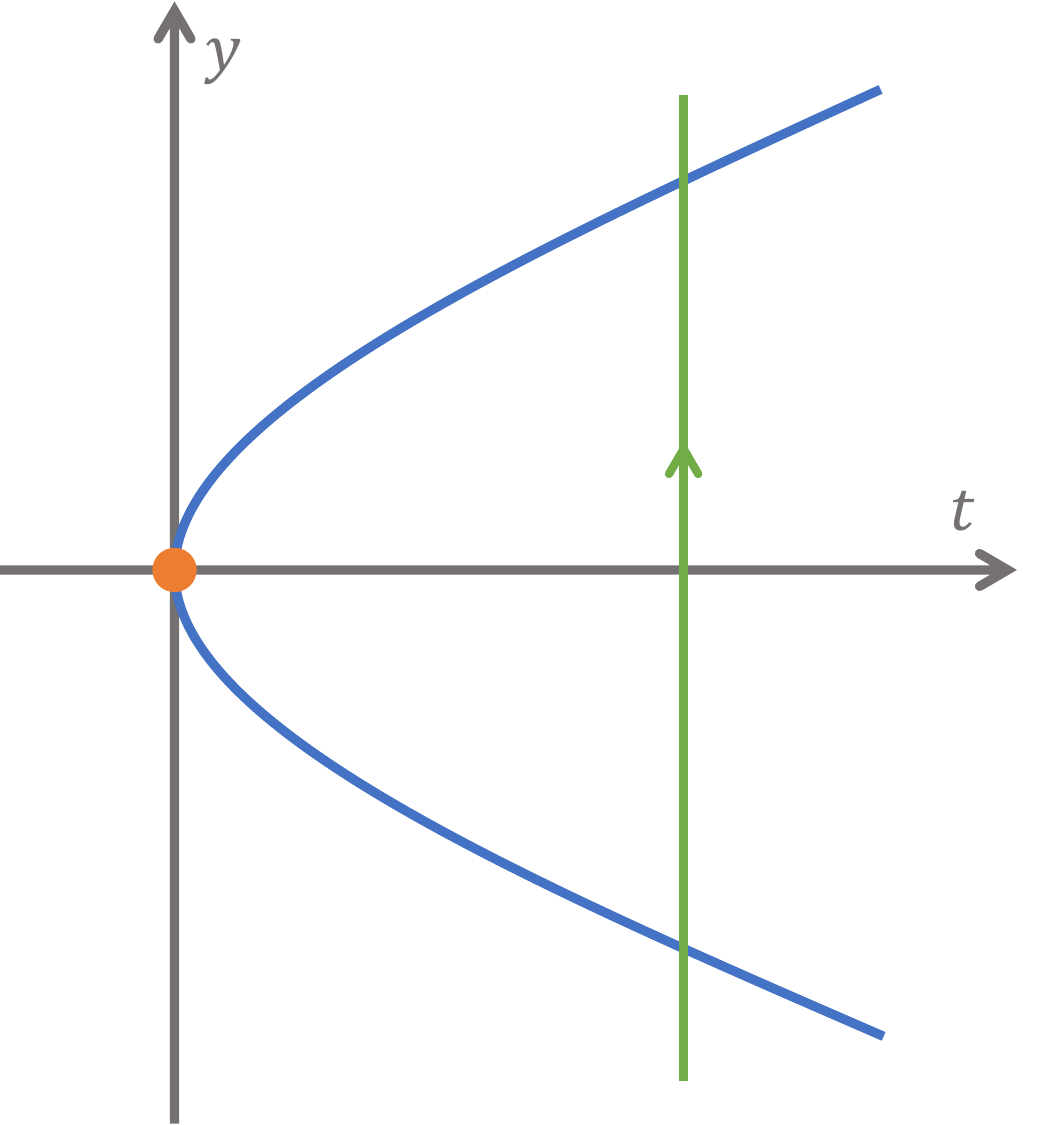}
        \caption{The locus of singularities of the standard cusp. We indicated in green a path in the domain.} \label{fig:cuspDomain}
    \end{subfigure}
    \hspace{0.05\linewidth}
    \begin{subfigure}[t]{0.4\linewidth}
        \centering
        \includegraphics[width=0.7\linewidth,page=3,clip=true,trim = 14cm 0 2cm 0 ]{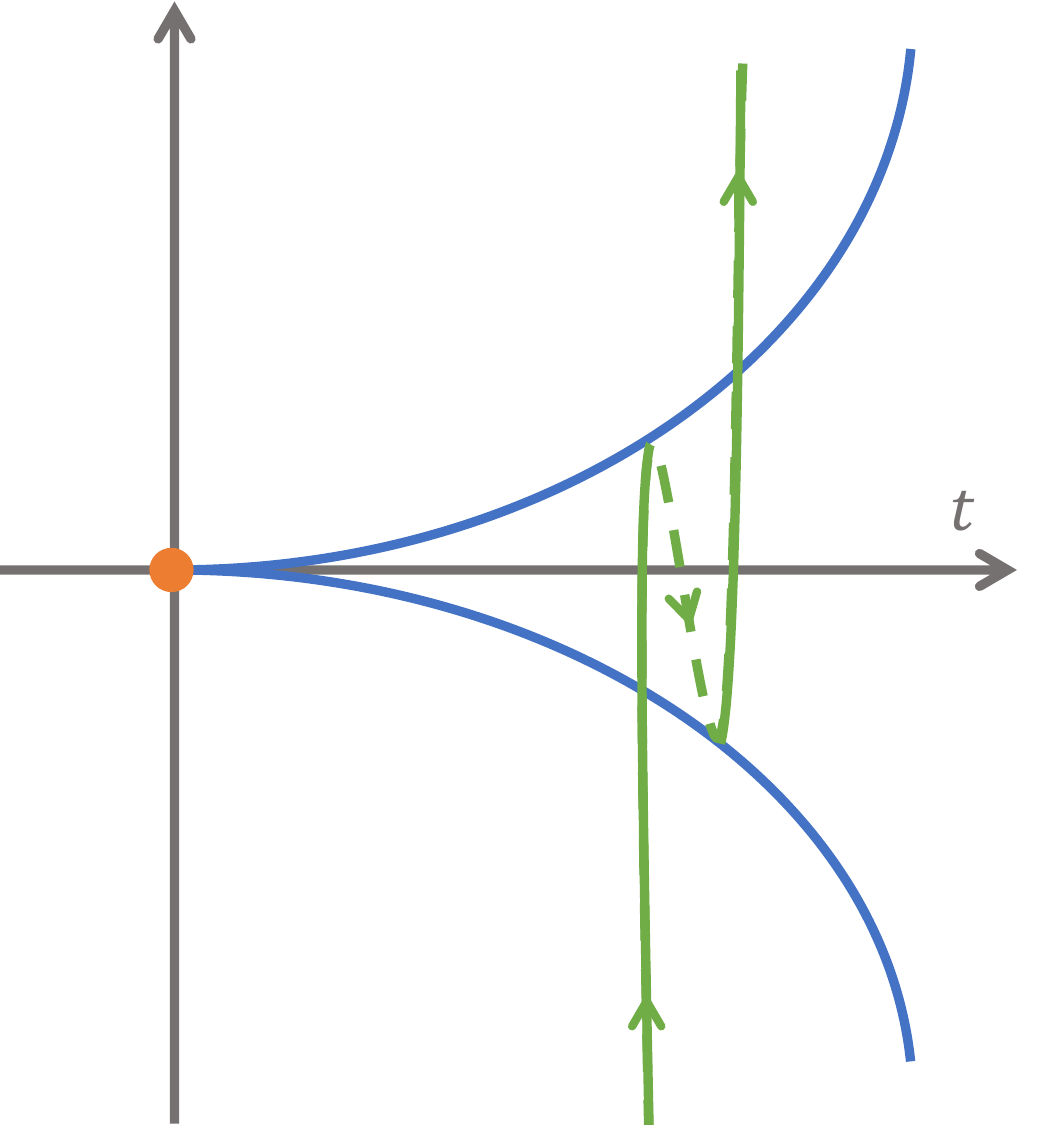}
        \caption{The image of the locus of singularities of the standard cusp. We indicated in green the image of the path in the domain from \cref{fig:cuspDomain}, although we have slightly separated the overlapping segments.} \label{fig:cuspImage}
    \end{subfigure}
    \hfill
    \begin{subfigure}[t]{0.4\linewidth}
        \centering
        \includegraphics[width=\linewidth,page=4,clip=true,trim = 8cm 3cm 2cm 3cm ]{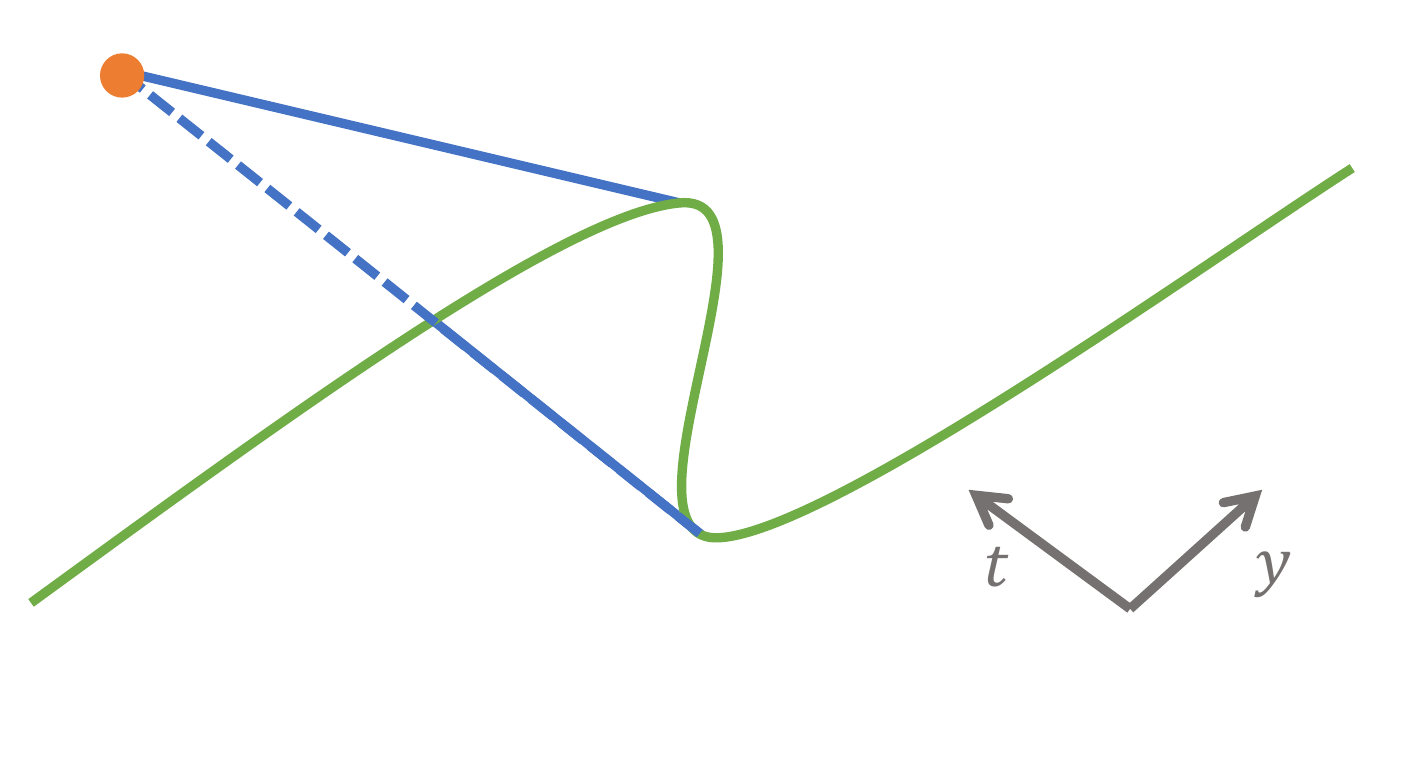}
        \caption{The graph of the last component of the cusp, where the horizontal plane represents the domain and the vertical direction the codomain.} \label{fig:cuspGraph}
    \end{subfigure}
    \hspace{0.05\linewidth}
    \begin{subfigure}[t]{0.4\linewidth}
        \centering
        \includegraphics[width=\linewidth,page=5,clip=true,trim = 3cm 0 3cm 1cm ]{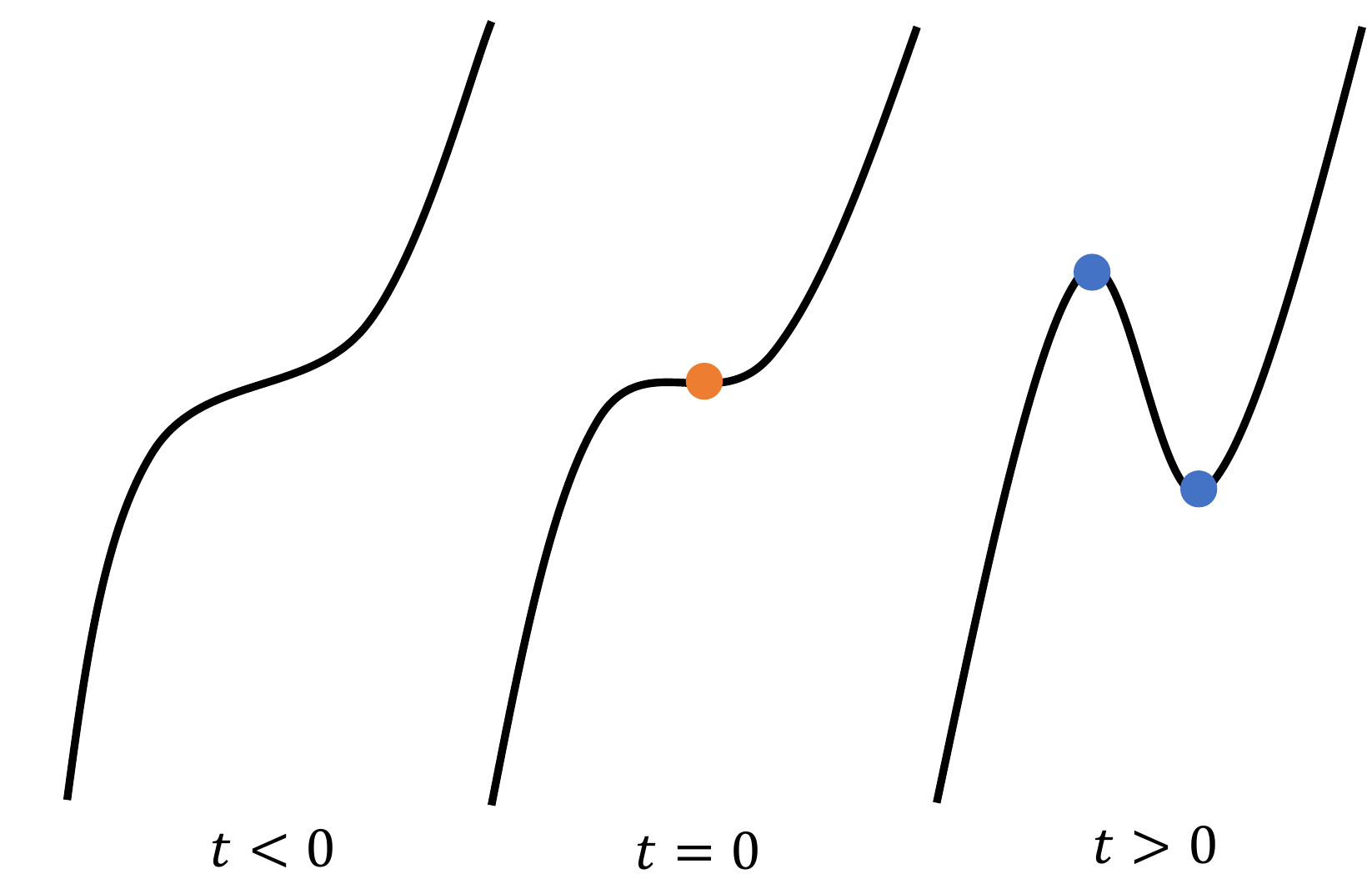}
        \caption{The cusp interpreted as a family of maps. For $t<0$ we have a non-singular map, for $t=0$ an embryo, and for $t>0$ two folds in canceling position (that is, a wrinkle).} \label{fig:cuspFamily}
    \end{subfigure}
	\centering
	\caption{Illustrations of the cusp.} \label{fig:cusp}
\end{figure}

\subsubsection{Wrinkle}\label{sec:wrinkle}
We combine the cusp and the (double) fold in a formally canceling manner in a singularity called the wrinkle. It is in some sense the simplest manner in which folds can cancel, because the first order singularity locus of the wrinkle is just the sphere, and not a topologically more complicated space. We have illustrated the wrinkle in \cref{fig:locusWrinkle}.

\begin{definition}\label{def:stdWrinkle}
    The \textbf{standard wrinkle} $W: \R^m \to \R^n$ of index $k$ is the map given by
    \begin{align*}
        W = W(m,n,s) : \R^{n-1} \times \R^{m-n} \times \R  &\to \R^{n-1} \times \R, \\
         (t,x,y) &\mapsto \left(t,y^3 + 3(|t|^2-1)y - \sum_{i=1}^k x_i^2 + \sum_{j=k+1} ^{m-n} x_j^2 \right).
    \end{align*}
\end{definition}

For the locus of singularity we obtain
\[\Sigma^1(W) = \{|t|^2 + y^2 =1 , x=0 \}\]
which decomposes further as $\Sigma^1(W) = \Sigma^{1,0}(W) \cup \Sigma^{1,1}(W)$, where
\[
    \Sigma^{1,0}(W) = \{|t|^2 + y^2 =1,y\neq0 , x=0 \} \qquad \text{ and} \qquad
     \Sigma^{1,1}(W) = \{|t|^2 =1 , x=y=0 \}.
\]

The lower and upper hemispheres $\Sigma^{1,0}(W)$ consist of folds, whereas the equator $\Sigma^{1,1}(W)$ consists of cusps. The wrinkle restricted to $\Sigma^{1,1}(W)$ is a regular map, so we do not need to consider higher order loci of singularities.

\begin{figure}[h]
    \centering
    \includegraphics[width=0.6\linewidth,page=1]{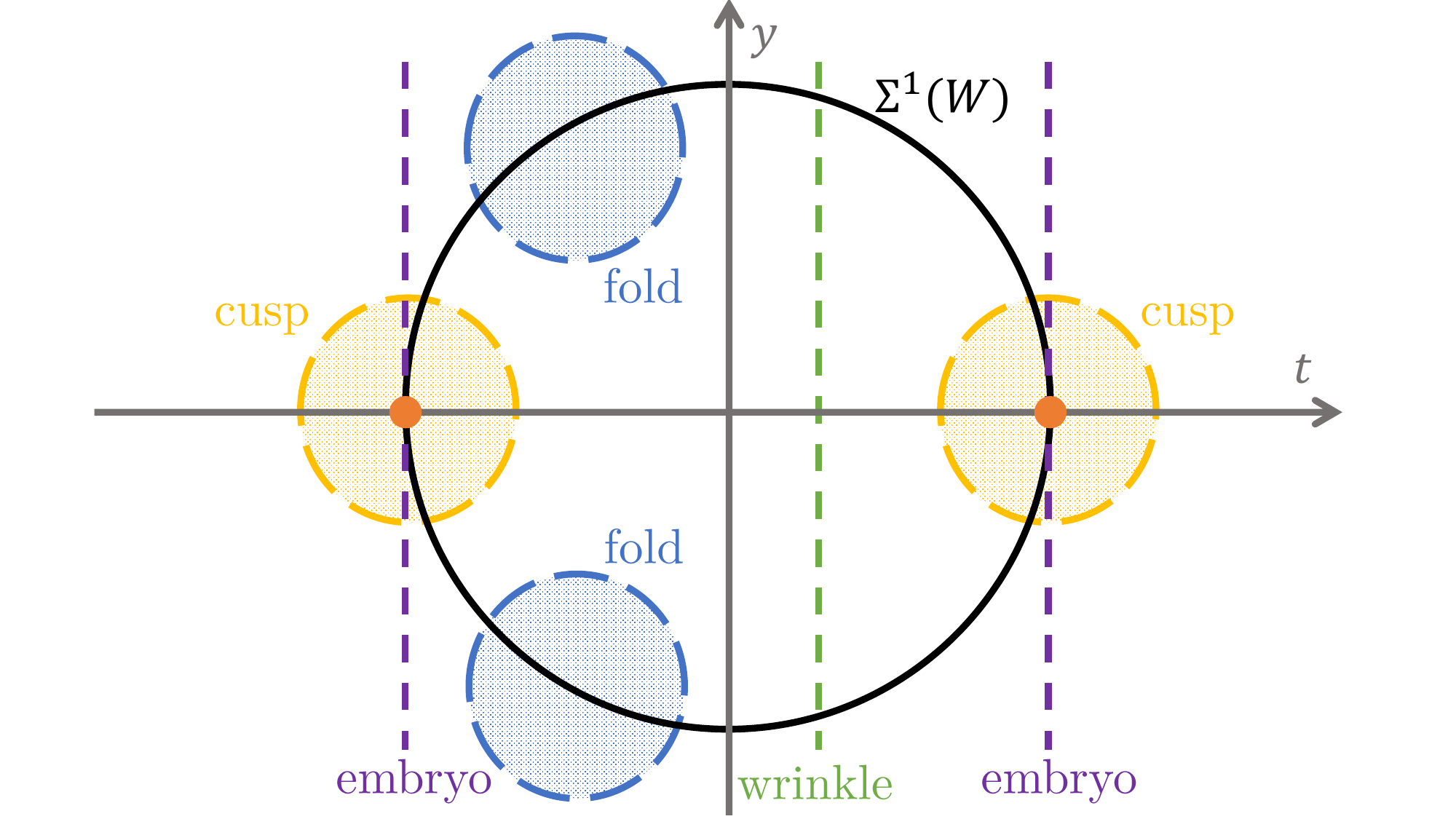}
    \caption{Locus of singularities of a wrinkle. We have indicated what singularity we obtain when restricted to the indicated domain.}
    \label{fig:locusWrinkle}
\end{figure}

When we compute the Jacobian of the wrinkle, we obtain a matrix similar to the Jacobian for the double fold from \cref{eq:DFoldJacob}. The difference is the term in the bottom right corner, which is now $3(y^2 +|t|^2-1)$. However, we again see that, up to homotopy, there is a canonical way of modifying the term in the bottom right corner over $\Op (\D^n)$ such that it becomes positive everywhere. This formally removes the singularity and defines (up to homotopy) the \textbf{regularized differential} $\RD(W): T\R^m \to T\R^n$ of $W$. 

\subsubsection{Embryo}
The event in which the wrinkle is born (or dies) is called the embryo. Hereby we mean that if we generically perturb the embryo, we either obtain a regular map, or a wrinkle. Hence the embryo is not a stable singularity (but it is if we consider families of maps, as in \cref{sec:singFib}). 

\begin{definition}
    The \textbf{standard embryo} $E: \R^m \to \R^n$ of index $k$ is the map given by
    \begin{align*}
        E = E(m,n,s) : \R^{n-1} \times \R^{m-n} \times \R &\to \R^{n-1} \times \R, \\
         (t,x,y) &\mapsto \left(t,y^3 - 3 |t|^2 y - \sum_{i=1}^k x_i^2 + \sum_{j=k+1} ^{m-n} x_j^2 \right).
    \end{align*}
\end{definition}
We obtain
\[\Sigma^1(E) =\{ y^2 = |t|^2, x=0\} \]
which decomposes as $\Sigma = \Sigma^{1,0}(E) \cup \Sigma^{1,1}(E)$, where
\[
\Sigma^{1,0}(E) = \{ y^2 = |t|^2>0, x=0\} \qquad \text{and} \qquad
    \Sigma^{1,1}(E) = \{t=x=y=0\}.
\]

\subsection{Singularities of maps} \label{sec:singMaps}
In \cref{sec:singMapsEucl} we discussed the standard singularities for maps $\R^m \to \R^n$. Hence in \cref{sec:equivMaps} we discuss the notion of equivalence of maps, which allows us in \cref{sec:singMapsMfd} to speak of singularities of maps $M \to N$ between manifolds where $\dim(M) \geq \dim(N)\geq 1$.

\subsubsection{Equivalence of maps} \label{sec:equivMaps}
We introduce the following notion of equivalence for maps. It tells us that locally, around a given subset, the given maps agree up to diffeomorphism.
\begin{definition} \label{def:equivMaps}
    A map $f: M \to N$ is \textbf{equivalent} at $A \subset M$ to a map $g: M' \to N'$ at $A'\subset M'$ if there exist diffeomorphisms $\phi: \Op(A) \to \Op(A')$ and $\psi: \Op(f(A)) \to \Op(g(A'))$ such that $f = \psi g \phi^{-1}$.
\end{definition}

\subsubsection{Singularities of maps between manifolds} \label{sec:singMapsMfd}
With the notion of equivalence for maps, we extend the standard singularities of maps $\R^m \to \R^n$ to singularities of maps between manifolds. We note that the domains we consider differ: for the fold, cusp and embryo, we consider the respective standard maps around the origin, whereas for the double fold and the wrinkle, we consider larger domains.
\begin{definition} \label{def:singMapCFEW}
    A map $f: M \to N$ has a \textbf{singularity of ... type} 
    \begin{itemize}
        \item \textbf{fold/cusp/embryo}: in $x \in M$ if $f$ at $x$ is equivalent to the standard fold/cusp/embryo at $0$.
        \item \textbf{double fold}: in a closed subset $A \subset M$ if $f$ at $A$ is equivalent to the standard double fold at $\{0\} \times [-1,1] \times \{0\}$.
        \item \textbf{wrinkle}: in a closed subset $A \subset M$ if $f$ at $A$ is equivalent to the standard wrinkle at $\overline{D^n}$.
    \end{itemize}
\end{definition}

We note that the singularities of the standard wrinkle may seem to cancel out, as the folds on the upper and lower hemisphere are born and die in the cusps on the equator. Similarly, the singularities of the double fold seem to cancel out. However, this cancellation is only possible when the domain of the double fold or wrinkle is considered to be $\R^n$. We, on the other hand, restrict their domains to much smaller neighborhoods. This implies that the double fold and wrinkle can no longer be canceled out relative to the boundary of the model (but it could cancel out globally). However, we have already seen in \cref{sec:DoubleFold,sec:wrinkle} that they do cancel out in a formal sense.

\subsection{Their fibered nature} \label{sec:singFib}
In this section we discuss the fibered nature of the singularities in \cref{sec:singMapsEucl}. To this end, we first recall the necessary terminology. 

\subsubsection{Fibered maps}
A map $f: M \to N$ is \textbf{fibered} over a manifold $K$ along two submersions $p: M \to K$ and $q: N \to K$ if the diagram
\[\begin{tikzcd}
       M \arrow[r,"f"] \arrow[d,"p"] & N \arrow[d,"q"] \\
       K \arrow[r,equal] & K
   \end{tikzcd}\]
commutes. A \textbf{fibered homotopy} between two maps $f,g: M \to N$ is then a family $h_t$ of fibered maps such that $h_0=f$ and $h_1 =g$. When considering a fibered map $f: M \to N$, it is also natural to consider the \textbf{fibered differential}, which is the map $df|_{T_K M} : T_K M \to T_K N$, where $T_K M$ and $T_K N$ are respectively the subbundles $\ker(p) \subset TM$ and $\ker(q) \subset TN$. We note that the fibered differential is itself a map fibered over $K$.

We generalize the notion of equivalence for maps from \cref{def:equivMaps} to the fibered setting as follows:
\begin{definition}\label{def:equivMapsFib}
    Two fibered maps $f: M \times K \to N \times K$ and $g: M'\times K'\to N'\times K'$ are \textbf{(fibered) equivalent} along $ A\subset M \times K$ and $A'\subset M' \times K'$ and $B \subset K$ and $B'\subset K'$ with $\pr_K(A) \subset B$ and $\pr_K(A') \subset B'$ if there exist diffeomorphisms $\phi: \Op(A) \to \Op(A')$ and $\psi:\Op(f(A)) \to \Op(f(A'))$ and $\chi: \Op(B) \to \Op(B')$ such that the following diagram
    \[\begin{tikzcd}
        M \times K \arrow[dr,"\phi"] \arrow[rrrr,"f"] \arrow[dddrr,"\pr_K",swap] & & & & N \times K \arrow[dl,"\psi",swap] \arrow[dddll,"\pr_K"]\\
       & M'\times K' \arrow[rr,"g"] \arrow[dr,"\pr_{K'}",swap] & & N'\times K' \arrow[dl,"\pr_{K'}"] & \\
       & & K' & & \\
       & & K \arrow[u,"\chi"] & & 
   \end{tikzcd}\]
    commutes when restricted to the domains of map listed in the diagram.
\end{definition}

\subsubsection{Fibered standard singularities}\label{sec:FibStdSing}
We now make the observation that each of the standard singularities in \cref{sec:singMapsEucl} is fibered over $\R^{n-1}$. 

In particular, when we go back to the cusp from \cref{def:strCusp} we see that if we fix $t_1 \in \R^{n-1}$ such that $t_1>0$, the resulting map is a double fold. If instead $t_1<0$ we obtain a regular map. If $t_1=0$ we obtain a map known as the double fold embryo, which we have no explicit use for, and hence did not define separately.

Similarly, when we go back to wrinkle from \cref{def:stdWrinkle}, we see that if we fix $t\in \R^{n-1}$ such that $|t|>1$, the resulting map is non-singular. If instead $|t|<1$, we obtain a wrinkle of lower dimension, and if $|t|=1$ an embryo. In the case that $n=1$, so that $t\in\R$, the wrinkle coincides with the double fold.

\subsubsection{Fibered regularized differential}
For the double fold $D$ we introduced in \cref{sec:DoubleFold} its regularized differential $\RD_D(D)$. When interpreting $D$ as a map fibered over $\R^{n-1}$, we define its \textbf{fibered regularized differential} as the regularization of its fibered differential $d_{\R^{n-1}}D$.

We define the fibered regularized differential similarly for the wrinkle $W$.

\subsubsection{Singularities of fibered maps between manifolds} \label{sec:singFibMfds}
Analogously to \cref{def:singMapCFEW}, a fibered map $M \to N$ has a singularity of \textbf{fibered fold/cusp/embryo/double fold/wrinkle} type if it is fibered equivalent to the corresponding standard singularity as in \cref{def:singMapCFEW}.

\subsection{Wrinkled submersions} \label{sec:wrinklSubm}
Eliashberg and Mishachev study in \cite{ElMiWrinI} submersions with wrinkle singularities, to which we will refer to as wrinkled submersions. We point out that their definition in particular makes sure that wrinkles are not nested.
\begin{definition} \label{def:wrMap}
    A map $f: M \to N$ between manifolds $M$ and $N$ with $\dim M \geq \dim N$, is a \textbf{wrinkled submersion} if there exists a collection of disjoint opens $\{B_i\}_{i \in I}$ in $M$, each diffeomorphic to $\R^m$, such that: 
    \begin{itemize}
        \item $f | _{M \setminus \sqcup_i B_i}$ is regular, and 
        \item $f|_{B_i} $ is a wrinkle.
    \end{itemize}
\end{definition}

Analogously, we say that a family of maps $F_k : M \to N$, where $k$ varies in a compact manifold $K$, is a \textbf{fibered family of wrinkled submersions}, if it is a fibered family of submersions whose singularities are fibered wrinkles. 

\subsubsection{Regularized differential}\label{sec:regDiff}
Given a (fibered) wrinkled submersion $s: M \to N$ as in \cref{def:wrMap}, we define its (fibered) regularized differential $\RD(s)$ on $M \setminus \sqcup_i B_i$ as its ordinary (fibered) differential. Locally, on each of the $B_i \subset M$, we define it as the (fibered) regularized differential from \cref{def:stdWrinkle}. By a slight abuse of notation we will also refer by (fibered) regularized differential to the first jet of $W$ modified in a similar manner.

\section{Wrinkling into \'etale space} \label{sec:wrinklingEtale}

In this section we prove our statements about wrinkled solutions of partial differential relations. We prove the existence h-principle of \cref{thm:wrinklingEtale} locally in \cref{ssec:wrinklingCube} and globally in \cref{sec:ProofWrinklingEtale} (including their respective relative and parametric versions).

\subsection{Wrinkling in a cube} \label{ssec:wrinklingCube}
The key observation behind \cref{thm:wrinklingEtale} is the upcoming proposition (and its fibered counterpart \cref{lem:wrinkFib}). The main idea of its proof is that, given a formal section $F: I^n \to J^r\Psi$ over the $n$-dimensional unit cube $I^n$, the parametric version of holonomic approximation provides us with a family of holonomic approximations $F_y$ each defined on a domain that is diffeomorphic to a neighborhood of $I^{n-1} \times \{y\}$. Hence each of these $F_y$ lift on their domain of definition to a map into $\Etale(\Psi)$. With a construction similar to the one for wrinkled submersions in \cite[Lemma 2.3A]{ElMiWrinI}, we glue the family of holonomic approximations to one single map into $\Etale(\Psi)$. This gluing introduces the wrinkle singularities.

\begin{proposition} \label{lem:wrink}
Let $\Psi \rightarrow I^n$ be a fiber bundle over the $n$-cube. Consider a section $F:I^n \to J^r\Psi$ such that $F = j^r f$ on a neighborhood $V$ of $\partial I^n$, for some $f:I^n \to \Psi$. Then, there exists a wrinkled submersion $G:I^n \to \Etale{\Psi}$ such that:
\begin{itemize} 
    \item $p_r \circ G$ and $F$ are $C^0$-close and agree on $\partial I^n$, and \label{lem:wrinksecdist}
    \item the wrinkled submersion $p_b \circ G: I^n \to I^n$ is $C^0$-close to the identity, and 
    \item the regularized differential of $p_b \circ G$ is homotopic to $j^1(\id)$ via formal submersions relative to $\partial I^n$.
\end{itemize}
\end{proposition}
\begin{proof}
Throughout the proof it will be convenient to assume that $F$ is defined on an open neighborhood $\Op(I^n) \subset \R^n$. This can always be arranged by replacing $F$ by its restriction to $[-1+\alpha,1-\alpha]^n \subset I^n$ for some small $\alpha > 0$. We let $x = (\tilde{x},x_n)$ denote the coordinates on $I^n = I^{n-1} \times I$. We fix a metric in the cube and in $\Psi$ so that we can make sense of distances in the cube and between sections of $\Psi$.

\pfstep{Holonomic approximation} We define a $1$-parameter family 
\[ F_y: U_y  \to J^r\Psi \]
by the formula $F_y(x) = F(x)$, where $U_y\subset \R^n$ is an open neighborhood of $I^{n-1}\times \{y\}$. This family can be made holonomic using the $1$-parametric version of holonomic approximation \cite{ElMi} by wiggling the core $I^{n-1} \times \{y\}$ of each strip $U_y$ relative to $V$, which we now make precise:

Recall that holonomic approximation constructs a family of diffeotopies $h_{y,t}:U_y  \to U_y$, that wiggle (the core in) the domain of $F_y$. In our case we impose that this wiggling takes place in the direction of the last variable $x_n$. We point out that the amount of wiggling depends on the size of the derivatives of $F_y$. Because of compactness, we have uniform bounds for the derivatives and we can hence find a uniform amount of wiggling with which we are able to make all of the $F_y$ holonomic. In particular we can arrange for the necessary wiggling of all cores $I^{n-1} \times \{y\}$, where $y$ is away from $\partial I$, to be achieved by a single diffeotopy. We can assume the diffeotopy to be the identity close to $\partial I^n$, since $F$ is already holonomic there.

That is, holonomic approximation tells us that: For any $\eta > 0$ there exists a diffeotopy
\[ h_t: \Op(I^{n}) \to \Op(I^{n})  \]
with $t \in [0,1]$
and a smooth family of holonomic sections 
\[ \tilde F_y : \tilde{U}_y  \to J^r\Psi \]
with $y\in I$, each defined on an open neighborhood $\tilde{U}_y\subset U_y$ of $h_1(I^{n-1} \times \{y\})$,
such that:
\begin{enumerate}[label=(\roman*)]
    \item $h_t = \id$ on $\Op(\partial I^{n})$ and if $t=0$, 
    \item $d_{C^0}(h_t(x),x)<\eta$ for all $x \in \Op(I^{n})$ \label{lab:wrdiffsmall},
    \item $d_{C^0}(\tilde{F}_y(x),F_y(x)) < \eta$ \label{lab:wrhaclose} for $x \in \tilde{U}_y$ and $y \in I$, and
    \item $\tilde{F}_y = F$ on $\tilde{U}_y \cap \Op(\partial I^{n-1} \times I)$ for all $y \in I$, and on $\tilde U_y$ for $y \in \Op(\partial I)$.
\end{enumerate}
    
By compactness there exists $\delta >0 $ such that each $\tilde F_y$ is holonomic on $h_1(I^{n-1} \times(y - \delta,y + \delta))$. Denote by $\tilde{f}_y: h_1(I^{n-1} \times (y-\delta,y+\delta)) \to \Psi$ the family of functions such that $\tilde{F}_y = j^r \tilde{f}_y$. Taking germs we obtain a family of sections
\begin{align*}
    s_y: h_1(I^{n-1} \times (y-\delta,y+\delta)) &\to \Etale(\Psi) \\
    x &\mapsto [\tilde{f}_{y}]_x,
\end{align*}
for all $y \in I$.

At this point we reparametrize the cube and just work in the coordinates given by $h_1$. This allows us to assume that the maps $\tilde F_y$, $\tilde f_y$, and $s_y$ all have $I^{n-1} \times (y-\delta,y+\delta)$ as domain.

\pfstep{Gluing the holonomic sections}
Each $\tilde F_y$ is a $C^0$-close holonomic approximation of our starting data, so the rest of the proof amounts to using wrinkling to `glue together' the family $s_y$ in order to obtain a single map $M \to \Etale{\Psi}$ whose projection to $J^r\Psi$ is a $C^0$-close (multivalued) holonomic approximation of $F$. This idea is illustrated in \cref{fig:wrinkling}. Our gluing procedure is very similar to the one used by Eliashberg and Mishachev in \cite[Lemma 2.3A]{ElMiWrinI}.
\begin{figure}[h]
    \centering
    \includegraphics[width=0.7\textwidth,page=1]{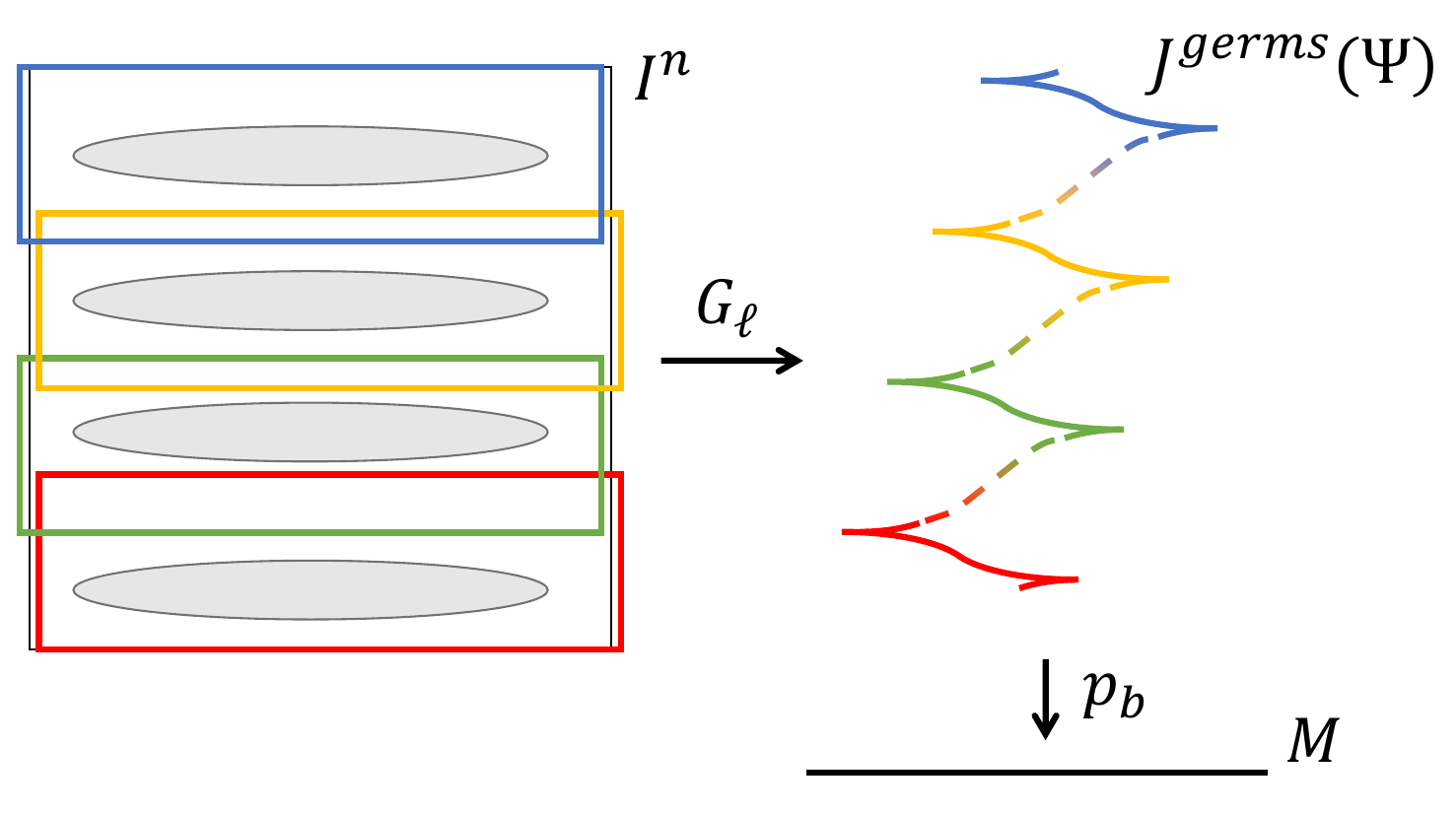} 
    \caption{A sketch of wrinkling map $G_\ell$: on the left we see its domain $I^n$ with the domains of four $s_y$. The grey ovals indicate where the wrinkling takes place. On the right we see the image of $G_\ell$ in $\Etale \Psi$. We have drawn the wrinkles there in the color of the strip in which they take place. The dotted lines indicate the interpolation in between the strips. Note that the wrinkles do not overlap in the domain, but they do in the target when projected to $M$.} \label{fig:wrinkling}
\end{figure}
 
Let $\ell \in \N$ to be determined later and define $\Delta_i \subset I$ as the interval of length $\frac{9}{16\ell}$ centered around $t_i = \frac{2i-1}{2\ell}$ for $i = 1,\dots, \ell$. Let $\lambda : I \rightarrow I$ be a function such that:
	\begin{itemize}
		\item $\lambda(0) = 0$ and $\lambda(1) = 1$,
		\item $\lambda (t) = t_i$ on each $\Delta_i$, and
		\item $0 < \frac{d\lambda}{d t} (t)<3$ if $t \in I \setminus \cup_{i=1}^\ell \Delta_i$.
	\end{itemize}
Then there exists a positive and increasing function $\nu:\N\to\R$ such that $\lim_{\ell\to\infty}\nu(\ell)= 0$ and $\lambda(x) \in [x-\nu(\ell),x+\nu(\ell)]$ for all $x\in I$.

We also define the map
\begin{align*}
    \gamma_{\delta,\ell}&:I^n \to I^n \\
    (\tilde x,x_n) &\mapsto (\tilde x,\tilde \gamma_{\delta,\ell}(x)), 
    \intertext{which is fibered over $I^{n-1}$ and where $\tilde \gamma_{\delta,\ell}:I^n \to I$ is defined by\footnotemark}\noalign{\footnotetext{Compared to \cite[Lemma 2.3A]{ElMiWrinI} we have the extra factor $\delta$ in our definition of $\tilde\gamma_{\delta,\ell}$ since our $s_y$ are only defined over $I^{n-1} \times (y-\delta,y+\delta)$. This implies we also get an extra factor $\delta$ in the lower bound on $\dfrac{d \tilde \gamma_{\delta,\ell}}{dt}$, but since we take $\ell$ arbitrarily large, all the estimates of \cite[Lemma 2.3A]{ElMiWrinI} still hold true.}}
    \tilde\gamma_{\delta,\ell}(x) &= x_n + \chi(\tilde x)  \phi_\ell(x_n) \delta \sin(2\pi \ell x_n).
\end{align*}

Here $\phi$ and $\chi$ are two bump functions making sure that the following holds:\footnote{More precisely, we choose the bump functions as in \cite[Lemma 2.3A]{ElMiWrinI}: We choose a bump function $\phi_\ell:I \to I$ such that $\phi|_{[0,\frac1{16\ell}] \cup [1-\frac1{16\ell},1]} = 0$ and $\phi|_{[\frac1{8\ell},1-\frac1{8\ell}]} = 1$. We let $\chi:I^{n-1} \to I$ be another bump function such that $\chi$ is equal to $0$ on a neighborhood of $\partial I^{n-1}$ and $1$ outside a larger neighborhood of $\partial I^{n-1}$. 
We also require that the superlevel sets $\{x \in I^{n-1} \mid \chi(x) \geq t \}$ are convex for $t\in[0,1]$, and that the boundary of the level set $\{x \in I^{n-1} \mid \chi(x) = 1 \}$ is contained in $V$. }
\begin{itemize}
    \item $\gamma_{\delta,\ell}$ has a single wrinkle on $I^{n-1} \times \Delta_i$,
    \item $\gamma_{\delta,\ell}$ is a submersion on $I^n \setminus \bigcup_i I^{n-1} \times \Delta_i$, and
    \item $\gamma_{\delta,\ell}$ is homotopic to the identity relative to the given neighborhood $V$ of $\partial I^n$. 
\end{itemize}

Then we define a map $G_\ell : I^n \rightarrow \Etale{\Psi}$ by
	\[
	G_\ell (x) = s_{\lambda(x_n)}(\gamma_{\delta,\ell}(x)).
	\]

We point out that each wrinkle of $\gamma_{\delta,\ell}$ takes place over $I^{n-1} \times \Delta_i$ where $\lambda$ is constant. Hence over $I^{n-1} \times \Delta_i$, the map $G_\ell$ takes values in a basic open of \'etale space, specifically the graph of an $s_y$ where $y$ is fixed, which is diffeomorphic to an open in Euclidean space. We furthermore check that 
	\begin{itemize}
		\item $G_\ell(x) = [f]_x $ for $x \in \Op(\partial I^n)$,
		\item $G_\ell$ is a submersion over $I^{n-1} \times (I \setminus \cup_i \Delta_i)$ for $\ell$ large enough\footnote{The proof for this item is analogous to \cite[Lemma 2.3A]{ElMiWrinI}. This is also the part of the proof where the bounds on the derivatives of $\lambda$ and $\tilde{\gamma}_{\delta,\ell}$ come into play.}, and
		\item $G_\ell = s_{t_i} \circ \gamma_{\delta,\ell}$ over $I^{n-1} \times \Delta_i$.
	\end{itemize}

\pfstep{Chopping the wrinkles}
A composition of a submersion and a wrinkled map is a wrinkled map provided that the image of each wrinkle is sufficiently small (so that it is contained in a submersion chart). To achieve this, we homotope the map $\gamma_{\delta,\ell}$ into a map with more wrinkles, but that are smaller, by a procedure known as ``chopping the wrinkles''. Since $\gamma_{\delta,\ell}$ is a map $\R^n \to \R^n$ and in particular does not depend on $\Etale\Psi$, the reasoning by Eliashberg and Mishachev \cite[Lemma 2.1C]{ElMiWrinI} is also valid in our \'etale case.\footnote{A more detailed account of chopping wrinkles can also be found in \cite[Section 13.4]{CiElMi}} Hence we define $G$ as the result of chopping the wrinkles of $G_\ell$ for $\ell$ large enough.

\pfstep{As a map $I^n \to I^n$} We observe that $p_b \circ G : I^n \to I^n$ equals the wrinkled submersion as constructed in \cite{ElMiWrinI}. Hence the second and third property from the proposition follow from immediately.

\pfstep{Checking $C^0$-closeness} Here we now deviate again from Eliashberg and Mishachev. The only property left to verify is that $p_r \circ G$ and $F$ are $C^0$-close. Since the chopping of the wrinkles results in a $C^0$-perturbation, it suffices to argue that $p_r \circ G_\ell$ approximates $F$.

Define the map $\tilde F: I^n \to I^n$ by $\tilde F(x) = \tilde F_{x_n}(x)$ where $x = (x_1,\dots,x_n)$. We may assume that $d_{C^0}(\tilde{F} , F)$ is arbitrarily small by taking $\eta$ from \cref{lab:wrhaclose} arbitrarily small. It then suffices to show that
\[d_{C^0}(p_r \circ G_\ell , \tilde{F}) < \epsilon.\]

We recall that it holds that $\lambda(x_n) \in [x_n - \nu(\ell) , x_n + \nu(\ell)]$ and that $\tilde{\gamma}_{\delta,\ell} (x) \in [x_n - \delta, x_n + \delta]$. By compactness we can assume that $\delta$ is small enough such that 
\[  d_{C^0}( \tilde{F}_y (\tilde{x}, x_n) , \tilde{F}_y (\tilde{x}, x'_n) ) < \epsilon/2 \]
for all $y \in I$, all $\tilde{x} \in I^{n-1}$ and $x_n, x'_n \in [y-\delta,y+\delta]$.  

Since the family $\{\tilde{F}_y\}$ is smooth in the parameter we additionally have that for $\ell$ large enough
    \[ 	d_{C^0}( \tilde{F}_a (x), \tilde{F}_b (x) ) < \epsilon/2. \]
for all $x \in I^n$ and $a,b\in I$ such that $|a - b | \leq 2\nu(\ell)$.

Together, this implies that 
	\begin{align*}
		d( p_r \circ G_\ell (x), \tilde{F}_{x_n}(x) ) 
		&= d( \tilde{F}_{\lambda(x_n)} (\tilde x, \tilde{\gamma}_{\delta,\ell} (x)), \tilde{F}_{x_n} (x) ) \\
		&\leq d( \tilde{F}_{\lambda(x_n)} (\tilde x, \tilde{\gamma}_{\delta,\ell} (x)), \tilde{F}_{\lambda(x_n)} (x) ) + d( \tilde{F}_{\lambda(x_n)} (x), \tilde{F}_{x_n} (x) ) \\
        &< \epsilon,
	\end{align*}
where we write $x=(\tilde x, x_n)$, which concludes the proof.
\end{proof}

Next we prove the fibered version of \cref{lem:wrink}, by indicating the changes that need to be made to the proof of \cref{lem:wrink}. 
\begin{proposition}\label{lem:wrinkFib}
Let $F : I^n \times I^\ell \rightarrow J^r(\Psi \times I^\ell)$ be a section fibered over $I^\ell$. Suppose that $F$ is holonomic over $\Op(\partial I^n) \times I^\ell$ and $I^n \times \Op(\partial I^\ell)$. Then, there exists a wrinkled submersion $G: I^n \rightarrow \Etale{(\Psi)}$ that is fibered over $I^\ell$ such that:
\begin{itemize}
    \item the maps $p_r \circ G$ and $F$ are $C^0$-close, and agree on $\partial I^n \times I^\ell$ and $I^n \times \partial I^\ell$,
    \item the map $p_b \circ G: I^n \times I^\ell \rightarrow I^n \times I^\ell$ is a wrinkled submersion fibered over $I^\ell$ and $C^0$-close to the identity, and
    \item the regularized differential of $p_b \circ G$ is homotopic to the constant family $j^1(\id)$ via families of fibered formal submersions, relative to $\partial I^n \times I^\ell$ and $I^n \times \partial I^\ell$.
    \end{itemize}
\end{proposition}
\begin{proof}
We apply the proof of \cref{lem:wrink} to $F$ and obtain the desired wrinkled submersion $G:I^n \times I^\ell \to \Etale{(\Psi \times I^\ell)}$. There are two parts of the proof that warrant further comment:

The first part is when we use holonomic approximation to pass from $F'$ to an $I^{\ell +1}$-parameter family of holonomic sections on $\Op(I^{n-1} \times \{(y,k)\}$. Here we point out that by wiggling in the $I^n$-direction, we can arrange the diffeotopy $h_t$ on $\Op(I^n \times I^\ell)$ to be fibered over $I^\ell$.

The second part is where we glue the holonomic approximations to a single map using wrinkles. The wrinkles as constructed in the proof of \cref{lem:wrink} can be made to be fibered over $I^\ell$. Hence the proof will result in a map $G: I^n \times I^\ell \rightarrow \Etale{(\Psi \times I^\ell)}$ fibered over $I^\ell$.  
\end{proof}

\subsection{Wrinkling in a manifold}
\label{sec:ProofWrinklingEtale}
We now apply a reasoning quite common for h-principles to deduce the global analogue of \cref{lem:wrink}. That is, we triangulate the manifold $M$, apply holonomic approximation on its codimension one skeleton and use \cref{lem:wrink} to deal with the top dimensional cells.

\wrinklingEtale
\begin{proof}
    Consider a section $F:M\to J^r\Psi$ over an $m$-dimensional manifold $M$. We choose a triangulation $\ST$ of $M$, and denote its $i$-skeleton by $\ST^{(i)}$. Applying holonomic approximation, we obtain a $\delta$-small (in $C^0$-sense) diffeotopy $h_t: M \rightarrow M$ with $t \in [0,1]$ and a formal solution $\tilde{F} : M \rightarrow \SR$ such that $\tilde{F}$ is holonomic over $\Op(h_1(\ST^{(m-1)}))$ and such that $F$ and $\tilde F$ are arbitrarily close in the $C^0$-sense.
	
    It remains to make $\tilde{F}$ holonomic on the top dimensional cells of $h_1(\ST)$. To do so, we work one $m$-simplex $\Delta$ at a time. Our section $\tilde{F}$ is holonomic in a neighborhood of the boundary of $\Delta$. We then parametrize $\Delta \setminus \Op(\partial \Delta)$ by a cube $I^n$. As such, we are in the situation of \cref{lem:wrink}, which produces a wrinkled submersion over the cube, relative to a neighborhood of the boundary. Repeating this for every simplex gives us the wrinkled submersion $G: M \rightarrow \Etale{\Psi}$.
\end{proof}

\begin{remark}
    We expect that the main result in \cite{PT1}, by the second and third author, can also be used to prove \cref{thm:wrinklingEtale}. It is \cite[Theorem 1.2]{PT1} that states that any section of $r$-jet space can be approximated by a holonomic multi-valued section that is additionally topologically embedded (and moreover, its projection to $\Psi$ is an $r$-order wrinkled embedding). This is a stronger result than \cref{thm:wrinklingEtale}, but also much more technical due to the embedding requirement. Moreover, it does not directly imply our our main result since some work is necessary to lift the multi-valued section to \'etale space. Hence we provide here a direct proof, which is as such also easier than using \cite{PT1}. 
\end{remark}

We now turn to the fibered version of \cref{thm:wrinklingEtale}, whose proof adapts the arguments from the proof of \cref{thm:wrinklingEtale} to the fibered setting. 

\begin{proposition} \label{pro:wrinklingEtaleMfdFib}
    Let $K$ be a compact manifold and let $F : M \times K \rightarrow J^r(\Psi \times K)$ be a section fibered over $K$. Suppose that $F$ is holonomic over on $\Op(L)$ where $L\subset M \times K$ is a closed subset. Then, there exists a wrinkled submersion $G: M \times K \rightarrow \Etale{(\Psi \times K)}$ that is fibered over $K$ such that:
\begin{itemize}
    \item the maps $p_r \circ G$ and $F$ are $C^0$-close, and agree on $L$,
    \item the map $p_b \circ G: M \times K \rightarrow M \times K$ is a wrinkled submersion fibered over $K$ and $C^0$-close to the identity, and
    \item the regularized differential of $p_b \circ G$ is homotopic to the constant family $j^1(\id)$ via families of fibered formal submersions, relative to $L$.
    \end{itemize}
\end{proposition}
\begin{proof}
    The manifold $M \times K$ is foliated by the fibers of $\pi_K$, and by Thurston's jiggling~\cite{Th2,FPTjiggling} there exists a triangulation $\ST$ of $M \times K$ in general position with respect to this foliation. This means in particular that each top dimensional simplex $\sigma \in \ST^{(m+k)}$ is transverse to the leaves of the foliation.

    Over the $(m+k-1)$-skeleton we apply holonomic approximation, to get a ($C^0$-small) diffeotopy $h_t: M \times K \to M \times K$, $t \in [0,1]$, and a map $\tilde F:M \times K \to J^r(\Psi \times K)$ that is both holonomic and $C^0$-close to $F'$ in a neighborhood of $h_1(\ST^{m+k-1})$. Recall that the proof of holonomic approximation allows us to choose the direction in which $h_t$ wiggles $\ST^{m+k-1}$. As such we arrange that $h_1(\ST)$ is still transverse to $\ker (\d \pi_K)$.

    It remains to extend $\tilde F$ to the top dimensional simplices of $h_1(\ST)$. As in the proof of \cref{thm:wrinklingEtale} we do this one simplex at a time, using \cref{lem:wrinkFib}. This results in the wrinkled submersion $G:M \times K \to \Etale{(\Psi \times K)}$. 
    
    If $L$ is non-empty we apply all our arguments relative to $L$. In particular we need to choose a triangulation sufficiently small such that each top dimensional cell is either contained in $\Op(L)$ or is disjoint from $L$.
\end{proof}

The fibered statement of \cref{pro:wrinklingEtaleMfdFib} implies in particular the following statement about families from the introduction.

\label{sec:ProofWrinklingEtaleParametric}
\wrinklingEtaleParametric
\begin{proof}
    Consider the section $F:M \times K \to J^r(\Psi)$ defined by $F(x,k) = F_k(x)$. We choose an arbitrary lift 
    \begin{equation*}\label{eq:ProductSection}
    F':M \times K \to J^r(\Psi \times K)
    \end{equation*}
    We apply the proof of \cref{pro:wrinklingEtaleMfdFib} to $F'$ and obtain a wrinkled submersion $G':M \times K \to \Etale{(\Psi \times K)}$ that is fibered over $K$. We define the $K$-family $G_k: M \to \Etale(\Psi)$ by $G_k(m) = 
\pi G(m,k)$, where $\pi$ is the projection $\Etale(\Psi \times K) \to \Etale(\Psi)$. 

    Since the map $G'$ is fibered, it follows immediately that each $G_k$ is continuous for the \'etale topology. However, since the variables in $M$ and $K$ are independent, we cannot deduce that the family $G$ as a whole is also continuous for the \'etale topology. It will however be continuous for the Whitney topology, as $G'$ is continuous and the Whitney topology is coarser than the \'etale topology (recall also \cref{ex:WhitneyEtaleTop}).

    We leave it to the reader to check that the $G_k$ satisfy the other properties in \cref{thm:wrinklingEtaleParametric}.
\end{proof}

\section{Applications of wrinkling} \label{sec:applWrinkling}

We now cover some of the consequences of \cref{thm:wrinklingEtale}. In \cref{sec:proofConnectivityEtale} we address the proof of \cref{prop:connectivityEtale}, which deals with the connectivity of \'etale space. In \cref{ssec:folded} we recover the h-principle for folded symplectic forms of Cannas da Silva. In \cref{ssec:horizontal} we recover a result about horizontal homotopy groups in jet spaces.

\subsection{Connectivity of \'etale space} 
\label{sec:proofConnectivityEtale}
\Cref{thm:wrinklingEtale} tells us that we can lift sections of $\SR$ to maps into $\EtSol{M}$. By considering instead just maps into $\SR$ and seeing whether they can be lifted to $\EtSol{M}$, we relate the connectivity of $\EtSol{M}$ to $\SR$. 

Given a map $F:S^i \to\SR$, we can triangulate $S^i$ and make the images of the simplices graphical over the base, so that they can be interpreted as images of formal sections of $\SR$. However, in this situation we cannot apply holonomic approximation to the codimension one skeleton as we did for \cref{thm:wrinklingEtale}, since the simplices may overlap in the base $M$. Hence, we need to individually make each of the formal sections holonomic. 

An illustration of the upcoming proof is given in \cref{fig:etale}.
\begin{figure}[h]
    \centering
    \begin{subfigure}[t]{0.3\linewidth}
        \centering
        \includegraphics[width=\linewidth,page=2,clip=true,trim = 2cm 0 2cm 0 ]{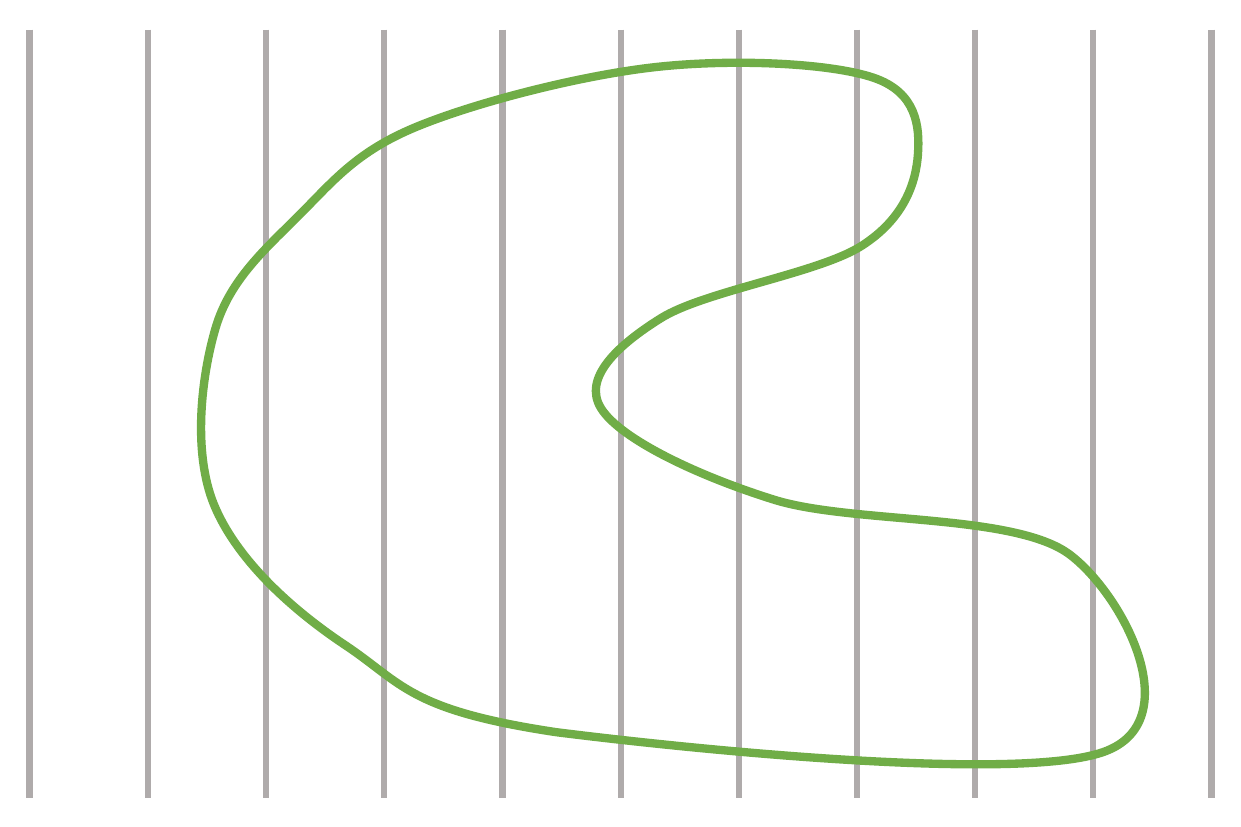}
    \end{subfigure}
    \hfill
    \begin{subfigure}[t]{0.3\linewidth}
        \centering
        \includegraphics[width=\linewidth,page=3,clip=true,trim = 2cm 0 2cm 0 ]{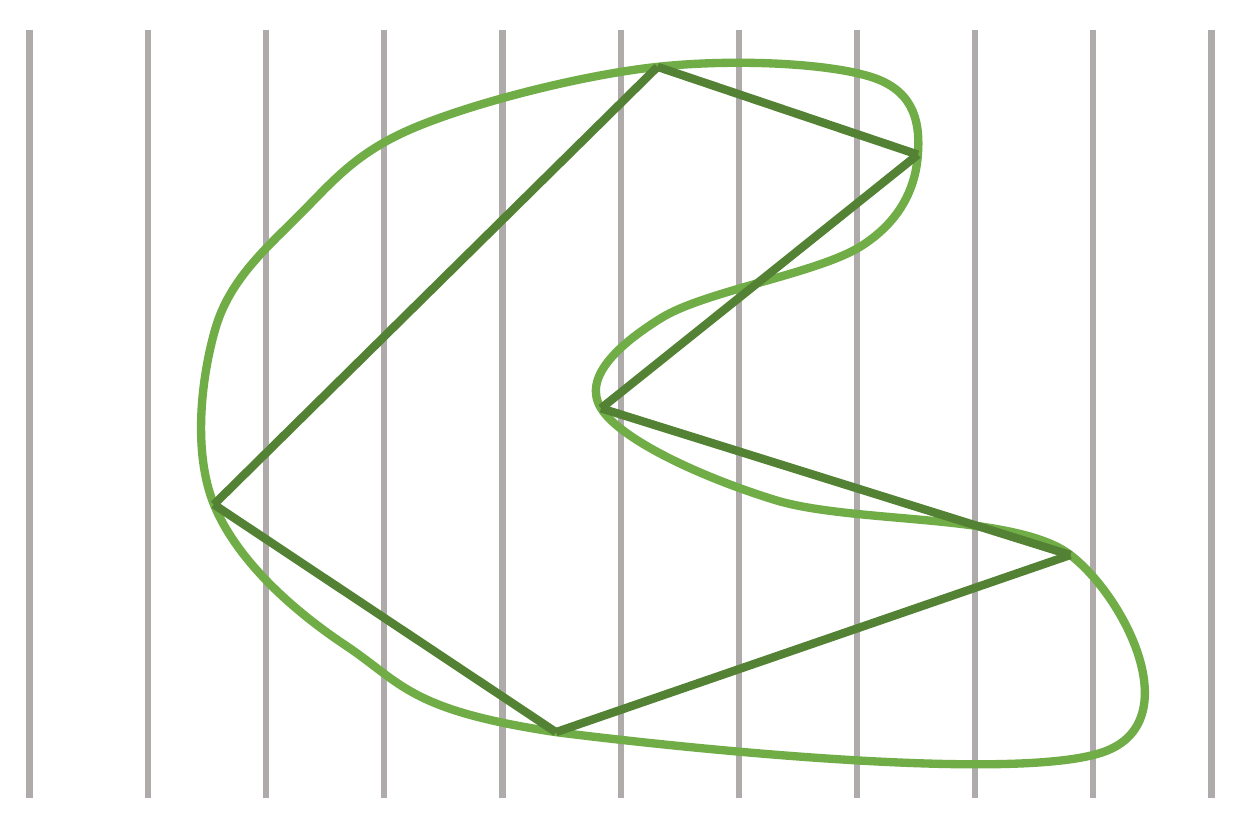}
    \end{subfigure}
    \hfill
    \begin{subfigure}[t]{0.3\linewidth}
        \centering
        \includegraphics[width=\linewidth,page=4,clip=true,trim = 2cm 0 2cm 0 ]{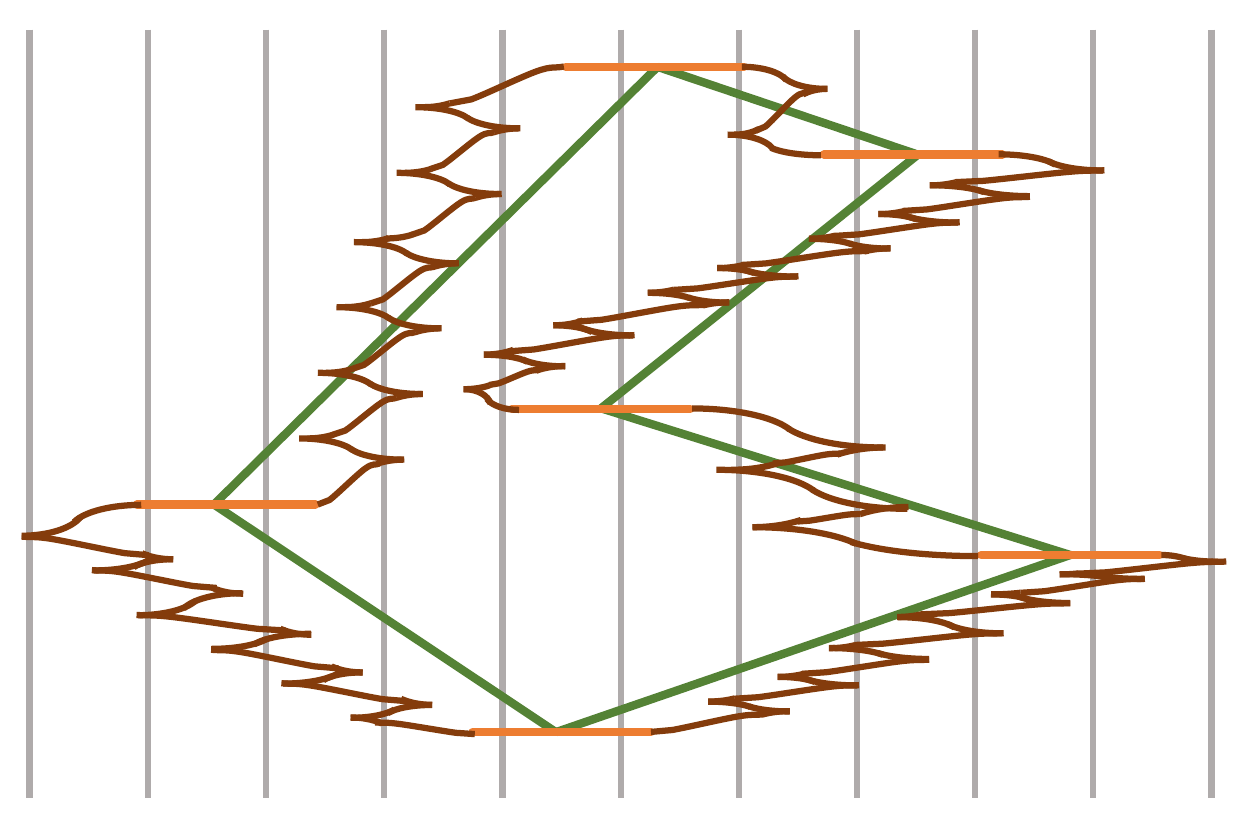}
    \end{subfigure}
	\centering
	\caption{On the left, the image of a map $F: S^i \to \SR$. The fibers of the bundle $\SR \to M$ run vertically. In the middle, a piecewise embedding of $S^i$ into $\SR$ that is transverse to the fibers and a $C^0$-approximation of $F$. On the right, a ``holonomic approximation'' of $F$ by a map projected down from \'etale space. Along each top simplex we see wrinkling.} \label{fig:etale}
\end{figure}

\connectivityEtale
\begin{proof}
    We first prove surjectivity in rank $i \leq \dim(M)$. 

    Consider a map $F: S^i \rightarrow \SR$ representing a given homotopy class. We use Thurston's jiggling to find a piecewise embedding $G: S^i \to \SR$ with respect to a triangulation of $S^i$ such that $G$ is $C^0$-close to $F$ and transverse to the fibers of the base projection $p_b: \SR \rightarrow M$. The existence of $G$ is stated as \cref{lem:technicalJiggling} below. This makes all simplices graphical over $M$, so they can each be regarded as a formal solution over a simplex in $M$. That is, given a simplex $\Delta \subset S^i$ we consider its image $G(\Delta)$. It is graphical over the simplex $\sigma \coloneqq p_b (G(\Delta)) \subset M$ and is thus the image of a formal section $H: \sigma \rightarrow \SR$. Inductively on the dimension of the simplices, and simplex per simplex, we then argue as follows:

    We start with the $0$-simplices. Let $\Delta$ be a $0$-simplex in $S^i$ and let $H: \sigma \to \SR$ be the corresponding formal solution where $\sigma = p_b(G(\Delta))$ is a $0$-dimensional simplex in $M$. We choose a holonomic extension $\tilde H_\sigma$ of $H$ over a neighborhood $V$ of the point $\sigma$. Next we ``flatten out'' the map $G$ over the simplices adjacent to $\Delta$. That is, we homotope $G$ over these simplices such that the image of a neighborhood $A$ of $\Delta$ under the resulting map is contained in $\tilde H_\sigma(V)$. We can assume that $A$ contains $\sigma$ as the only $0$-simplex. We repeat this for all $0$-simplices and obtain a piecewise smooth map $G^{(0)}: S^i\to \SR$ homotopic to $G$ and hence $F$, such that over a neighborhood of the $0$-skeleton of $S^i$, it is smooth and lifts to $\EtSol{M}$.

    We now let $k\leq m-1$ and assume that we have constructed a map $G^{(k-1)}: S^i \to \SR$ which, over a neighborhood of the $(k-1)$-skeleton in $S^i$, is smooth and lifts to $\EtSol{M}$.
    
    Let $\Delta$ be a $k$-simplex in $S^i$ and let $\sigma \coloneq p_b(G^{(k-1)}(\Delta))$ be the corresponding $k$-simplex in $M$. We denote by $H$ the section $\sigma \to \SR$. We extend $H$ to a section $H' : U \to \SR$ defined on a neighborhood $U$ of $\sigma$. Since $G^{(k-1)}$ lifts to $\EtSol{M}$ over a neighborhood $U_{\partial \Delta}$ of $\partial \Delta$, we can make sure that the image of $H'$ over $U_\partial \coloneqq p_b(G^{(k-1)}(U_{\partial \Delta}))$ is contained in the image of $G^{(k-1)}(U_{\partial \Delta})$. In particular this ensures that $H'$ is holonomic over $U_\partial$.
    
    We now apply holonomic approximation to $H'$ over $\sigma$ relative to $\partial \sigma$. This produces a $C^0$-small diffeomorphism $\psi: U \to U$ which fixes a neighborhood $V_\partial \subset U_\partial$ of $\partial\sigma$ and is $C^0$-close to the identity. Additionally, it produces a neighborhood $V\subset U$ of $\psi(\sigma)$, and a holonomic section $\tilde H : V \to \SR$ that is $C^0$-close to $H$ and agrees with $H$ over $V_\partial$. 

    Hence differently put, holonomic approximation produces an isotopy $B: [0,1]\times \Delta \to \SR$ starting at $G^{(k-1)}|_{\Delta}$ and ending at the embedding with image $\tilde H(\psi(\sigma))$. Using the isotopy extension theorem, we obtain an isotopy defined in a neighborhood $W$ of $\Delta$, which is the identity outside of a compact subset containing $W$ and ends at an embedding with image $\tilde H(V)$. Since $\phi$ and and $\tilde H$ are $C^0$-close to respectively the identity and $H'$, we can assume that $ W$ does not intersect other simplices of dimension $\leq k$, other than in the neighborhood $U_{\partial U}$ of $\partial \Delta$, where $G^{(k-1)}$ already lifts to $\EtSol{M}$ and where the isotopy $B$ is hence constant over $[0,1]$. Moreover, since the isotopy on $\Delta$ is fibered over $M$, we can assume the isotopy of $ W$ to also be fibered over $M$. In particular the isotopy preserves the transversality of the simplices to the fibers of $p_b : \SR \to M$.
    
    We now again ``flatten out'' the map $G^{(k-1)}$ over the simplices in $S^i$ adjacent to $\Delta$. Hence we homotope $G^{(k-1)}$ over these simplices such that the image of a neighborhood $A$ of $\Delta$ under the resulting map is contained in $\tilde H(V)$. 
    As a result, we obtain a map $G^{(k-1)}_\Delta : S^i \to \SR$ which is smooth and lifts to $\EtSol{M}$ over a neighborhood of the $(k-1)$-skeleton and over a neighborhood of $\Delta$. 

    We repeat this for all $k$-simplices in $S^i$, and obtain the map $G^{(k)}: S^k \to M$ which lifts to $\EtSol{M}$ over a neighborhood of the $k$-skeleton. If $i<m$, this finishes the proof of the surjectivity in rank $i$.
    
    If instead $i=m$, we proceed as follows. Let $\Delta$ be an $m$-simplex in $S^m$ and let $H: \sigma \to \SR$ be the corresponding formal solution defined over an $m$-simplex $\sigma$ in $M$. We extend $H$ to $H' : U \to \SR$ as for the lower dimensional simplices. Since $\sigma$ is of the same dimension as $m$, we cannot apply holonomic approximation. Instead, we apply \cref{thm:wrinklingEtale} relative to $\partial \sigma$ which introduces wrinkles in the top-dimensional cells. As in lower dimensions, we apply the isotopy extension theorem to obtain a map $G^{(m-1)}_\Delta: S^i \to \SR$ that is smooth and lifts to $\EtSol{M}$ over a neighborhood of the $(m-1)$-skeleton and over a neighborhood of $\Delta$.
    
    We repeat this for every $m$-simplex $\Delta$. As a result we obtain a map $G^{(m)}: S^m \to \SR$ that is homotopic to $F$ and lifts to $\EtSol{M}$. This finishes the proof of the surjectivity in rank $i \leq m$.

    For injectivity we argue using a map $F: D^{i+1} \rightarrow \SR$ together with a lift to $\EtSol{M}$ along the boundary. By a small perturbation we can assume that this lift extends to a collar. We then argue as above relative to this collar.
\end{proof}

It remains to prove the following auxiliary technical lemma, which we used in the proof above. Here we denote by $|K|\subset \R^N$ the linear polyhedron associated to a simplicial complex $K$ consisting of linear simplices in $\R^N$. That is, $|K|$ is the union of the simplices of $K$. 

\begin{lemma} \label{lem:technicalJiggling}
Let $M$ and $N$ be manifolds of dimension $m$ and $n$ respectively, where $m \leq N$. Let $T: |K| \to M$ be a triangulation of $M$ and let $N$ be endowed with a distribution $\xi$. Consider a map $F: M \rightarrow N$ that is piecewise smooth with respect to $T$. Then, there exists a subdivision $T':|K'| \to M$ of $T$ and a map $F': M \rightarrow N$ such that:
\begin{itemize}
    \item $F'$ is piecewise smooth with respect to $T'$,
    \item $F'$ is $C^1$-close to $F$,
    \item $F'$ is a piecewise embedding, and
    \item $F'$ is piecewise transverse to $\xi$.
\end{itemize}
\end{lemma}
\begin{proof} 
    The requested pair $(F',T')$ can be obtained by applying an improved version \cite{FPTjiggling} of Thurston's jiggling \cite{Th2} to the pair $(F,T)$. Jiggling subdivides $T$ and perturbs $F$, such that $F'$ is piecewise transverse to $\xi$ and $F'$ is a piecewise immersion, both with respect to $T'$. Since $T$ is subdivided to $T'$ in a way that the diameter of the simplices in the triangulation decreases uniformly, we can assume that the simplices of $T$ are small enough such that $F'$ is in fact a piecewise embedding. 
\end{proof}

\subsection{Folded symplectic structures} \label{ssec:folded}
Our results achieve a theorem similar to the one of Cannas da Silva~\cite{CdS} on folded symplectic structures. Here a \textbf{folded symplectic} structure on an $m$-manifold $M$ is a closed $2$-form $\omega$ that is non-degenerate, except for on a hypersurface $V\subset M$ around which it locally agrees with the pullback by a fold of the standard symplectic structure $\omega_{\mathrm{std}}$ on $\R^m$. Analogously, we define the notion of a double folded symplectic structure. Then the formal regularization below is the result of pulling back the standard symplectic structure $\omega_{\mathrm{std}}$ via the formal regularization of the double fold (recall \cref{sec:DoubleFold}). 
\begin{proposition} \label{pro:foldedSympl}
    Let $M$ be an even dimensional manifold with a formal symplectic structure $\omega$. Then there exists a (double) folded symplectic structure $\omega'$ such that its formal regularization is homotopic to $\omega$ via formal symplectic structures. If $\omega$ is symplectic in an open neighborhood of a closed subset $M'\subset M$, it can be formally regularized relative to $M'$.
\end{proposition}
\begin{proof}
    Since the symplectic relation $\SRsympl$ is $\Diff$-invariant, we recall from \cref{sec:tautsol} that the \'etale space $\EtSol[\SRsympl]{M}$ of symplectic structures over $M$ has a canonical symplectic structure $\omega_{\rm{can}}$. Moreover, \cref{cor:wrinklingEtale} produces a wrinkled submersion $G: M \rightarrow \EtSol[\SRsympl]{M}$ such that $p_1 \circ G$ is homotopic to $\omega$ via formal symplectic structures. It follows from \cref{cor:wrinklingEtaleParametric} that this homotopy is relative to $M'$. Using surgery of singularities (see \cite[Section 2.10]{ElMiWrinEmb} or \cite[Appendix B]{PT1}), we replace each wrinkle in $G$ by a collection of double folds, which replaces $G$ by a map $H$ that is instead a folded submersion. We can make sure that the formal regularizations of $p_1 \circ G$ and $p_1 \circ H$ are homotopic via formal symplectic structures relative to $M'$. The pullback $H^*\omega_{\rm{can}}$ is then the requested folded symplectic structure on $M$.
\end{proof}
We note that since \cref{cor:wrinklingEtale} also holds parametrically (\cref{cor:wrinklingEtaleParametric}), a parametric version of \cref{pro:foldedSympl} follows as well. For this we need to allow double folds to cancel in an embryo event as the parameter varies.

\begin{remark}
The key point behind this result is that differential forms can be pulled back by maps and, in the symplectic case, one can provide a (local) normal form for said pullback as long as the maps under study have controlled singularities (like folds).

For general geometries, is this much more subtle. Indeed, in general there will be many models for the singularities of the pullback, depending on the relative positions of the geometric structure and the map. It is then meaningful to study whether one may be able to homotope the map and the structure further to reduce the possible models to a controlled collection. 

This idea enters crucially the h-principle for higher-dimensional contact structures. In \cite{BEM} they use wrinkling to prove that any formal contact structure is homotopic to a contact structure with singularities isomorphic to a ``universal hole''. These are then removed using an overtwisted disc. The h-principles for overtwisted Engel structures \cite{CPPP,CPP19,PV} use similarly the fact that one can produce Engel structures whose singularities are controlled (but it is not known whether a unique ``universal hole'' exists in the Engel setting).

In this paper we consider arbitrary $\SR$ and do not pursue the idea of adapting our wrinkles further to the geometric structures under consideration (since this is highly dependent on the nature of $\SR$).
\end{remark}

\subsection{Horizontal homotopy groups} \label{ssec:horizontal}

A topic that has received attention in the last few years \cite{HajSchTy,WenYoung,Perry19,Perry21} is the study of the \emph{horizontal homotopy groups} $\pi_i^H$. These consist of homotopy classes of tangent maps into a manifold $N$ endowed with a distribution $\xi$, up to tangent homotopy. As we remarked in the introduction, maps into $\EtSol{M}$ project down via $p_r$ to maps tangent to the Cartan distribution $\xi_\can$ in $\SR \subset J^r\Psi$. Hence \cref{prop:connectivityEtale} recovers the following statement:
\begin{corollary} \label{cor:connectivityHorizontal}
The inclusion in homotopy groups $\pi_i^H(\SR,\xi_\can) \rightarrow \pi_i(\SR)$ is surjective in degrees $i \leq \dim M$. 
\end{corollary}
Observe that injectivity does not follow from \cref{prop:connectivityEtale}, since not every horizontal map can be lifted to a map to \'etale space. 

Related to the study of horizontal homotopy groups are Lipschitz homotopy groups $\pi_i^L$, as they both arise in the study of sub-Riemannian manifolds \cite{DHLTLipschitz}. These Lipschitz homotopy groups are defined similarly to ordinary homotopy groups, but the maps (and homotopies) are required to be Lipschitz continuous. Lipschitz and ordinary homotopy groups coincide for Riemannian manifolds, but this is not the case for more general metric spaces, such as sub-Riemannian manifolds equipped with a Carnot-Carath\'eodory metric.

Reverting our attention back to \cref{cor:connectivityHorizontal}, we note that the maps into $\SR$ that are projected down from $\EtSol{M}$ are Lipschitz for any Carnot-Carath\'eodory metric on $\xi_\can$. Hence we also obtain a surjection if we take the domain to be the Lipschitz homotopy group $\pi_i^L(\SR,\xi_\can)$.

As far as the authors are aware, a proof of \Cref{cor:connectivityHorizontal} was first sketched by Thom \cite{Thom1}, though it was phrased differently. The proof we have just presented suggests an interesting open question.
\begin{question}
Observe that $p_r^{r'}: J^{r'}\Psi \rightarrow J^r\Psi$ with $r'>r$ takes horizontal maps to horizontal maps and therefore defines the following morphism of horizontal homotopy groups 
\[ \pi_i^H(J^{r'}\Psi,\xi_\can) \rightarrow \pi_i^H(J^r\Psi,\xi_\can). \]
What can be said about its connectivity (particularly in degree $\dim M$)? An analogous question can be posed for Lipschitz homotopy groups.
\end{question}
\section{Preliminaries: groupoids, principal bundles and microbundles} \label{sec:prelimGroupoids}
Our next goal is to establish the results related to microbundles as stated in \cref{sec:IntroRMb}. However, before we actually start with this, we will discuss the necessary language. We start by discussing groupoids and their principal bundles in \cref{ssec:PGBundles}. In \cref{sec:prelim2Haefliger} we discuss how principal groupoid bundles relate to microbundles when the groupoid is effective, and we summarize the correspondences in \cref{ssec:summaryBundles}. We introduce the classifying space of principal groupoid bundles in \cref{sec:prelim2classspace}. A reader already familiar with all of this may skip this section. It is in \cref{sec:Rgroupoids} that we introduce geometries, and in particular $\GammaR$, to the story.

For more details we refer the reader to Haefliger's original articles \cite{Haef1,Haef2,Haef3}, Thurston's work on the flexibility of foliations and its relation to diffeomorphism groups \cite{Th3,Th2,Th1}, as well as more modern references in Lie groupoid theory \cite{MoerMrcun,CraFer,Carchedi,AccCra} and h-principles for foliations \cite{ElMiWrinIII,Mei}.

\subsection{Groupoids} \label{ssec:PGBundles}
A succinct definition of a \textbf{groupoid} is as a category in which every arrow is invertible. We however prefer to think of a groupoid as consisting of a space of arrows $\SG$ and a base space $X$ with the following structure maps: source and target, multiplication, inversion and unit. All together we denote such a groupoid as $\SG \rightrightarrows X$. We write $\source$ and $\target$ for respectively the source and target maps of the upcoming groupoids and write morphisms as $g: \source(g) \to \target(g)$.

We consider either topological or Lie groupoids. For \textbf{topological groupoids} we require the sets $\SG$ and $X$ to be topological spaces and all the maps from the definition to be continuous. For \textbf{Lie groupoids}, we require the sets $\SG$ and $X$ to be smooth manifolds, although we allow them to be non-Hausdorff and not second-countable.  Additionally we require that all structure maps are smooth, and the source and target to be smooth submersions.

\begin{remark}
    Even though we do not require the base space $X$ of a Lie groupoid $\SG\rrarrows X$ to be Hausdorff and second-countable, this will be our default. Similarly, we reserve the term manifold for topological spaces that are locally Euclidean, Hausdorff and second-countable.  If we deviate from this, we mention so explicitly. As a guideline, the only manifolds we encounter that are not Hausdorff or second-countable are the \'etale spaces $\EtSol{M}$ where $M$ is a (Hausdorff and second-countable) manifold. As a consequence, groupoids over $\EtSol{M}$ will be the only groupoids with a base manifold that is not Hausdorff or second-countable. 
\end{remark}

We say that a Lie groupoid is \textbf{\'etale} if the source and target maps are local diffeomorphisms.  We refer to a Lie groupoid $\SG$ as Hausdorff and/or second-countable if both its arrow and base space are.
\begin{remark} \label{rem:groupoidTopDef}
    In the remainder of this section we focus on Lie groupoids, and their related smooth notions, since they are our main object of study. However, all of the following definitions can be generalized to the topological category by replacing the smooth notions (eg. smooth manifold, diffeomorphism, smooth map) with their respective topological counterparts (eg. topological space, homeomorphism, continuous map).
\end{remark}

A \textbf{groupoid morphism} between Lie groupoids $\SG \rrarrows X$ and $\SH \rrarrows Y$ is then a functor, when we interpret both groupoids as categories. More explicitly, a groupoid morphism is given by two smooth maps $F: \SG \to \SH$ and $f: X \to Y$ such that the diagram
\begin{equation*}
        \begin{tikzcd}
            \SG \arrow[d,shift left] \arrow[d,shift right]  \arrow[r,"F"] &  \SH  \arrow[d,shift left] \arrow[d,shift right]  \\
            X \arrow[r,"f"] & Y.
        \end{tikzcd}
    \end{equation*}
commutes. A groupoid morphism $(F,f)$ as above is a \textbf{groupoid isomorphism} if there exists a morphism $(G,g)$ from $\SH$ to $\SG$ such that $(F\circ G, f \circ g)$ and $(G \circ F, g \circ f)$ are the identity maps. 

A (local/global) section $\sigma$ of $\source: \SG \rrarrows X$ is called a \textbf{(local/global) bisection} if $\target \circ \sigma$ is a diffeomorphism.

\subsubsection{The groupoid \texorpdfstring{$\Gamma$}{Γ} of local diffeomorphisms} \label{sec:Gamma}
A classical example of a groupoid arises from the following, which is in general just a space.
\begin{definition}\label{def:GammaMN}
    Let $M$ and $N$ be smooth manifolds. We denote the space of germs of locally defined diffeomorphisms $M \to N$ by $\Gamma(M,N)$ and endow it with the \'etale topology. 
\end{definition}
In the case where $M=N$ we obtain the groupoid classically denoted by $\Gamma$. It has been an object for study, for reasons we will encounter later (as for instance in \cref{sec:folmb,rem:ClassFramedPB}).

\begin{definition} \label{def:Gamma}
    Let $M$ be a smooth manifold. The Lie groupoid of germs of locally defined diffeomorphisms $\Gamma(M,M)$ is denoted by $\GammaM{M} \rightrightarrows M$. In particular we write $\Gamma^n \rightrightarrows \R^n$ for $\Gamma^{\R^n} \rightrightarrows \R^n$.
\end{definition}
In the above definition the base $M$ is a smooth manifold in the usual sense (that is, $M$ is Hausdorff and second-countable), and the space of arrows is endowed with the unique smooth structure turning the source and target maps into \'etale maps. This makes the groupoid into an \'etale groupoid. In particular, the source and target fibers carry the discrete topology.

We discuss two variations of the above spaces: Firstly, we can endow $\Gamma(M,N)$ and $\Gamma^M$ with the Whitney topology instead of the \'etale topology (recall \cref{sec:EtSolWhitneyTop}). To indicate this, we respectively write $\Gamma^{n,\Whitney}(M,N)$ and $\Gamma^{M,\Whitney}$. We do not endow them with a compatible smooth structure, so that $\Gamma^{n,\Whitney}(M,N)$ and $\Gamma^{M,\Whitney}$ are just topological objects to us.\footnote{We imagine that if we want to endow $\Gamma^{n,\Whitney}(M,N)$ and $\Gamma^{M,\Whitney}$ with a smooth structure compatible with the Whitney topology, we would have to model these spaces on an infinite dimensional vector space. This would turn $\Gamma^{n,\Whitney}(M,N)$ and $\Gamma^{M,\Whitney}$ into respectively an infinite dimensional manifold and groupoid, and hence we consider them as just topological objects.}

Secondly, by taking $r$-jets of the germs in $\Gamma(M,N)$ and $\Gamma^M$, we obtain the space $J^r\Gamma(M,N)$ and the groupoid $J^r\Gamma^M$. In this case $J^r\Gamma(M,N)$ and $J^r\Gamma^M$ are canonically endowed with the Whitney topology. When $r<\infty$, this makes the groupoid $J^r\Gamma^M$ a finite dimensional, Hausdorff and second-countable Lie groupoid, which is not \'etale.

\subsubsection{Pseudogroups} \label{sec:pseudogroups}
 We recall that in \cref{def:pseudogroup} we defined the notion of a pseudogroup $G$ over a smooth manifold $M$. To any such $G$ we now associate the following groupoid: 
\[ \Gamma(G) = \{[f]_x \mid f\in G \text{ and } x\in M \text{ in the domain of definition of } f  \}. \]
Vice versa, we associate to every \'etale groupoid $\SG$ the pseudogroup 
\[ \gtop(\SG) = \{ \target \circ \sigma \mid \sigma \text{ is a local bisection of } G \}. \]
We note that $\gtop(\Gamma(P))=P$ for any pseudogroup $P$, whereas $\Gamma(\gtop(\SG))$ produces what is known as the \textbf{effect} $\Eff(\SG)$ of a groupoid $\SG$.
\begin{definition} \label{def:effective}
    Let $\SG \rightrightarrows X$ be an \'etale Lie groupoid. We define the Lie groupoid homomorphism $\Eff$ from $ \SG \rightrightarrows X$ to $\GammaM{X} \rightrightarrows X$ as the identity on $X$ and as $\Eff(g) = [\target (\source|_U)^{-1}]_{\source (g)}$ on $g \in \SG$, where $U \subset \SG$ is an open neighborhood of $g$ such that $\source|_U$ and $\target|_U$ are diffeomorphisms. We say that an \'etale Lie groupoid $\SG$ is \textbf{effective} if $\Eff : \SG \to \Gamma^X$ is injective.
\end{definition}
As an example, the groupoid $\Gamma^M$ is effective for any manifold $M$ and is in particular mapped to itself via $\Eff$. Moreover, the pseudogroup associated to $\Gamma^M$ is $\Diff_\loc(M)$, and the groupoid associated to $\Diff_\loc(M)$ is $\Gamma^M$.

By the discussion above we obtain the following result, usually attributed to Haefliger~\cite{Haef4}.
\begin{lemma}
    There is a one-to-one correspondence between pseudogroups on a manifold $M$ and isomorphism classes of effective groupoids over $M$.
\end{lemma}

\subsubsection{Groupoid actions}
Similar to groups acting on a space, one can also consider groupoids acting on a space. However, the difference is that groupoids act along a so-called moment map, encoding when the action is defined.
\begin{definition}
    Let $\SG \rightrightarrows X $ be a Lie groupoid. A \textbf{(left) action} of $\SG$ on $P$ along a map $\mu : P \to X$ is a map $\SG \timeslr{\source}{\mu} P \to P$ denoted by $(g,p) \mapsto g \cdot p = gp$ such that for all $g,h \in \SG$ and $p \in P$ we have
    \begin{itemize}
        \item $h \cdot (g \cdot p)=(hg) \cdot p$ whenever both sides are defined, and
        \item $1_x \cdot p  = p$ where $\mu(p) = x \in X$.
    \end{itemize}
    This makes $P$ into a \textbf{(left) $\SG$-space} with \textbf{moment map} $\mu$.
\end{definition}
\begin{example} \label{ex:actionGroupoid}
    The (left) action of a groupoid $\SG$ on a space $P$ gives rise to the \textbf{action groupoid} $\SG \timeslr{\source}{\mu} P \rrarrows P$ with source map $(g,p) \mapsto p$ and target map $(g, p)\mapsto g\cdot p$. The composition of arrows is given by $(h,g \cdot x)\cdot(g,x) = (hg,x)$.
\end{example}

\subsubsection{Principal groupoid bundles}
When a Lie groupoid action on a space $P$ is free and proper (encoded together in the following definition as \cref{it:freeProper}), the quotient $P/\SG$ has a canonical smooth structure such that the canonical quotient map $P \to P/\SG$ is a smooth submersion. This observation gives rise to the notion of principal groupoid bundles. 
\begin{definition} \label{def:PGB}
    Let $\SG \rightrightarrows X$ be a Lie groupoid and let $M$ be a manifold. A \textbf{(left) principal $\SG$-bundle} over $M$ is a (left) $\SG$-space $P$ with a moment map $\mu: P \to X$ and a surjective submersion $\pi: P \to M$ such that
    \begin{enumerate}[label=(\arabic*)]
        \item $\pi$ is $\SG$-invariant, and
        \item\label{it:freeProper} the map $\SG \timeslr{\source}{\mu} P \to P \times_M P$ defined by $(g,p) \mapsto (p,g \cdot p)$ is a diffeomorphism.
    \end{enumerate}
    We summarize the information of a principal $\SG$-bundle in a diagram as follows:
    \begin{equation*}
        \begin{tikzcd}
            \SG \arrow[d,shift left] \arrow[d,shift right] & P \arrow[loop left, distance=3em, in =195, out = 165] \arrow[d,"\pi"] \arrow[dl,"\mu"]     \\
            X & M.
        \end{tikzcd}
    \end{equation*}
\end{definition}
Two of the easiest examples of principal groupoid bundles are the following. 
\begin{example}\label{ex:unitPB}
    Let $\SG \rrarrows X$ be a groupoid, then $X$ is canonically endowed with the \textbf{unit principal $\SG$-bundle}. This is the bundle $\pi = \source : \SG \to X$ with moment map $\mu = \target :\SG \to X$. The left action of $\SG$ on $\SG$ is just multiplication in $\SG$. The reason for the name is that, when we define this bundle by a cocycle (see \cref{pro:corCocyclePrBundle}), this cocycle consists of a single map $\gamma: X\to\SG$ given by $x \mapsto \id_x$.
\end{example}
\begin{example} \label{ex:trivPb}
    Let $\SG \rrarrows X$ be a groupoid and $f: M \to X$ a smooth map between manifolds. We define a principal groupoid as \textbf{trivial} if it is the pullback of the unit principal $\SG$-bundle from \cref{ex:unitPB}. In particular, we define the \textbf{trivial $\SG$-bundle} $P_f$ over $M$ associated to $f$ as
    \[ P_f = \{(m,g) \in M \times \SG \mid f(m) = \source(g) \} \]
    with the obvious projection $P_f \to M$, the moment map $\mu:P
    _f\to X$ defined by $(m,g) \mapsto \target(g)$ and the left action defined by $h\cdot (m,g) = (m,hg)$ whenever $\source (h) = \target(g)$. 
\end{example}

The natural notion of morphism for principal $\SG$-bundles is the following.
\begin{definition}
    A \textbf{bundle morphism} between principal $\SG$-bundles $P$ and $P'$ over a manifold $M$ with respective moment maps $\mu : P \to X$ and $\mu':P'\to X$, is a smooth map $F: P\to P'$ such that
    \begin{itemize}
        \item $\mu' \circ F = \mu$ as maps $P \to X$, and
        \item $F(g\cdot p) = g \cdot F(p) $ for all $p \in P$ and $g \in \SG$ whenever $g \cdot p$ is defined.
    \end{itemize}
    If there exists additionally a bundle morphism $G: P'\to P$ such that $F\circ G$ and $G \circ F$ are the identity maps, the map $F$ is a \textbf{bundle isomorphism}. We denote the \textbf{set of isomorphism classes of principal $\SG$-bundles} on $M$ by $\pb{M}{\SG}$.
\end{definition}

\subsubsection{Cocycles}

Each principal $\SG$-bundle is locally trivial~\cite[Remarks 5.34]{MoerMrcun}: given a principal $\SG$-bundle $\pi : P \to M$ we know that $\pi$ is a surjective submersion and hence admits local sections $\sigma: U \to P$. The composition $\mu \circ \sigma$ is then a map $U \to X$ and hence we obtain the trivial bundle $P_{\mu\circ\sigma}$. We observe that  $P_{\mu\circ\sigma}$ and $P|_U$ are isomorphic via the map $P_{\mu\circ\sigma} \to P|_U$ defined by $(g,x) \mapsto g \cdot \sigma(x) $.

If we now consider the transition maps between the $P_{\mu\circ\sigma}$ associated to different local sections $\sigma$ of $P$, we obtain a cocycle as in the definition below. The reader should keep in mind that the $\gamma_{ii}$ from the definition below correspond to the maps $\mu \circ \sigma$. For this, we note that the maps $\gamma_{ii}: U_i \to \SG$ in a cocycle map to identity elements in $\SG$, and hence we can interpret the $\gamma_{ii}$ as maps to the base $X$ of $\SG$.

\begin{definition}\label{def:PGBCocyle}
    Let $\SG \rightrightarrows X$ be a Lie groupoid and let $M$ be a manifold with an open covering $\{U_i\}_{i \in I}$. A \textbf{(smooth) $\SG$-cocycle} over $M$ is a collection $\{ \gamma_{ij} : U_i \cap U_j \to \SG \}_{i,j\in I}$ of smooth maps such that 
    \[
        \gamma_{ik}(x) = \gamma_{ij}(x)\gamma_{jk}(x)
    \]
    for all $x \in U_i \cap U_i \cap U_k$ and $i,j,k \in I$. 

    Two cocycles are \textbf{equivalent} if their union can be extended to a cocycle over a refinement of the associated opens. That is, two cocycles $\{\gamma_{ij}: U_i \cap U_j \to \SG \}_{i,j \in I}$ and $\{ \delta_{kl}: V_k \cap V_l \to \SG\}_{k,l\in K} $ are equivalent if there exist maps $\{\eta_{ik}: U_i \cap V_k \to \SG\}_{i \in I,k \in K}$ such that 
    \begin{align*}
        \eta_{ik} \delta_{kl} = \eta_{il} &\text{ on }  U_i \cap V_k \cap V_\ell\text{, and} \\
        \gamma_{ji} \eta_{ik} = \eta_{jk} &\text{ on }  U_i \cap U_j \cap V_k
    \end{align*} 
    for all $i,j \in I $ and $k,l\in K$.
\end{definition}

We obtain the following correspondence between cocycles and principal bundles, similar to the case of groups.
\begin{proposition}[{\cite[Proposition I.3.1]{Mrcun}}] \label{pro:corCocyclePrBundle}
    Let $M$ be a manifold. There is a one-to-one correspondence between $\SG$-cocycles over $M$ and principal $\SG$-bundles over $M$, up to respectively equivalence and isomorphism.
\end{proposition}
\begin{proof}
    Given a $\SG$-cocycle $\{ \gamma_{ij} : U_i \cap U_j \to \SG\}_{i,j \in I}$ over $M$ we define the associated principal $\SG$-bundle $P$ over $M$ as $P = \sqcup_i P_{\gamma_{ii}} / \sim$, where we write $(i,x,g) \sim (j,y,h)$ if and only if $x=y$ and $h = \gamma_{ij}(x) \cdot g$. We define the projection $\pi : P \to M$ as induced by the map $(i,x,g) \mapsto x$, the moment map $\mu: P \to X$ as induced by $(i,x,g) \mapsto \target(g)$ and the action as induced by $h \cdot (i,x,g) \mapsto (i,x,hg)$. We leave the remaining details to the reader (or \cite[Proposition I.3.1]{Mrcun}).
\end{proof}

\begin{remark}\label{rem:SubmCocPb}
    In the particular case of an \'etale groupoid $\SG\rrarrows X$, we note that if the $\gamma_{ii}$ in a cocycle are submersions, the moment map of the associated principle bundle is a submersion. Vice versa, if the moment map $\mu$ of a principle $\SG$-bundle is a submersion, we obtain that the associated $\gamma_{ii}$ are submersions. Hence when $\SG$ is \'etale, the above correspondence specifies to a correspondence between $\SG$-cocycles with submersive $\gamma_{ii}$ and principal $\SG$-bundles with submersive moment map.
\end{remark}

\subsubsection{Morita equivalence}
Groupoids $\SG \rrarrows X$ can be interpreted as smooth objects representing their orbit space $X/\SG$, which is in general a singular space. An example is the holonomy groupoid associated to a foliation, which can be interpreted as representing the leaf space of the foliation. More generally, a Lie groupoid can be used to encode a transverse geometry. 

The groupoid encoding such a transverse geometry does not need to be unique. Going back to the example of the holonomy groupoid, this transverse geometry can also be encoded in the holonomy groupoid restricted to a complete transversal of the foliation. This ambiguity in the choice of groupoid is captured by the notion of Morita equivalence. This is discussed in more detail in~\cite{hoyo2013lie}.

\begin{definition}\label{def:MoritaEquiv}
    Two groupoids $\SG \rightrightarrows X$ and $\SG' \rightrightarrows X'$ are \textbf{Morita equivalent} if there exists a principal $\SG$-$\SG'$-bibundle $P$ with moment maps $\mu: P \to X'$ and $\mu': P \to X$. That is, we require that:
    \begin{itemize}
        \item $\mu: P \to X$ is a (right) principal $\SG'$-bundle with moment map $\mu': P \to X'$,
        \item $\mu': P \to X'$ is a (left) principal $\SG$-bundle with moment map $\mu: P \to X$, and
        \item the actions of $\SG$ and $\SG'$ on $P$ commute.
    \end{itemize}
    We summarize the information of a Morita equivalence in a diagram as
    \[
    	\begin{tikzcd}
	   \SG  \arrow[d, shift right] \arrow[d, shift left]  &  P \arrow[loop left, distance=3em, in =165, out = 195] \arrow[loop right, distance=3em, in =15, out = -15] \arrow[dl,"\mu"] \arrow[dr,"\mu'" swap] & \SG'  \arrow[d, shift right]  \arrow[d, shift left]  \\
	   X &  & X'.
        \end{tikzcd}
    \]
\end{definition}

As an example of an Morita equivalence, we mention the following.
\begin{lemma} \label{lem:MEGamma}
    Let $M$ be an $m$-manifold. The groupoids $\Gamma^M \rrarrows M$ and $\Gamma^m \rrarrows \R^m$ are Morita equivalent.
\end{lemma}
\begin{proof}
    It is the space $\Gamma(M,\R^m)$ that forms the $\Gamma^m$-$\Gamma^M$-bibundle. The groupoid $\Gamma^m$ acts by post-composing, whereas the groupoid $\Gamma^M$ acts be pre-composing. We leave the further details to the reader.
\end{proof}

If $\SG \rrarrows X$ and $\SG' \rrarrows X'$ are Morita equivalent and $E$ is a right principal $\SG$-bundle over $M$, then together they induce a principal $\SG'$-bundle over $M$. To be precise, if $P$ is the $\SG$-$\SG'$-bibundle induced by the Morita equivalence, then $\SG$ acts diagonally on $E \times_X P$ and $(E \times_X P)/\SG$ is a principal $\SG'$-bundle over $M$. We note that if $E$ is the unit principal $\SG$-bundle, the resulting $\SG'$-bundle is just $P$. Up to isomorphism, we can reverse the construction, which implies the following.
\begin{proposition}[{\cite[Proposition II.1.7]{Mrcun}}] \label{pro:MePrBundlesSame}
    Let $\SG$ and $\SG'$ be two groupoids that are Morita equivalent. Then there is a one-to-one correspondence between principal $\SG$ and $\SG'$-bundles over a fixed manifold, up to isomorphism.
\end{proposition}

\subsubsection{\texorpdfstring{$\SG$}{G}-foliations}\label{sec:GFol}
Haefliger \cite{Haef1,Haef2} realized that a foliation $\SF$ of corank $n$ on a manifold $M$ is defined by submersions $s_i : U_i \to \R^n$ and transition functions $\gamma_{ij} : s_j(U_i \cap U_j) \to s_i(U_i \cap U_j)$, where the $U_i$ are opens in $M$ and $\gamma_{ij} \circ s_j = s_i$ on $U_i \cap U_j$. As a consequence, a foliation can be represented by a $\Gamma^n$-cocycle. 

As a generalization we define $\SG$-foliations, which are modeled on the base $X$ of $\SG$. 

\begin{definition}
    Let $\SG$ be an effective groupoid. A $\SG$-foliation on a manifold $M$ is a $\SG$-cocycle $\{\gamma_{ij} : U_i \cap U_j \to \SG \}_{i,j \in I}$, where all $\gamma_{ii}$ are submersions.
\end{definition}

We note that if $\SG = \Gamma^n$, a $\SG$-foliation reduces to an ordinary foliation. Moreover, using the Morita equivalence between $\Gamma^X$ and $\Gamma^n$ from \cref{lem:MEGamma}, we note that if the base manifold $X$ of $\SG$ is of dimension $n$, a $\SG$-foliation defines in particular a foliation of corank $n$.

Using \cref{rem:SubmCocPb}, we relate $\SG$-foliations to a specific class of principal $\SG$-bundles.
\begin{corollary}\label{cor:corGFolSubmPb}
    Let $M$ be a manifold and $\SG$ an effective groupoid. There is a one-to-one correspondence between $\SG$-foliations on $M$ and principal $\SG$-bundles over $M$ with submersive moment map, up to respectively equivalence and isomorphism. 
\end{corollary}

\subsection{Microbundles} \label{sec:prelim2Haefliger}
A microbundle is meant to be the germ of a bundle around a section and hence we speak of a \emph{representative} of a microbundle in the definition below. By a slight abuse of notation we will often identify $M$ with its image under the section $\iota: M \to E$.

\begin{definition} \label{def:MB}
    A \textbf{representative of a (smooth) microbundle} of rank $n$ over an $m$-manifold $M$ consists of 
    \begin{itemize}
        \item an $(m+n)$-manifold $E$,
        \item a smooth submersion $\pi: E \to M$, and 
        \item a smooth section $\iota: M \to E$ of $\pi$,
    \end{itemize}
    such that $E$ is locally trivial. Hereby we mean that for each $x \in M$ there exists a neighborhood $ U \subset M$ of $x$ and a neighborhood $V \subset E$ of $\iota(x)$ satisfying $\iota(U) \subset V$ such that there exists an embedding $\phi: V \to U \times \R^n$ making the following diagrams commute
    \[
    	\begin{tikzcd}
            V \arrow[d, "\pi"] \arrow[r,"\phi"] & U \times \R^n \arrow[d,"\pr_1"] & & V \arrow[r,"\phi"] & U \times \R^n \\
            M \arrow[r,hookleftarrow] &  U & & U \arrow[u,"\iota"] \arrow[r,hookrightarrow] & U \times \{0\} \arrow[u,hookrightarrow].
        \end{tikzcd}
    \]
    Two representatives $(E,\pi,\iota)$ and $(E',\pi',\iota')$ are \textbf{equivalent} if there exists an open neighborhood $W \subset E$ of $\iota(M)$ and a smooth embedding $f:W \to E'$ such that $f \circ \iota = \iota'$ and $\pi = \pi' \circ f$. A \textbf{microbundle} is an equivalence class of representatives.
\end{definition}

The example that Milnor \cite{Milnor64micro} originally used to motivate microbundles was the tangent bundle over manifolds without a differentiable structure.
\begin{example}[{\cite[Lemma 2.1]{Milnor64micro}}]\label{ex:MBtangentbundle}
    Let $M$ be a topological manifold. The topological microbundle $\tau M = (M\times M, \pr_1,\Delta)$ where $\Delta: M \to M\times M$ is the diagonal map, is known as the \textbf{tangent bundle}. If $M$ is a smooth paracompact manifold, the microbundle $\tau M$ is equivalent, as a smooth microbundle, to the tangent bundle $TM \to M$. 
\end{example}

\subsubsection{Foliated microbundles}\label{sec:folmb}

As mentioned in \cref{sec:GFol}, Haefliger realized that foliations of a manifold can be described by a $\Gamma^n$-cocycle where the $\gamma_{ii}$ are submersive. He noticed furthermore that if we drop the submersivity condition on the $\gamma_{ii}$ from a $\Gamma^n$-cocycle, we still obtain a foliation, albeit a singular one. Nowadays such singular foliations are referred to as Haefliger structures or Haefliger microbundles, when represented as foliated microbundles. 
We obtain Haefliger's correspondence between $\Gamma^n$-cocycles and foliated microbundles in \cref{cor:corPGammaNB-FolMB}, as a special case of a more general correspondence between cocycles and microbundles which we discuss in \cref{sec:CorPbMb}. However, we first introduce foliated microbundles and some more general theory.

\begin{definition} \label{def:FolMb}
    A \textbf{representative of a (smooth) foliated microbundle} is a microbundle $E$ endowed with a smooth foliation $\SF$ that is transverse to the fibers of $\pi$ and satisfies that $\corank(\SF)= \rank(E)$. Two smooth foliated microbundles $(E,\pi,\iota, \SF)$ and $(E',\pi',\iota', \SF')$ are \textbf{equivalent} if there exists an equivalence of microbundles that maps leaves of $\SF$ to leaves of $\SF'$. A \textbf{foliated microbundle} is an equivalence class of representatives.
\end{definition}

Continuing with \cref{ex:MBtangentbundle}, we see that the tangent bundle $\tau M$ can be endowed with the structure of a foliated microbundle.
\begin{example}[{\cite[Lemma 2.1]{Milnor64micro}}]\label{ex:MBtangentbundleFol}
     We see that $\tau M$ has a canonical foliation with leaves $ M \times \{*\}$, making it into a foliated microbundle. It is equivalent, as foliated microbundles, to the tangent bundle foliated by the fibers of $\exp: TM \to M$ for any choice of metric on $M$.
\end{example}

We introduce the following terminology to describe how the foliation and manifold interact. 
\begin{definition}\label{def:FolMbSing} 
    Let $(E \to M, \SF)$ be a foliated microbundle over an $m$-dimensional manifold $M$. The foliated microbundle is \textbf{regular} if $M$ is transverse to the foliation $\SF$. The \textbf{foliation $\SF$ has (only) wrinkle singularities} with respect to $M$ if there exists a collection of disjoint opens $\{B_i\}_{i \in I}$ in $M$, each diffeomorphic to $\R^m$, and open neighborhoods $V_i\subset E|_{B_i}$ of $\iota(B_i)$ such that: 
    \begin{itemize}
        \item $E | _{M \setminus \sqcup_i B_i}$ is regular, and 
        \item $\iota(B_i) \to V_i / \SF$ is a wrinkle.
    \end{itemize}
\end{definition}

\subsubsection{\texorpdfstring{$\SG$}{G}-microbundles}
When we endow a microbundle with a $\SG$-foliation, we obtain the notion of a $\SG$-microbundle.

\begin{definition}
    Let $\SG \rrarrows X$ be an effective groupoid, where $X$ is of dimension $n$. A \textbf{representative of a $\SG$-microbundle} over an $m$-manifold $M$ consists of a representative of a microbundle $E \to M$ of rank $n$, endowed with a $\SG$-foliation that is transverse to the fibers of $E$.

    Two representatives $(E,\pi,\iota,\{\gamma_{ij}\}_{i,j\in I})$ and $(E',\pi',\iota',\{\delta_{kl}\}_{k,l\in K})$ are \textbf{equivalent} if there exists an open neighborhood $W \subset E$ of $\iota(M)$ and a smooth embedding $f:W \to E'$ such that $f$ satisfies $f \circ \iota = \iota'$ and $\pi = \pi' \circ f$, and such that the cocycle $f^*\{\delta_{kl}\}_{k,l\in K}$ is equivalent to $\{\gamma_{ij}\}_{i,j\in I}$. A \textbf{$\SG$-microbundle} is an equivalence class of representatives.
\end{definition}

By considering a $\SG$-cocycle representing the $\SG$-foliation that is compatible with a trivializing atlas of the underlying microbundle, we obtain the following description of a representative of a $\SG$-microbundle. 
We see that opens in the base $X$ form the fibers of the representative microbundle, and the germs of its transition functions will lie in $\SG$. Hence we also say that a $\SG$-microbundle is a microbundle with \textbf{structure groupoid} $\SG$.

\begin{lemma} \label{lem:GMbAtlas}
    Let $\SG \rrarrows X$ be an effective groupoid, where $X$ is of dimension $n$. Let $E\to M$ be a representative of a $\SG$-microbundle over a manifold $M$. Then there exists a cover by opens $ \{U_i\}_{i \in I}$ of $M$ such that for each $U_i$ there exists an open neighborhood $V_i \subset E$ of $\iota(U_i)$ and an embedding $\phi_i: V_i \to U_i \times X$ making the following diagrams commute
    \[
    	\begin{tikzcd}
            V_i \arrow[d, "\pi"] \arrow[r,"\phi_i"] & U_i \times X \arrow[d,"\pr_1"] & & V_i \arrow[r,"\phi_i"] & U_i \times X \\
            M \arrow[r,hookleftarrow] &  U_i & & U_i \arrow[u,"\iota"] \arrow[r,hookrightarrow] & U_i \times \{0\} \arrow[u,hookrightarrow].
        \end{tikzcd}
    \]
    Moreover, the transition maps given by the composition
    \begin{equation*} 
        \phi_j(V_i \cap V_j) \xrightarrow{\phi_j^{-1}} V_i \cap V_j \xrightarrow{\phi_i} \phi_i(V_i \cap V_j)
    \end{equation*}
    satisfy the following for all $y \in U_i \cap U_j$ and $z \in \pr_2( \phi_j(V_i\cap V_j))$:
    \begin{itemize}
        \item they are of the form $(y,z) \mapsto (y,(\phi_{ij}(y)) (z))$ where $\phi_{ij}: U_i \cap U_j \to \Diff_\loc(X)$, and 
        \item the assignment $y \mapsto [\phi_{ij}(y)]_{\phi_j(\iota(y))}$ defines a smooth map $U_i \cap U_j \to \SG$. 
    \end{itemize}
\end{lemma}

Since a $\SG$-foliation defines in particular an ordinary foliation, we note that a $\SG$-microbundle is in particular a foliated microbundle. 
In the case that $\SG = \Gamma^n$, we obtain the following.
\begin{corollary} \label{cor:corGammaNMB-FolMB}
    Let $M$ be a manifold. There is a one-to-one correspondence between $\Gamma^n$-microbundles and foliated microbundles over $M$.
\end{corollary}
\begin{proof}
    It remains to show that a foliated microbundle defines a $\Gamma^n$-microbundle. If we start with a foliated microbundle we find trivializing opens $V_i \subset E$ diffeomorphic to $U_i \times \R^n$, where the leaves of the foliation restricted to $V_i$ are sent to $U_i \times \{*\}$. Since the bundle is foliated, the transition functions $\phi_{ij}: U_i \cap U_j \to \Diff_\loc(\R^n)$ as in \cref{lem:GMbAtlas} are constant as functions on $U_i \cap U_j$. Hence we can write $\phi_{ij} \in \Diff_\loc(\R^n)$, implying that the microbundle has structure groupoid $\Gamma^n$.
\end{proof}

\subsubsection{Correspondence between principal bundles and microbundles} \label{sec:CorPbMb}
We now generalize the correspondence between $\Gamma^n$-cocycles and foliated microbundles from Haefliger \cite{Haef1,Haef2} to the larger class of effective groupoids $\SG \rrarrows X$. The reader should think of the microbundle as the associated bundle of a principal $\SG$-bundle. Our motto is that microbundles are more convenient to work with, because we can use arguments from differential topology to study them.
\begin{lemma} \label{lem:corPGB-Mb}
    Let $M$ be a manifold and $\SG \rightrightarrows X$ an effective groupoid. Then there is a one-to-one correspondence between equivalence classes of $\SG$-cocycles and $\SG$-microbundles over $M$.
\end{lemma}
\begin{proof} 
    Let $\{\gamma_{ij}: U_i \cap U_j \rightarrow \SG\}_{i,j\in I}$ be a $\SG$-cocycle. By passing to an equivalent cocycle we can assume the opens $U_i$ to be arbitrarily small. Hence, using that the groupoid is \'etale, we take representatives to find:
    \begin{itemize}
        \item opens $V_i \subset X$ such that $g_i = \source \circ \gamma_{ii}:U_i \to V_i$, and
        \item diffeomorphisms $g_{ij}: V_j \rightarrow V_i$ with $[g_{ij}]_{g_j(x)} = \gamma_{ij}(x)$ for all $x\in U_j$,
    \end{itemize}
    such that we obtain $g_{ij} \circ g_j = g_i$. We define $E = \sqcup_i (U_i \times V_i )/ \sim$ and write $(i,x,a) \sim (j,y,b)$ if and only if $x=y$ and $b = g_{ji}(a)$. By construction this space projects to $M$. Moreover, there is a smooth section $\iota: M \rightarrow E$ given locally over $U_i$ by $(\id,g_i)$. 

    For the other direction, we let $\{U_i\}_{i \in I}$ be a cover by opens of $M$ such that the microbundle trivializes over each $U_i$ and there exist transition functions $\phi_{ij}: U_i \cap U_j \to \Diff_\loc(X)$ as in \cref{lem:GMbAtlas}. Then we define the $\SG$-cocycle $\{\delta_{ij}: U_i \cap U_j \to \Eff(\SG)\}_{i,j \in I}$  by $\delta_{ij}(y) = [\phi_{ij}]_{\phi_j(\iota(y))}$ for $y \in U_i \cap U_j$. 
\end{proof}
\begin{remark}\label{rem:MbHasSmoothStr}
    We note that in the proof of \cref{lem:corPGB-Mb}, the space $E$ admits a smooth structure such that the inclusion $\iota$ and projection $\pi$ are smooth, because the manifold $M$ is smooth and the groupoid $\SG$ is Lie and \'etale. Additionally we use the continuity of the given cocycle, but not its smoothness.
\end{remark}

Then by combining \cref{pro:corCocyclePrBundle,cor:corGammaNMB-FolMB,lem:corPGB-Mb} we recover Haefliger's correspondence.
\begin{corollary}[Haefliger \cite{Haef2}] \label{cor:corPGammaNB-FolMB}
    Let $M$ be a manifold. There is a one-to-one correspondence between isomorphism classes of principal $\Gamma^n$-bundles and foliated microbundles over $M$. 
\end{corollary}

The proofs of \cref{pro:corCocyclePrBundle,lem:corPGB-Mb}, show us how we can obtain from a $\SG$-cocycle respectively a principal $\SG$-bundle or $\SG$-microbundle, and vice versa. By combining these constructions, we obtain the following method of directly obtaining from a $\SG$-microbundle a $\SG$-principal bundle. We end up with a construction similar to the frame bundle in the case that $\SG= \Gamma^n$. 
\begin{remark} \label{rem:corPrBMb}
    Let $\SG \rrarrows X$ be an effective groupoid and let $E\to M$ be a representative of a $\SG$-microbundle. Then the corresponding principal $\SG$-bundle $P \to M$ is the fiber bundle 
    \[P = \left\{ (m,[\phi|_{E_m}]_{\iota(m)}) \;\middle|\;
		\begin{varwidth}{0.6\linewidth}
			$m \in M$ and $\phi:V \to U \times X$  a local trivialization of $E$ containing  $(m,\iota(m)) $ as in \cref{lem:GMbAtlas}
		\end{varwidth}
    \right\}, \]
    where we write by a slight abuse of notation $\iota(m) \in E_m$ instead of $\iota(m) \in E$. 
    The moment map $\mu : P \to X$ is defined by $(m,[\phi]_{\iota(m)}) \mapsto \phi(\iota(m))$ and the action $\SG \timeslr{\source}{\mu}  P  \to P$ is defined by $[g]_{\phi(\iota(m))} \cdot (m,[\phi]_{\iota(m)})  \mapsto (m,[g \circ \phi ]_{\iota(m)})$.
\end{remark}

\subsection{Summary} \label{ssec:summaryBundles}
We summarize the previous sections as follows, where all $\leftrightarrow$ indicate a one-to-one correspondence. This correspondence is up to isomorphism for principal bundles and up to equivalence in the case of cocycles:
\[
    \begin{tikzcd}[column sep={0.8em},row sep = tiny]
        \{ \SG\text{ - cocycles} \} \arrow[r,leftrightarrow] & \{ \text{principal $\SG$-bundles} \} \arrow[r,leftrightarrow,"\SG \text{ effective}"] &[2.6em] \{ \SG \text{-microbundles} \}, & \\
        \{ \Gamma^n \text{ - cocycles} \} \arrow[r,leftrightarrow] & \{  \text{principal $\Gamma^n$-bundles} \} \arrow[r,leftrightarrow] & \{ \Gamma^n  \text{-microbundles} \} \arrow[r,leftrightarrow] & \{ \text{foliated microbundles} \}.
    \end{tikzcd}
\]

\subsection{Homotopy classification of principal bundles}\label{sec:prelim2classspace}
Our overall goal in the context of h-principles is to construct and classify a certain class of principal groupoid bundles. In order to speak of a classification of principal groupoid bundles, which we understand to be \emph{up to homotopy}, we have to construct a \emph{space} of principal groupoid bundles, which is the goal of this section.

To define a space of principal bundles, we have to define what it entails to be a continuous family of $\SG$-bundles. This role will be taken up by the following notion of concordance.
\begin{definition}
    Let $M$ and $K$ be smooth manifolds, where $K$ serves as parameter space. A \textbf{$K$-concordance} of principal $\SG$-bundles over $M$ is a principal $\SG$-bundle over $M \times K$.
\end{definition}
In the particular case where $K=[0,1]$, we refer to a $K$-concordance simply as a concordance. The two principal $\SG$-bundles over respectively $M \times\{0\}$ and $M \times \{1\}$ are then said to be \textbf{concordant}.
Analogously we define a $K$-concordance of microbundles.
We note that $K$-concordances respect isomorphism classes of principal $\SG$-bundles, so that a $K$-concordance is also defined between elements of $\pb{M}{\SG}$. We recall that we introduced the notation $\pb{M}{\SG}$ to denote the set of isomorphism classes of principal $\SG$-bundles over $M$.

It was proven by Haefliger~\cite{Haef2} that any topological groupoid $\SG$ admits a \textbf{classifying space} $B\SG$. We state the following consequence for a certain class of Lie groupoids. 
\begin{proposition} \label{pro:HaefligerClassifying}
Let $M$ be a manifold and $\SG$ a Lie groupoid that is either \'etale or both Hausdorff and second-countable. There is a bijection between concordance classes of principal $\SG$-bundles over $M$ and homotopy classes of maps from $M$ into the classifying space $B\SG$:
\[ \dfrac{\pb{M}{\SG}}{\textrm{concordance}}  \,\cong\, [M,B\SG]. \]
\end{proposition}
As mentioned, the statement holds true more generally in the topological setting, where the manifold $M$ is replaced by a topological space, the groupoid $\SG$ is a topological groupoid, and the bundles (in particular also the concordances) are topological bundles, as long as one restricts to so-called numerable principal bundles\footnote{A principal bundle is called numerable if it can be given as a cocycle defined on opens $\{U_i\}$ such that there exists a locally finite partition of unity $\{\eta_i\}$ satisfying that $\eta_i^{-1}(0,1] \subset U_i$.}. We do not need this generality, but we comment on the transition from the topological to the smooth setting in \cref{sec:classifyingTopSmooth}.

\Cref{pro:HaefligerClassifying} should be understood as a statement about the path-components of the space of isomorphism classes of principal $\SG$-bundles over $M$. This suggests taking $\Maps(M,B\SG)$ as the space of principal $\SG$-bundles over $M$. This is further supported by the following consequence of \cref{pro:HaefligerClassifying}:
\[ \dfrac{\pb{M \times S^i}{\SG}}{\textrm{concordance}}  \,\cong\, [M \times S^i ,B\SG] \,\cong\, [S^i ,\Maps(M,B\SG)]. \]
By choosing base points and hence passing to the pointed setting, this implies that the right hand side computes $S^i$-concordances up to concordance. We interpret this as being the $i$th homotopy group of the space of principal $\SG$-bundles over $M$.

Haefliger's construction of the classifying space of any topological groupoid is explicit and based on Milnor's join construction for the Lie group case. We will later use that the construction is functorial in the groupoid.

\section{Technical details about principal \texorpdfstring{$\SG$}{G}-bundles} \label{sec:TechResultsPb}
\Cref{sec:prelimGroupoids} discussed some of the classical material related to Lie groupoids $\SG$, principal $\SG$-bundles, $\SG$-cocycles and $\SG$-microbundles. Here we discuss two topics that also lie in this direction, but for which the authors were not aware of a reference in literature. Firstly, in \cref{sec:classifyingTopSmooth}, we discuss the smoothing of topological principal $\SG$-bundles. Secondly, we recall that all of the microbundles in \cref{sec:prelim2Haefliger} have effective structure groupoids and are thus endowed with at least a foliation. Hence in \cref{sec:GammaWhMicrobundles} we discuss microbundles without any extra structure and their relation to principal groupoid bundles.

\subsection{Smoothing of principal bundles} \label{sec:classifyingTopSmooth}
In order to study spaces of principal $\SG$-bundles over smooth manifolds $M$, where $\SG$ is a Lie groupoid that is either \'etale or both Hausdorff and second-countable, it turns out to be sufficient to work in the smooth setting. This section provides a justification of this.

Our definitions in previous sections were in the smooth category. However, as mentioned in \cref{rem:groupoidTopDef}, they can all be generalized to the topological category. 
This defines the set $\pbT{M}{\SG}$ of isomorphism classes of topological principal $\SG$-bundles. We prefer the smooth category because it allows us to use arguments from differential topology, similar to our motivation for introducing microbundles to deal with principal bundles.

We were able to restrict ourselves to smooth principal $\SG$-bundles in \cref{pro:HaefligerClassifying}, due to the fact that, when working with sufficiently regular Lie groupoids $\SG$ and smooth manifolds $M$, every topological principal $\SG$-bundle over $M$ can be upgraded to a smooth one via a (topological) concordance. We spend the remainder of this section on proving this fact:

\begin{theorem}\label{cor:ContPBIsSmoothPBSpace}
    Let $M$ be a smooth manifold, and $\SG$ a Lie groupoid that is either \'etale or both Hausdorff and second-countable. Then we have the following bijection
    \[ \frac{\pb{M}{\SG}}{\text{smooth concordance}} \cong \frac{\pbT{M}{\SG}}{\text{topological concordance}}.\]
\end{theorem}

In the specific case where $\SG$ is effective, the result follows quite readily by using microbundles. 

\begin{lemma} \label{lem:smoothingContPB}
    Let $\SG \rrarrows X$ be an effective Lie groupoid and $M$ a smooth manifold. Every topological principal $\SG$-bundle $P_{\mathrm{top}} \to M$ is topologically concordant to a smooth principal $\SG$-bundle $P \to M$. Moreover, $P$ is unique up to smooth concordance.
\end{lemma}
\begin{proof}
    We first apply \cref{lem:corPGB-Mb} to a topological cocycle representing $P_{\mathrm{top}}$, yielding a (representative of) a topological $\SG$-microbundle $(E,\pi,\iota_{\mathrm{top}})$. By \cref{rem:MbHasSmoothStr}, we see that the space $E$ admits a smooth structure such that the projection $\pi: E \to M $ is smooth, because the groupoid $\SG$ is an \'etale Lie groupoid and the manifold $M$ is smooth.
    
    The zero section $\iota_{\mathrm{top}}$ however does not need to be smooth as the cocycle consists of continuous maps. We address this by smoothing $\iota_{\mathrm{top}}$, which can be achieved while staying a section.\footnote{Segal already observed that, in the definition of a smooth microbundle, we do not need to require the section $\iota$ to be smooth because it can always by smoothened via a concordance~\cite{segal1978classifying}.} The interpolation between $\iota_{\mathrm{top}}$ and its smoothing $\iota$ provides a concordance between $(E,\pi,\iota_{\mathrm{top}})$ and $(E,\pi,\iota)$. Converting back to principal bundles, provides us with $P$.

    To show uniqueness, we let $F_0 \to M$ and $F_1 \to M$ be two smooth $\SG$-microbundles, to which $E$ is topologically concordant. We combine these concordances to a topological concordance $F\to M \times [0,1]$ such that $F|_{M\times \{t\}} = F_i$ for all $t$ in a neighborhood of respectively $0$ and $1$. We now apply the above procedure relative to $M \times\{0\}$ and $M \times \{1\}$ to upgrade $F$ to a smooth concordance between $F_0$ and $F_1$. 
\end{proof}

The general case, where $\SG$ is a Lie groupoid which is not necessarily effective, but either \'etale or both Hausdorff and second-countable, is more involved. Hence in the remainder of this section we turn our attention to such Lie groupoids $\SG$. For a trivial principal $\SG$-bundle, we can just smoothen the underlying map into the base $X$ of $\SG$.
\begin{lemma} \label{lem:smoothCocycle1}
    Let $\SG\rrarrows X$ be a Lie groupoid and $\gamma: U \to X$ a continuous map defined on a smooth manifold $U$. The trivial (topological) principal $\SG$-bundle $P_\gamma \to U$ is topologically concordant to a smooth principal $\SG$-bundle. 
\end{lemma}
\begin{proof}
    Smoothing $\gamma$ provides us with a homotopy of maps $H : U \times [0,1] \to X$ such that $H|_{U \times \{0\}}$ equals $\gamma$ and such that $H|_{U \times \{t\}}$ is smooth for $t>0$. The trivial principal $\SG$-bundle $P_H \to U \times [0,1]$ provides a topological concordance between $P_H|_{ U \times \{0\} } = P_\gamma$ and a smooth principal $\SG$-bundle $P_H|_{ U \times \{1\} }$.
\end{proof}

When the principal $\SG$-bundle we consider is not trivial, we do not just have to smoothen the moment maps, but also the transition functions. In terms of cocycles, this means we have to homotope a cocycle $\{\gamma_{ij}: U_i \cap U_j \to \SG\}_{i \in I}$ while preserving the cocycle condition. Here we define a \textbf{$K$-family of $\SG$-cocycles} as a $K$-family of maps $\{t\mapsto (\gamma_{ij}^t :U_i \cap U_j \to \SG)\}_{i \in I}$ such for each $t\in K$ the cocycle condition is satisfied by $\{\gamma_{ij}^t \}_{i \in I}$. We emphasize that the cover $\{U_i\}_{i \in I}$ does not depend on $t\in K$. If $K=[0,1]$ we refer to the family as a \textbf{homotopy of $\SG$-cocycles}. Two $K$-families of $\SG$-cocycles are \textbf{equivalent} if their union can be extended to a $K$-family of $\SG$-cocycles.
Using \cref{pro:corCocyclePrBundle}, we obtain the following.
\begin{lemma} \label{lem:famCocycleConcordance}
    Let $\SG \rrarrows X$ be a Lie groupoid and let $M$ be a smooth manifold. Then a (topological/smooth) $K$-family of $\SG$-cocycles over $M$ defines a (topological/smooth) $K$-concordance of $\SG$-bundles over $M$.
\end{lemma}

To find a homotopy from a continuous cocycle to a smooth cocycle, we will work inductively. To illustrate the main strategy, assume we have a cocycle $\{\gamma_{ij}\}$ defined on two opens $U_1$ and $U_2$ that cover a manifold $M$. We smoothen the moment map $\gamma_{11}$ and lift this homotopy to a homotopy of the transition function $\gamma_{21}$ which we then also smoothen. The cocycle condition $\gamma_{22}=\gamma_{21}\gamma_{12}$ prescribes the smoothing of $\gamma_{22}$ on the overlap $U_1 \cap U_2$. We now choose two opens $V_1$ and $V_2$ such that $\overline{V_i} \subset U_i$ and such that $V_1 \cup V_2=M$. We extend the smoothing of $\gamma_{22}$ to $U_2$ relative to $\overline{V_1}\cap \overline{V_2}$. This will stop the cocycle condition to hold on $U_2 \setminus (\overline{V_1}\cap \overline{V_2})$, but it does hold over $V_2$. The result will hence be a smooth cocycle over $M$ with respect to the cover $\{V_1,V_2\}$. We implement this for general finite covers, and in greater detail, as follows.

\begin{lemma}\label{lem:smoothCocycleN}
    Let $\SG \rrarrows X$ be a Lie groupoid that is either \'etale or both Hausdorff and second-countable, and let $P \to M$ be a topological principal $\SG$-bundle over a smooth compact manifold $M$. Then there exists a topological concordance between $P$ and a smooth principal $\SG$-bundle.
\end{lemma}
\begin{proof}
    We prove this by induction on the size $n$ of the covering of $M$. The base case $n=1$ is covered by \cref{lem:smoothCocycle1}. Let $\{\gamma_{ij} : U_i \cap U_j \to \SG \}_{i,j = 1,\dots ,n}$ be a cocycle defining $P$ where $\{U_i\}_{i=1,\dots,n}$ is a finite cover of $M$. By induction we assume that we are given a homotopy of cocycles $\{ \gamma_{ij}^t \}_{i,j=1,\dots,n-1}$ where for all $i,j \in \{1,\dots,n-1\} $ it holds that $\gamma_{ij}^0=\gamma_{ij}$ and that $\gamma_{ij}^t$ is smooth whenever $t\in(0,1]$. We now want to extend this homotopy of cocycles to include the $\gamma_{ij}$ where $i$ and $j$ are allowed to take the value $n$.

    We however do not extend the homotopy of cocycles $\{ \gamma_{ij}^t \}_{i,j=1,\dots,n-1}$ such that it satisfies the cocycle conditions with respect to the cover $\{U_i\}_{i=1,\dots,n}$, but with respect to a smaller cover. That is, we will define a family of cocycles $\{ \gamma_{ij}^t : \overline{V_i} \cap \overline{V_j} \to \SG\}_{i,j=1,\dots,n}$ where $\{V_i\}_{i=1,\dots,n}$ is a cover by opens of $M$ such that $V_i \subset \overline{V_i} \subset U_i$. This allows us to extend a lift or smoothing over $\overline{V_i}\cap \overline{V_j}$ to $\overline{V_i}$. Moreover, we choose the cover $\{V_i\}_{i=1,\dots,n}$ such that the closures $\overline{V_i}$ are compact and all possible intersections of the $\overline{V_i}$ are cell complexes.

    By induction on $i\leq n$ we now define the families $\gamma_{ni}^{t,s}: \overline{V_n} \cap \overline{V_i} \to \SG$ (and thus their inverses $\gamma_{in}^{t,s}$). Here we have two parameters $t\in[0,t_i]$ with $t_i \in (0,1]$ and $s \in[0,1]$, where $t$ arises from the induction on $n$, and $s$ will be used for additional smoothings in the induction on $i$. We define $\gamma_{ij}^{t,s}=\gamma_{ij}^{t}$ when both $i,j<n$. We construct the $\gamma_{ni}^{t,s}$ such that the following induction hypothesis holds:
    \begin{itemize}
        \item we have $\gamma_{ni}^{0,0} = \gamma_{ni}$ and $\gamma_{ni}^{t,0} = \gamma_{ni}^t$,
        \item the map $\gamma_{ni}^{t,s}: \overline{V_n} \cap \overline{V_i} \to \SG$ is smooth when $t>0$ and $s>0$, and
        \item we have 
        \begin{align*}
            \gamma_{jk}^{t,s} \gamma_{kl}^{t,s} = \gamma_{jl}^{t,s} &\text{ over } \overline{V_j} \cap \overline{V_k} \cap \overline{V_\ell}
        \end{align*}
        for all $t \in [0,\min_i{t_i}]$, $s\in[0,1]$ and all triples $(j,k,l)$ such that \[\{jk\},\{kl\},\{jl\} \notin \{\{i+1, n\},\dots,\{n,n\}\}.\]
    \end{itemize}
    We note that because of the cocycle conditions, a family $\gamma_{ni}^{t,s}$ will be (partly) determined by the $\gamma_{nj}^{t,s}$ for all $j<i$ such that $\overline{V_i} \cap \overline{V_j} \neq \emptyset$. 

    \pfstep{Homotoping $\gamma_{n1}$}
    We start with the base case $i=1$ by constructing $\gamma_{n1}^{t,s}$. 
    The source map $\source : \SG \to X$ is a submersion and therefore a microfibration. Using the compactness of $\overline{V_1}$, this means that there exists $t_1>0$ such that we can lift the homotopy $\gamma_{11}^t$ along $\source$ to a homotopy $\gamma_{n1}^t$ for $t\in [0,t_1]$ where $\gamma_{n1}^0 = \gamma_{n1}$. We note it then holds that $\source (\gamma_{n1}^{t}) = \gamma_{11}$. The homotopy $\gamma_{n1}^t$ generally consists of continuous maps and hence we apply smoothing. We do so fiberwise along the source fibers, using again that $\source$ is a submersion. This produces the required family $\gamma_{n1}^{t,s}: \overline{V_1} \cap \overline{V_n} \to \SG$ of continuous maps with $t\in[0,t_1]$ and $s\in[0,1]$. 

    \pfstep{Homotoping the $\gamma_{ni}$}
    For the inductive step, we assume we have defined $\gamma_{nj}^{t,s}$ with $t \in [0,t_j]$ and $s\in [0,1]$ for all $j<i<n$ such that they satisfy the induction hypothesis. We now construct $\gamma_{ni}^{t,s}$. We proceed similarly as in the case where $i=1$, but now relative to the previously constructed homotopies. 
    
    For each $j<i$ we define $(\tilde \gamma_{ni}^t)^{(j)}$ on $\overline{V_i} \cap \overline{V_n} \cap \overline {V_j}$ by $(\tilde \gamma_{ni}^t)^{(j)} = \gamma_{nj}^{t,0}\gamma_{ji}^t$. By the induction hypothesis, the $(\tilde \gamma_{ni}^t)^{(j)}$ glue to a map $\tilde \gamma_{ni}^t$ defined on $\overline{V_i} \cap \overline{V_n} \cap (\cup_{j < i} \overline{V_j})$. Here $t$ takes values in $[0,\min_{j<i} t_j]$ and $s$ in $[0,1]$. We note that $\tilde \gamma_{ni}^t$ lifts the homotopy $\gamma_{ii}^t$ along $\source$ over $\overline{V_i} \cap \overline{V_n} \cap (\cup_{j < i} \overline{V_j})$, which is a cell complex. Hence we are able to lift the homotopy $\gamma_{ii}^t$ along $\source$ to define the homotopy $\gamma_{ni}^t : \overline{V_n} \cap \overline{V_i} \to \SG$ with $t \in[0,t_i]$ such that $\gamma_{ni}^t$ agrees with $\tilde \gamma_{ni}^t$ over $\overline{V_i} \cap \overline{V_n} \cap (\cup_{j < i}\overline{V_j})$. Here we assume that $t_i \leq \min_{j<i} t_j$. Then the cocycle condition for the $\{\gamma_{ij}^{t,s}\}$ from the induction hypothesis holds when restricted to $s=0$.
    
    The homotopy $\gamma_{ni}^t$ consists again of continuous maps and hence we apply smoothing to obtain a family $\gamma_{ni}^{t,s}$. We do so relative to the previously obtained families $\gamma_{nj}^{t,s}$ with $j<i$, analogous to the above. That is, because of the induction hypothesis we obtain a map $\tilde \gamma_{ni}^{t,s}$ defined on $\overline{V_i}\cap \overline{V_n} \cap (\cup_{j<i} \overline{V_j})$. We now smoothen $\gamma_{ni}^t$ along the fibers of $\source$, which produces the map $\gamma_{ni}^{t,s}$. We do so relative to $\tilde \gamma_{ni}^{t,s}$, which achieves that $\gamma_{ni}^{t,s}$ equals $\tilde \gamma_{ni}^{t,s}$ over $\overline{V_i}\cap \overline{V_n} \cap (\cup_{j<i} \overline{V_j})$. Hence the cocycle condition for the $\{\gamma_{ij}^{t,s}\}$ from the induction hypothesis holds. 

    \pfstep{Homotoping $\gamma_{nn}$} Having constructed all the $\gamma_{ni}^{t,s}$ with $i<n$, we now construct $\gamma_{nn}^{t,s}$. 
    By requiring that $\tilde\gamma_{nn}^{t,s} = \gamma_{ni}^{t,s} \gamma_{in}^{t,s}$ for all $i<n$, we obtain the family $\tilde\gamma_{nn}^{t,s}$ defined over $\overline{V_n} \cap (\cup_{i\leq n}\overline{V_i})$. The family $\tilde\gamma_{nn}^{t,s}$ is well defined, because of the cocycle condition satisfied by the $\gamma_{ij}^{t,s}$ over the cover $\{\overline{V_i}\}$.
    We smoothen $\gamma_{nn}: \overline{V_n} \to X$ to obtain a family $\gamma_{nn}^{t,s}$ with $t \in [0,\min_{i<n} t_i]$ and $s\in [0,1]$ such that $\gamma_{nn}^{0,0} = \gamma_{nn}$ and such that $\gamma_{nn}^{t,s}$ agrees with $\tilde \gamma_{nn}^{s,t}$ over $\overline{V_n}$. Then the family $\{\gamma_{ij}^{t,s} : \overline{V_i} \cap \overline{V_j} \to \SG \}_{i = 1,\dots ,n}$ satisfies the cocycle conditions.
    
    \pfstep{Conclusion} We have constructed a family of continuous cocycles $ \{ \gamma_{ij}^{t,s} :V_i \cap V_j \to \SG \}_{i,j=1,\dots,n}$, which is smooth when $t>0$ and $s>0$. In particular we obtain the homotopy $t\mapsto \{ \gamma_{ij}^{t,t} :V_i \cap V_j \to \SG \}_{i,j=1,\dots,n}$ of continuous cocycles starting at $\{\gamma_{ij}\}$ and ending at a smooth cocycle. By \cref{lem:famCocycleConcordance} this defines a topological concordance from $P$ to a smooth principal $\SG$-bundle.
\end{proof}

We observe that the proof of \cref{lem:smoothCocycleN} is relative in nature, in the sense that we are given a homotopy $\{ \gamma_{ij}^t \}_{1\leq i,j <n}$ and extend this to include $\gamma_{ij}^t$ where $i$ and $j$ may take the value $n$. We do not alter the $\gamma_{ij}^t$, except that we shrink their domain. Hence we obtain the following relative version of \cref{lem:smoothCocycleN}. 
\begin{corollary}\label{lem:smoothCocycleRel}
    Let $\SG \rrarrows X$ be a Lie groupoid that is either \'etale or both Hausdorff and second-countable, and let $P \to M$ be a topological principal $\SG$-bundle over a compact manifold $M$. We assume that $M$ is covered by two opens $V$ and $U$ such that there exists a topological concordance $Q\to V \times [0,1]$ from $P|_V$ to a smooth principal $\SG$-bundle. Then for any compact subset $V'$ of $V$, the concordance $Q|_{V'}$ can be extended to a topological concordance between $P$ and a smooth principal $\SG$-bundle.
\end{corollary}
\begin{proof}
    We cannot assume that $Q$ is defined by a homotopy of cocycles. Hence we proceed as follows: 
    Let $C$ be a compact subset of $M$ such that $V'\subset \Int(C) \subset C \subset V$. By compactness of $C$ and $[0,1]$ we assume that there exists $T\in[0,1)$ and a cover $\{U_i\}_{i,j\leq n}$ by opens of $M$ such that
    \begin{itemize}
        \item there exists $m<n$ such that $C\subset \cup_{j\leq m } U_j$,
        \item $Q|_{\cup_{j\leq m}U_j\times[T,1]}$ is defined by a homotopy of cocycles $\{ \gamma_{ij}^t :  U_i \cap U_j  \to \SG \}_{i,j\leq m}$ with $t\in[T,1]$ and $m\leq n$, satisfying that the cocycle is smooth for $t=1$, and
        \item $Q$ extends to a topological concordance $\tilde Q $ over $M\times[0,T]$ such that $\tilde{Q}|_{M\times\{T\}}$ is defined by a cocycle $\{ \gamma_{ij}: U_i \cap U_j \to \SG \}$, where $\gamma_{ij}^T = \gamma_{ij}$.
    \end{itemize}
    
    We let $\{W_i\}_{i,j\leq n}$ be a cover by opens of $M$ such that $\overline{W_i} \subset U_i$ and $V' \subset \cup_{j\leq m} W_i \subset \cup_{j\leq m} \overline{W_i} \subset \Int(C)$.
    Then as in the proof of \cref{lem:smoothCocycleN} we extend the homotopy of cocycles $\{\gamma_{ij}^t:  U_i \cap U_j  \to \SG\}_{i,j\leq m}$ with $t\in[T,1]$ to allow for $i$ and $j$ to take values in $\{1,\dots,n\}$. This produces a homotopy of cocycles $\{\gamma_{ij}^t: W_i \cap W_j \to \SG\}_{i,j\leq n}$ with $t\in[0,1]$, which restricted to $V'$ defines $Q|_{V'\times[T,1]}$ and is smooth for $t=1$. Hence we have constructed a topological concordance between $\tilde Q|_{M \times\{T\}}$ and a smooth principal $\SG$-bundle which extends $Q|_{M \times[T,1]}$. Composing with $\tilde Q$ provides us with the topological concordance between $P$ and a smooth principal $\SG$-bundle extending $Q|_{V'\times[0,1]}$.
\end{proof}

Combining \cref{lem:smoothCocycle1,lem:smoothCocycleN} provides us with the claim that we can smoothen any topological principal $\SG$-bundle over a smooth manifold. Since concordances are themselves principal $\SG$-bundles, the relative case from \cref{lem:smoothCocycleRel} tells us that this smoothing is unique up to smooth concordance. We combine these two claims in the following proposition.
\begin{proposition} \label{pro:smoothCocycle}
    Let $\SG \rrarrows X$ be a Lie groupoid that is either \'etale or both Hausdorff and second-countable, and let $P\to M$ be a topological principal $\SG$-bundle over a smooth manifold $M$. Then $P$ is topologically concordant to a smooth principal $\SG$-bundle over $M$. Moreover, the smooth principal $\SG$-bundle is unique up to smooth concordance.
\end{proposition}
\begin{proof}
    The case where $M$ is compact follows from inductively combining \cref{lem:smoothCocycle1,lem:smoothCocycleN}.
    
    To show uniqueness, we let $Q_0 \to M$ and $Q_1 \to M$ be two smooth principal $\SG$-bundles, to which $P$ is topologically concordant. We combine these concordances to a topological concordance $Q\to M \times [0,1]$ such that $Q|_{M\times \{t\}} = Q_i$ for all $t$ in a neighborhood of respectively $0$ and $1$. We now smoothen $Q$ relative to $M \times \{0,1\}$. This constructs a smooth concordance between $Q_0$ and $Q_1$. 

    If $M$ is non-compact, we reduce to the compact case using an exhaustion by compacts. That is, we let $\{K_i\}_{i \in \N}$ be an exhaustion of $M$ by compacts and define the ``annuli'' $A_i = \Int{K_{i+1}} \setminus K_{i-1}$ for $i\geq1$ and $A_0=\Int{K_1}$. Since all the $A_{2i}$ are disjoint, we can reason for all of them at the same time. Hence we first smoothen $P$ over all $A_{2i}$ simultaneously. Next, we smoothen $P$ over all $A_{2i+1}$ at once relative to the obtained smoothing over the $A_{2i}$. The result is a smoothing of $P$ over $M$.
\end{proof}

The above proposition directly implies \cref{cor:ContPBIsSmoothPBSpace}.

\subsection{Microbundles without extra structure} \label{sec:GammaWhMicrobundles}
In \cref{sec:CorPbMb} we discussed the correspondence between $\SG$-microbundles and $\SG$-cocycles when the groupoid $\SG$ is effective. Any $\SG$-microbundle comes with a canonical foliation, and in particular we obtained in \cref{cor:corGammaNMB-FolMB} a one-to-one correspondence between foliated microbundles and $\Gamma^n$-microbundles. This raises the question whether microbundles, without any extra structure, can also be related to a cocycle taking values in a groupoid.

The property that ensured that the locally defined foliations pieced together to a global foliation defined on the entire microbundle, was that the groupoids $\SG$ we considered, and in particular $\Gamma^n$, were \'etale. Hence we recall that we defined in \cref{sec:Gamma} the groupoid $\Gamma^{n,\Whitney}$. As a set it agrees with $\Gamma^n$, but we endow $\Gamma^{n,\Whitney}$ with the Whitney topology instead of the \'etale topology. This makes $\Gamma^{n,\Whitney}$ just a topological groupoid, whereas $\Gamma^n$ is a Lie groupoid. For notational convenience we shorten $\Gamma^{n,\Whitney}$ to $\Gamma^{\Whitney}$ in this section.  We now claim that a microbundle defines a $\Gamma^{\Whitney}$-cocycle.
\begin{lemma} \label{lem:MbToGammaWhCocycle}
    A microbundle $E \to M$ canonically defines a $\Gamma^{\Whitney}$-cocycle.
\end{lemma}
\begin{proof}
    Let $\{U_i\}$ be a cover by opens of $M$ such that $E|_{U_i}$ trivializes in a neighborhood of $\iota(M)$. That is, we assume there exist opens $V_i \subset E$ of $\iota(U_i)$ and embeddings $\phi_i : V_i \to U_i \times \R^n$ such that $\phi_i$ respects the projection onto $U_i$. Then the transition functions $U_i \times \R^n \to U_j \times \R^n$ are of the form $(x,y) \mapsto (x,(\phi_{ji}(x))(y))$. We now define the $\Gamma^{\Whitney}$-cocycle $\{\gamma_{ij}: U_i \cap U_j \to \Gamma^{\Whitney}\}$ by $\gamma_{ij}(x)=[\phi_{ij}]_{\phi_i(\iota(x))}$.
\end{proof}

However, not every $\Gamma^{\Whitney}$-cocycle defines a smooth microbundle, which can be seen as follows: Given a $\Gamma^{\Whitney}$-cocycle $\{ \gamma_{ij}: U_i \cap U_j \to \Gamma^{\Whitney} \}$ we are able to find a neighborhood $V_i\subset \R^n$ of $\gamma_{ii}(U_i)$ and a map $g_{ij}(x): V_j \to V_i$ for each $x \in U_i \cap U_j$ that represents $\gamma_{ij}(x)$. Hereby we mean that the maps $g_{ij}(x)$ satisfy $[g_{ij}(x)]_{\gamma_{jj}(x)}=\gamma_{ij}(x)$ for all $x\in U_i\cap U_j$. The map $x\mapsto [g_{ij}(x)]_{\gamma_{jj}(x)}$ is then continuous for the Whitney topology and hence the jets $j^r_{\gamma_{jj}(x)}(g_{ij}(x))$ vary smoothly in $x$ for all $r\in \N \cup\{\infty\}$. However, the transition maps $(x,y)\mapsto(x,(g_{ij}(x))(y))$ are generally not smooth, and as a result they generally do not define a smooth microbundle.
\begin{example}
    An example of the above phenomenon is the continuous map $\gamma: I \to \Gamma^{\Whitney} = \Gamma^{1,\Whitney}$ from \cref{ex:WhitneyEtaleTop}. For any neighborhood $V\subset \R$ of $0$, the map $\gamma$ does not allow for a representative $g: V\to V$ such that both $[g(x)]_0=\gamma(x)$ for all $x \in I$ and such that the map $I\times V\to I \times V$ given by $(x,y)\to(x,g(x)y)$ is smooth.
\end{example}

Instead we obtain the following:
\begin{lemma} \label{lem:PbToMbWhitney}
    A $\Gamma^{\Whitney}$-cocycle $\{\gamma_{ij}: U_i \cap U_j \to \Gamma^{\Whitney} \}$ defines a smooth microbundle if there exists a family of smooth embeddings $g_{ij}(x): V_j \to V_i$ indexed by $x \in U_i \cap U_j$ where $V_i\subset \R^n$ is a neighborhood of $\gamma_{ii}(U_i)$ satisfying that
    \begin{itemize}
        \item $[g_{ij} (x)]_{\gamma_{jj}(x)} = \gamma_{ij}(x)$, and
        \item the map $U_i \times V_j \to U_j \times V_i$ defined by $(x,y)\mapsto(x,(g_{ij}(x))(y))$ is smooth.
    \end{itemize}
\end{lemma}

Combining \cref{lem:MbToGammaWhCocycle,lem:PbToMbWhitney}, we obtain the following. 
\begin{corollary}
    There is a one-to-one correspondence between microbundles and equivalence classes of $\Gamma^{\Whitney}$-cocycles which have a representative cocycle satisfying the conditions of \cref{lem:PbToMbWhitney}.
\end{corollary}

Given a $\SG$-microbundle, we directly construct a principal $\SG$-bundle from it in \cref{rem:corPrBMb}. This also provides a direct way to obtain from a microbundle a principal $\Gamma^{\Whitney}$-bundle with the same underlying equivalence class of $\Gamma^{\Whitney}$-cocycles: 
\begin{remark}\label{lem:MbToPBGammaWhitney}
    Let $E \to M$ be a microbundle. Then the associated principal $\Gamma^{\Whitney}$-bundle $Q \to M$ consists of the fiber bundle
    \[ Q= \{(m,[\phi]_{\iota(m)}) \mid m \in M \text{ and } \phi: U \to \R^n \text{ an embedding of an open } U \subset E_m  \text{ containing } \iota(m)  \},  \]
    where we recall that we write by a slight abuse of notation $\iota(m) \in E_m$ instead of $\iota(m) \in E$. The bundle $Q$ has moment map $\mu : Q \to \R^n$ defined by $(m,[\phi]_{\iota(m)}) \mapsto \phi(\iota(m))$. The action $\Gamma^{\Whitney} \timeslr{\source}{\mu} Q  \to Q$ is defined by $[f]_{\phi(\iota(m))} \cdot (m,[\phi]_{\iota(m)}) \mapsto (m,[f \circ \phi ]_{\iota(m)})$. 
\end{remark}

\subsubsection{Lifting principal \texorpdfstring{$J^r\Gamma^n$}{J\^{}rΓ\^{}n}-bundles to microbundles} \label{sec:JrGammaToMb}
A disadvantage of $\Gamma^{\Whitney}$ is that it is just a topological groupoid, whereas we are interested in smooth microbundles. In \cref{sec:Gamma} we however also introduced the groupoid $J^r\Gamma^n$ which is a Lie groupoid when $r\in\N$. Moreover, the natural groupoid morphism $j^r: \Gamma^{\Whitney} \to J^r\Gamma^n$ is surjective and has contractible fibers on the level of both the base and arrow space. Hence we now consider the question whether $J^r\Gamma^n$-cocycles, and hence principal $J^r\Gamma^n$-bundles, lift to define a microbundle.

\begin{remark}
    Throughout this section we will restrict ourselves to $r\in\N$ (or even $r\in\N_{>0}$). In these cases the groupoid $J^r\Gamma^n$ is Lie and hence we consider only smooth cocycles and smooth principal $J^r\Gamma^n$-bundles. \end{remark}

On the level of cocycles we define lifting in the following straightforward manner.
\begin{definition}
    Given a $J^r\Gamma^n$-cocycle $\{\gamma_{ij}: U_i \cap U_j \to J^r\Gamma^n\}_{i,j\in I}$ and a $\Gamma^{\Whitney}$-cocycle $\{\delta_{ij}: U_i \cap U_j \to \Gamma^{\Whitney}\}_{i,j\in I}$ defined on the same cover $\{U_i\}_{i \in I}$ of $M$, we say that $\{\delta_{ij}\}_{i,j\in I}$ \textbf{lifts} $\{\gamma_{ij}\}_{i,j\in I}$ if $j^r \delta_{ij} = \gamma_{ij}$ for all $i,j \in I$. 
\end{definition}
Hence to define what it means for a microbundle or principal $\Gamma^\Whitney$-bundle to lift a $J^r \Gamma$-cocycle, we can ask that the respective bundles can be defined by a $\Gamma^\Whitney$-cocycle lifting the given $J^r \Gamma$-cocycle. Analogously, we define when a microbundle or principal $\Gamma^\Whitney$-bundle lifts a principal $J^r \Gamma$-bundle.

We can also state when a microbundle lifts a principal $J^r\Gamma$-bundle without the use of cocycles by using instead \cref{lem:MbToPBGammaWhitney}, which constructs from a microbundle $E$, the associated principal $\Gamma^{\Whitney}$-bundle $Q \to M$. By taking $r$-jets of the elements in the total space of $Q$ we obtain a principal $J^r \Gamma^n$-bundle, which we denote by $J^r Q \to M$. We refer to $J^r Q$ as the principal $J^rQ$-bundle associated to $E$, since they define the same equivalence class of $J^r\Gamma^n$-cocycles. Explicitly, we define $J^rQ$ as follows.
\begin{definition}\label{def:MbToJrGammaPb}
    Fix $r\in\N$ and let $E\to M$ be a microbundle. The associated principal $J^r\Gamma^n$-bundle $J^rQ \to M$ consists of the fiber bundle
    \[ J^rQ= \left\{(m,j^r_{\iota(m)}\phi) \;\middle|\; 
        \begin{varwidth}{0.6\linewidth}
        $m \in M$  and  $\phi: U \to \R^n$  an embedding of an open $U \subset E_m $  containing  $\iota(m)$ \end{varwidth} \right\}.  \]
    The bundle $J^rQ$ has moment map $\mu : J^rQ \to \R^n$ defined by $(m,j^r_{\iota(m)}\phi) \mapsto \phi(\iota(m))$. The action $J^r \Gamma^n \timeslr{\source}{\mu} J^rQ  \to J^rQ$ is defined by $j^r_{\phi(\iota(m))} f \cdot (m,j^r_{\iota(m)} \phi) \mapsto (m,j^r_{\iota(m)} (f \circ \phi))$. 
\end{definition}

We obtain the following necessary and sufficient condition for a microbundle to lift a principal $J^r\Gamma^n$-bundle. It is equivalent to reducing to cocycles, under the one-to-one correspondence between $J^r\Gamma^n$-cocycles and principal $J^r\Gamma^n$-bundles from \cref{pro:corCocyclePrBundle}.

\begin{lemma} \label{lem:MbLiftsPb}
    Fix $r\in\N$. Let $P \to M$ a principal $J^r \Gamma^n$-bundle and let $E \to M$ be a microbundle with associated principal $J^r \Gamma^n$-bundle $J^r Q \to M$. Then, the microbundle $E$ lifts $P$, if and only if $J^rQ$ is isomorphic to $P$ as a principal $J^r\Gamma^n$-bundle.
\end{lemma}

In the remainder of this section we prove the following statement.
\begin{theorem}\label{lem:JrGPbToMb}
    Fix $r\in\N_{>0}$. Let $M$ be a manifold and $P \to M $ a principal $J^r \Gamma^n$-bundle. Then there exists a microbundle $E \to M$ that lifts $P \to M$. Moreover, this microbundle is unique up to concordance.  
\end{theorem} 

\begin{remark}
    The smooth version of the Kister-Mazur theorem~\cite{kister1964microbundles,siebenmann1977stable} tells us that every microbundle of rank $n$ is equivalent to a vector bundle of rank $n$, which is unique up to isomorphism of vector bundles. The latter are in one-to-one correspondence with principal $\GL(n)$-bundles. We observe that such bundles can be interpreted as the linearizations of principal $J^1\Gamma$-bundles, which are affine in nature, by associating to a jet its differential. Under this observation, \cref{lem:JrGPbToMb} can be seen as an inverse to this association of principal $\GL(n)$-bundles to microbundles, since it lifts principal $J^r\Gamma$-bundles, and in particular principal $J^1\Gamma$-bundles, to microbundles.
\end{remark}

We start with the local statement, where the $J^r \Gamma^n$-bundle is trivial (recall \cref{ex:trivPb}).

\begin{lemma} \label{lem:TrivJrGPbtoMb}
    Fix $r\in \N$. Let $M$ be a manifold and $f: M \to \R^n$ a smooth map. The trivial principal $J^r\Gamma^n$-bundle $P_f \to M$ lifts to a microbundle $E\to M$.
\end{lemma}
\begin{proof}
    We define the microbundle $E = M \times \R^n \to M$ with the inclusion $\iota: M \to E$ defined by $x \mapsto (x,f(x))$. Hence $E$ is defined by the $\Gamma^{\Whitney}$-cocycle consisting of the single map $\delta:M \to \Gamma^{\Whitney}$ defined by $x\mapsto [\id]_{f(x)}$. On the other hand, we see that $P_f$ is defined by the $J^r\Gamma^n$-cocycle consisting of the single map $\gamma:M \to J^r\Gamma^n$ defined by $x\mapsto j^r_{f(x)}\id$. Hence we see that $\delta$ lifts $\gamma$.
\end{proof}

Next we assume that we are given two microbundles $E_i \to U_i$ that lift the respective restrictions $P|_{U_i} \to U_i$ of a $J^r \Gamma^n$-bundle $P \to M$, where the $U_i$ are opens in $M$. We then want to conclude that $P$ lifts on the whole of $M$, for which we reason as follows. If $U_1 \cap U_2 \neq \emptyset$, we have an isomorphism of the principal $J^r \Gamma^n$-bundles $P|_{U_1} $ and $P|_{U_2}$ over the intersection $U_1\cap U_2$. This provides us with a family of $r$-jets between fibers of $E_1$ and $E_2$ when restricted to $U_1 \cap U_2$. To glue the two microbundles to a single microbundle $E \to U_1 \cup U_2$, we need to lift the family of $r$-jets to a family of fiberwise diffeomorphisms, establishing the equivalence of $E_1$ and $E_2$ restricted to $U_1\cap U_2$. To this end, we introduce some notation, and prove that such lifts exist.

\begin{definition}
    Let $(E_1,\pi_1,\iota_1)$ and $(E_2,\pi_2,\iota_2)$ be two microbundles over a manifold $M$. The fiber bundle $J^r\Gamma^n_{\mathrm{fib}}(E_1,E_2) \to M$ over $M$ is the fiber bundle whose fiber over $m \in M$ consists of elements $j^r_{\iota_1(m)}f$, where $f$ is an embedding into $(E_2)_m$ of a neighborhood $U \subset (E_1)_m$ of $\iota_1(m)$ and where $f$ satisfies $f(\iota_1(m))=\iota_2(m)$.
\end{definition}

\begin{lemma}\label{lem:liftJetMbtoGermMb}
    Fix $r\in\N_{>0}$. Let $E_1$ and $E_2$ be representatives of two microbundles over a manifold $M$ of the same rank and let $f: M \to J^r\Gamma_{\mathrm{fib}}(E_1,E_2)$ be a section. Then there exists a family of maps $F: (E_1)_m \to (E_2)_m$ such that $j^r_{\iota_1(m)} F = f(m)$ and such that $(x,y)\mapsto(x,F(x)(y))$ is an embedding $E_1 \to E_2$. The map $F$ is unique up to homotopy. 
\end{lemma}
\begin{proof}
    Let $\{U_i\}_{i\in I}$ be a cover by small enough opens of $M$, such that we can use the local triviality of $E_1$ and $E_2$ to describe the restriction $f|_{U_i}$ as a map $U_i \to J^r\Gamma_{\mathrm{fib}}(\R^v,\R^v)$. Here $v$ is the rank of the fiber bundles $E_1$ and $E_2$. Then, by taking polynomial representatives, we find an $U_i$-family of embeddings $\tilde F_i(m): \R^v\to\R^v$ such that $j_{\iota_1(m)}^rF(m) = f(m)$ for all $m\in U_i$ and such that $(x,y) \mapsto(x,F(x)(y))$ is an embedding of $U_i \times \R^v \to U_i \times \R^v$. Using a partition of unity $\{\eta_i\}_{i\in I}$ subordinate to the cover $\{U_i\}_{i\in I}$ we glue the maps $F_i$ to the requested map $F \coloneqq \sum_{i \in I} \eta_i F_i$. Since we have $j_{\iota_1(m)} F_i = f(m)$ for all $i \in I$, we obtain that $j_{\iota_1(m)} F = f(m)$.

    The uniqueness of $F$ up to homotopy follows since the choice of $\tilde F_i$ is unique up to homotopy.
\end{proof}
We note that the map $F$ from \cref{lem:liftJetMbtoGermMb} shows in particular that the microbundles $E_1$ and $E_2$ are equivalent. With this lemma we now implement the lifting of a principal bundle in the setting we just discussed. 
\begin{lemma} \label{lem:CombineJrGPbtoMb}
    Fix $r\in\N_{>0}$. Let $M$ be a manifold and $U_1,U_2$ two open subsets of $M$ such that $M = U_1 \cup U_2$. Let $P \to M $ be a principal $J^r \Gamma^n$-bundle and assume we have lifted $P_i \coloneqq P|_{U_i} \to U_i $ to microbundles $E_i \to U_i$ respectively. Then there exists a microbundle $E \to M$ that lifts $P \to M$. 
\end{lemma}
\begin{proof}
    Since the bundles $P_i\to U_i$ are lifted to microbundles $E_i \to U_i$, there exist by \cref{lem:MbLiftsPb} principal $\Gamma^{\Whitney}$-bundles $Q_i \to U_i$ such that $J^rQ_i$ and $P_i$ are isomorphic as principal $J^r \Gamma^n$-bundles. Moreover, we know that the bundles $P_1\to U_1$ and $P_2\to U_2$ together form the bundle $P \to M$, and hence they are isomorphic when restricted to $U_1 \cap U_2$. Combining these two observations, we know there exists an isomorphism $F : J^r Q_1|_{U_1 \cap U_2} \to J^rQ_2|_{U_1 \cap U_2}$ of principal $J^r \Gamma^n$-bundles that lifts the identity map over the intersection $U_1 \cap U_2$. 

    We recall from \cref{def:MbToJrGammaPb} that $J^rQ_i \to U_i$ consists of  
    \[J^r Q_i = \left\{(m,j^r_{\iota_i(m)} \phi) \;\middle|\; 
        \begin{varwidth}{0.6\linewidth}
        $m \in U_i$  and $\phi: U \to \R^n$ an embedding of an open  $U \subset (E_i)_m$  containing $ \iota_i(m) $
        \end{varwidth}
    \right\}.\]

    Consider a point $m \in U_1 \cap U_2$ and an element $(m,j^r_{\iota_1(m)} \phi_1) \in J^r Q_1$. We write the element $F(m,j^r_{\iota_1(m)} \phi_1) \in J^r Q_2$ as $(m,j^r_{\iota_2(m)} \phi_2)$. Since $F$ respects the moment maps we note that we have $\phi_1(\iota_1(m)) = \phi_2(\iota_2(m))$. Hence we obtain that $j^r_{\iota_1(m)} (\phi_2^{-1} \circ \phi_1) $ is an element of $ J^r\Gamma^n((E_1)_m,(E_2)_m)$ which sends $\iota_1(m)$ to $\iota_2(m)$. Moreover, the map $j^r_{\iota_1(m)} (\phi_2^{-1} \circ \phi_1) $ does not depend on the initial choice of $j^r_{\iota_1(m)} \phi_1$ in the fiber of $J^r Q_1$ over $m$, since $F$ is invariant under the action of $J^r\Gamma^n$. Hence we obtain a section $\Phi: U_1 \cap U_2 \to J^r\Gamma^n_{\mathrm{fib}}(E_1,E_2)$, of which we now show its smoothness.

    Since $\pi_1: J^rQ_1 \to U_1$ is a surjective submersion, it admits for all $m \in U_1 \cap U_2$ a (smooth) section $\sigma_m$ defined on a neighborhood $V_m$ of $m$. The section $\sigma_m$ then induces as above a local section $\Phi_m : V_m \to J^r\Gamma^n_{\mathrm{fib}}(E_1,E_2)$, which is smooth because $\sigma$ is smooth. Since $\Phi_m(x)$ does not depend on $\sigma(x)$, it moreover follows that $\Phi_m = \Phi|_{V_m}$. Hence the section $\Phi$ is smooth.

    By \cref{lem:liftJetMbtoGermMb} it follows that $E_1|_{U_1\cap U_2}$ and $E_2|_{U_1\cap U_2}$ are equivalent microbundles, showing that they glue to a microbundle $E \to M$.
\end{proof}

We now combine \cref{lem:TrivJrGPbtoMb,lem:CombineJrGPbtoMb} to show that any principal $J^r\Gamma^n$-bundle lifts to a microbundle. Since \cref{lem:CombineJrGPbtoMb} is relative in nature, we also obtain that the lift is unique up to concordance. Hence we obtain a proof of \cref{lem:JrGPbToMb}.
\begin{proof}[Proof of \cref{lem:JrGPbToMb}]
    We first assume that $M$ is compact and find a finite covering $\{U_i\}_{i = 1,\dots,N}$ by opens such that $P|_{U_i}$ is trivial. Using induction on $i \in I$ we lift $P$ to a microbundle when restricted to $\cup_{j \leq i} U_j$. That is, for the base case we lift $P|_{U_1}$ to a microbundle using \cref{lem:TrivJrGPbtoMb}. Assuming we have already lifted $P$ over $\cup_{j \leq i-1} U_j$, we lift $P|_{U_i}$ to a microbundle using \cref{lem:TrivJrGPbtoMb} and glue the two microbundles to form a microbundle over $\cup_{j \leq i} U_j$ using \cref{lem:CombineJrGPbtoMb}. When the induction ends, we obtain the requested microbundle over $M$.

    To show that the constructed microbundle $E \to M$ is unique up to concordance, we assume that we are given two microbundles $E \to M$ and $\tilde E\to M$ that both lift $M$. Our goal is then to construct a concordance between $E$ and $\tilde E$. Hence we consider the constant concordance $\overline P \to M \times [0,1]$ of principal $J^r \Gamma^n$-bundles between $P$ and itself. We have lifts $E$ and $\tilde E$ of $\overline P$ over respectively $M \times \{0\}$ and $M \times \{1\}$, which we extend collar neighborhoods of $M \times \{0\}$ and $M \times \{1\}$. We now apply the above reasoning to $\overline P$, relative to the lifts over $M \times \{0,1\}$.

    If $M$ is not compact, we reduce to the compact case using an exhaustion by compacts as detailed in the proof of \cref{pro:smoothCocycle}.
\end{proof}
\section{Groupoids for \texorpdfstring{$\Diff$}{Diff}-invariant relations} \label{sec:Rgroupoids}
Our results about the classification of principal $\GammaR$-bundles are proven in \cref{sec:wrinklingHaefliger}. To this end, we now adapt the classic material from \cref{sec:prelimGroupoids} to the more general setup of arbitrary $\Diff$-invariant differential relations. 

We start in \cref{sec:prelimFolTransvStr} by introducing $\SR$-foliations, which are foliations with a transverse geometry modeled on a $\Diff$-invariant differential relation $\SR$. We construct a groupoid $\GammaR$ describing such $\SR$-foliations. We prove the Morita equivalence of two different models of $\GammaR$ in \cref{sec:MEGammaR}. In \cref{ssec:RMb} we endow microbundles with an $\SR$-foliation; these objects will be the singular analogues of $\SR$-foliations. We then  show how these relate to principal $\GammaR$-bundles. We summarize the correspondences of the various notions in \cref{ssec:summaryRBundles}. We end with tautological $\GammaR$-bundles in \cref{sec:tautBundle}, which provide a useful tool for proving the results in \cref{sec:wrinklingHaefliger}.

\subsection{Foliations with transverse structure}\label{sec:prelimFolTransvStr}
From this point onward we denote by $\SR$ an open and $\Diff$-invariant differential relation. We saw in \cref{ssec:diffInvariant} that a $\Diff$-invariant differential relation $\SR$ of dimension $n$ is a certain functor with domain the category of $n$-dimensional manifolds. In non-commutative geometry, one thinks of leaf spaces of corank $n$ foliations as generalizations of manifolds. Hence, morally speaking, the following definitions say that we can extend the domain of $\SR$ to the category of leaf spaces.

We start with two equivalent definitions of transverse $\SR$-structures on a given foliation of the manifold. Next, we note that foliations together with an $\SR$-transverse structure are just a specific case of the $\SG$-foliations from \cref{sec:GFol}.

\begin{definition}\label{def:RTransStrCocycle}
    Let $M$ be an $m$-dimensional manifold endowed with a foliation $\SF$ of corank $n$ defined by a maximal collection of 
    \begin{itemize}
        \item opens $\{U_i\}_{i \in I}$ covering $M$,
        \item submersions $\phi_i: U_i \to \R^n$, for each $i \in I$, with $\ker(d\phi_i) = T\SF$, and
        \item transition functions $\rho_{ij}: \phi_j(U_i \cap U_j) \to \phi_i(U_i \cap U_j)$ satisfying the cocycle condition $\rho_{ij} \rho_{jk} = \rho_{ik}$, where $\rho_{ii}=\phi_i$ for all $i,j,k \in I$.
    \end{itemize} An \textbf{$\SR$-transverse structure} is a collection $\{f_i : \phi_i(U_i) \rightarrow \Psi(\phi_i(U_i))\}_{i \in I}$ of solutions of $\SR$ satisfying $(\rho_{ji})_* f_i = f_j$ for all $i,j \in I$. Here $(\rho_{ji})_*$ denotes the action by pushforward of local diffeomorphisms on sections of $\Psi$.
\end{definition}

The $\SR$-transverse structure relies on the $\Diff$-invariance of $\SR$ to define the $f_i$ and the action of the $\rho_{ij}$ on the $f_j$. We interpret it as saying that the local solutions $f_i$ glue to a global solution $F$ of $\SR$ on the leaf space of $\SF$. The pair $(\SF,F)$ is said to be an \textbf{$\SR$-foliation}.

An alternative method of encoding an $\SR$-transverse structure is the following, where we formulate it as solutions of $\SR$ on transversals to the foliation, invariant under holonomy.

\begin{lemma}\label{def:RTransStrTransversal}
    Let $M$ be an $m$-dimensional manifold endowed with a foliation $\SF$ of corank $n$. An $\SR$-transverse structure on $\SF$ is equivalent to a collection consisting of 
    \begin{itemize}
        \item all transversals $T$ to $\SF$ of dimension $n$, and
        \item solutions $F_T: T \to \SR(T)$ such that the $F_T$ are invariant under holonomy. That is, we require that for any $x\in T$ and $[f]_x \in \Hol(\SF)$ we have $[f_*( F_{T}) ]_{x} = [F_{T'}]_{f(x)}$ for any transversal $T'$ containing $f(x)$.
    \end{itemize}
\end{lemma}
\begin{proof} 
    Let an $\SR$-foliation $(\SF,F)$ on a manifold $M$ be given as in \cref{def:RTransStrCocycle}. Let $T_i\subset U_i$ be an $n$-dimensional transversal to $\SF$ such that $\phi_i(T_i) = \phi_i(U_i)$. We define the solution $F_{T_i} : T_i \to \SR(T_i)$ by $F_{T_i}=(\phi_i|_{T_i})^* f_i$. To see that the solutions are invariant under holonomy, we consider a path $\alpha$. If we write $\alpha$ as the concatenation of paths contained in subsequently $U_{k_1},\dots,U_{k_{\ell-1}}$ and $U_{k_\ell}$, we obtain that the induced holonomy by $\alpha$ is $(\phi_{k_\ell}|_{T_{k_\ell}})^{-1} \cdot \rho_{k_\ell k_{\ell-1}} \cdots \rho_{k_2 k_1} \cdot \phi_{k_1}|_{T_{k_1}}$. Using the cocycle condition on the $\rho_{ij}$ and that $(\rho_{ji})_* f_i = f_j$, we obtain that the solutions are holonomy invariant.

    For the other direction, we assume we are given a collection of $n$-dimensional transversals to the foliation $\SF$ and solutions on those transversals. We then find submersions $\phi_i : U_i \to \R^n$ and transition functions $\rho_{ij}$ defining the underlying foliation $\SF$. Let $T_i \subset U_i$ be a transversal such that $\phi(T_i) = \phi(U_i)$ and let $F_{T_i}: T_i \to \SR(T_i)$ be the accompanying solution. We define solutions $f_i : \phi(U_i) \to \Psi(\phi_i(U_i))$ by $f_i = (\phi|_{T_i})_* F_{T_i}$. Since the solutions $f_i$ are invariant under holonomy and the transition functions $\rho_{ij}$ represent the holonomy, we obtain that $(\rho_{ji})_* f_i = f_j$.
\end{proof}
Alternatively, we describe an $\SR$-foliation using a groupoid cocycle, where the relevant groupoid is the following. We recall that the definition of the action groupoid can be found in \cref{ex:actionGroupoid}.
\begin{definition} \label{def:GammaR}
    Let $\Gamma^n$ act on $\EtSol{\R^n}$ along the projection $p_b: \EtSol{\R^n} \to \R^n$ as $[f]_x \cdot [s]_x \mapsto [f_*s]_x$ where $[f]_x \in \Gamma^n$ and $[s]_x \in \EtSol{\R^n}$. The \textbf{classifying groupoid} $\GammaR$ associated to $\SR$ is the action groupoid $\Gamma^n \ltimes \EtSol{\R^n} \rightrightarrows \EtSol{\R^n}$.
\end{definition}

The base is a possibly non-Hausdorff and non-second-countable $n$-manifold, which itself submerses onto $\R^n$. By construction, the base encodes all possible germs of solutions of $\SR$ over $\R^n$ and the arrows encode all the possible symmetries. Note that $\GammaR$ is \'etale and effective.

Hence we obtain the following manner of defining an $\SR$-foliation.  
\begin{lemma}
    Let $M$ be a manifold. There is a one-to-one correspondence between $\SR$-foliations on $M$ and equivalence classes of $\GammaR$-foliations on $M$.
\end{lemma}
\begin{proof} 
    Let an $\SR$-foliation $(\SF,F)$ on a manifold $M$ be given as in \cref{def:RTransStrCocycle}. We then define the $\GammaR$-cocycle $\{\gamma_{ij}: U_i \cap U_j \to \GammaR\}$ by $\gamma_{ij}(x) = ([\rho_{ij}]_{\phi_j(x)} , [f_j]_{\phi_j(x)})$. We see that $\gamma_{ii} = ([\id]_{\phi_i(x)},[f_i]_{\phi_i(x)})$ which is a submersion as map into the base $\EtSol{\R^n}$ of $\GammaR$. 

    Vice versa, we assume we are given a $\GammaR$-cocycle $\{\gamma_{ij}: U_i \cap U_j \to \GammaR\}$ such that $\source\circ \gamma_{ii}$ is submersive. We define the submersions 
    $\phi_i: U_i \to \R^n$ by \[\phi_i = p_b \circ \source \circ \gamma_{ii}. \]
    If the opens $U_i$ are small enough, the map $\pr_1 \circ \gamma_{ij}$ is contained in a basic open of $\Gamma^n$. Hence we can find a map $\rho_{ij}: \phi_j(U_i \cap U_j) \to \phi_i(U_i \cap U_j)$ such that 
    \[ [\rho_{ij}]_x = \pr_1 \circ \gamma_{ij}(x) \in \Gamma^n, \]
    for all $x \in U_i$. 
    Similarly, if the opens $U_i$ are small enough, also the map $\source \circ \gamma_{ii}$ is contained in a basic open of $\EtSol{\R^n}$. Hence we can find a map $f_i: \phi_i(U_i) \rightarrow \Psi(\phi_i(U_i))$ such that 
    \[ [f_i]_x = \source \circ \gamma_{ii}(x) \in \EtSol{\R^n}, \]
        for all $x \in U_i$. 
    We leave it to the reader to extend the above collection to a maximal one, and to then check the conditions of \cref{def:RTransStrCocycle}. 
\end{proof}

We recall from \cref{cor:corGFolSubmPb} that an equivalence class of $\GammaR$-foliations corresponds to an isomorphism class of principal $\GammaR$-bundles with submersive moment map.   
\begin{corollary}
    Let $M$ be a manifold. There is a one-to-one correspondence between $\SR$-foliations on $M$ and isomorphism classes of principal $\GammaR$-bundles over $M$ with submersive moment map.
\end{corollary}

\subsection{Morita equivalence between \texorpdfstring{$\GammaR$}{Γ\_R} and \texorpdfstring{$\GammaR^N$}{Γ\^{}N\_R}} \label{sec:MEGammaR}

In \cref{def:GammaR} we defined the groupoid $\GammaR$ as the action groupoid $\Gamma^n \ltimes \EtSol{\R^n}$. We can however also consider the action groupoid $\GammaRM{N} \coloneqq \GammaM{N} \ltimes \EtSol{N} $ where $N$ is an $n$-manifold. We claim that these two groupoids are Morita equivalent.
\begin{lemma} \label{lem:MEGammaRNn}
    Let $N$ be an $n$-manifold. The groupoids $\GammaRM{N}$ and $\GammaR$ are Morita equivalent.
\end{lemma}
\begin{proof}
    We recall that $\Gamma(N, \R^n)$ is the space of germs of local diffeomorphisms $f:U \to f(U) \subset \R^n$ where $U\subset N$ is an open subset. On one hand, the space $\Gamma(N, \R^n) \times_{N} \EtSol{N} $ is a left principal $\GammaR$-bundle over $\EtSol{N}$ along the moment map 
    \begin{align*}
    \mu: \Gamma(N, \R^n) \times_{N} \EtSol{N} &\to \EtSol{\R^n} \\
    ([\psi]_u,[t]_u) &\mapsto [\psi_* t]_{\psi(u)}.
    \end{align*}
    The action of $\GammaR$ on $\Gamma(N, \R^n) \times_{N} \EtSol{N}$ is given by 
    \begin{align*}
        \GammaR \timeslr{\source}{\mu} \left( \Gamma(N, \R^n) \times_{N} \EtSol{N} \right)   &\to \Gamma(N, \R^n) \times_{N} \EtSol{N} \\
       ([f]_{\psi(u)},[\psi_* t]_{\psi(u)}) \cdot ([\psi]_u , [t]_u)   &\mapsto ([f \circ \psi]_u ,[t]_u ).
    \end{align*} 
    
    On the other hand $\Gamma(N, \R^n) \times_{N} \EtSol{N}$ is a right principal $\GammaRM{N}$-bundle over $\EtSol{\R^n}$ along the moment map 
    \begin{align*}
        \mu_N: \Gamma(N, \R^n) \times_{N} \EtSol{N} &\to \EtSol{N} \\
        ([\psi]_u,[t]_u,) &\mapsto [t]_u.
    \end{align*}
    The action of $\GammaRM{N}$ on $\Gamma(N, \R^n) \times_{N} \EtSol{N}$ is given by 
    \begin{align*}
         \GammaRM{N} \timeslr{\source}{\mu_N} \left( \Gamma(N, \R^n) \times_{N} \EtSol{N} \right)  &\to \Gamma(N, \R^n) \times_{N} \EtSol{N} \\
        ([g]_u, [t]_u, ) \cdot ([\psi]_u, [t]_u)   &\mapsto ([\phi \circ g ]_{h^{-1}(u)},[g^* t]_{g^{-1}(u)} ).
    \end{align*} 
    Hence $\Gamma(N, \R^n) \times_{N} \EtSol{N}$ forms a $\GammaR$-$\GammaRM{N}$-bibundle.
\end{proof}

\subsection{\texorpdfstring{$\SR$}{R}-microbundles}\label{ssec:RMb}
As a combination of \cref{sec:prelim2Haefliger,sec:prelimFolTransvStr}, we now define $\SR$-microbundles. We think of $\SR$-microbundles as the singular analogues of $\SR$-foliations, just as a foliated microbundle is the singular analogue of a foliation.
\begin{definition} \label{def:Rmb}
	An \textbf{$\SR$-microbundle} over an $m$-manifold $M$ is a microbundle of rank $n$ with an $\SR$-foliation of corank $n$ that is transverse to the fibers of $E$. 
\end{definition}

We discuss two examples to see how $\SR$-microbundles naturally appear.

\begin{example} 
Let $f: M \to \Psi$ be a solution of $\SR$. Given any submersion $g: N \to M$, we can foliate $N$ by (connected components of) fibers of $g$ and pullback $f$ to a transverse $\SR$-structure $g^*f$. This endows $N$ with an $\SR$-foliation.

More generally, if $(\SF,F)$ is an $\SR$-foliation on $M$ and $g: N \to M$ is a map transverse to $\SF$, we can pullback both $\SF$ and $F$, yielding another another $\SR$-foliation $(g^*\SF,g^*F)$ on $N$.
\end{example}

\begin{example} \label{ex:SolToBundle} 
An important example, already encountered in the introduction to motivate $\SR$-microbundles, is produced by the exponential map $\exp: TM \rightarrow M$. This map is in general not a submersion due to the presence of the cut locus, but we can restrict $\exp$ to a sufficiently small neighborhood $N \subset TM$ of the zero section. Then any solution $f$ of $\SR$ over $M$ pulls back to an $\SR$-microbundle over $N$ that we call $\exp(f)$.
\end{example}

\subsubsection{\texorpdfstring{$\SR$}{R}-microbundles as principal groupoid bundles}
We want to relate $\SR$-microbundles to principal groupoid bundles, analogously to \cref{sec:CorPbMb} where we discussed how $\SG$-microbundles are related to $\SG$-cocycles. Since the groupoid $\GammaR$ is effective, it follows from \cref{lem:corPGB-Mb} that there is a one-to-one correspondence between $\GammaR$-cocycles and $\GammaR$-microbundles. We recall that the latter is a microbundle whose structure groupoid is $\GammaR$ and whose fibers are therefore canonically diffeomorphic to opens in $\EtSol{N}$. By projecting local diffeomorphisms of $\EtSol{N}$ to local diffeomorphisms of $\R^n$, we claim that the correspondence extends to $\SR$-microbundles. We implement this by adapting the proof of \cref{lem:corPGB-Mb}.

\begin{lemma} \label{lem:corPGRB-Rmb}
    Let $M$ be a manifold. There is a one-to-one correspondence between equivalence classes of $\GammaR$-cocycles and $\SR$-microbundles over $M$.
\end{lemma}
\begin{proof}
    Let $\gamma_{ij}: U_i \cap U_j \to \GammaR$ be a $\GammaR$-cocycle. We recall from the proof of \cref{lem:corPGB-Mb} that we can take representatives to find
    \begin{itemize}
        \item opens $V_i \subset \EtSol{\R^n}$ such that $g_i = \source \circ \gamma_{ii} : U_i \to V_i$, 
        \item diffeomorphisms $g_{ij}: V_j \rightarrow V_i$ with $[g_{ij}]_{\gamma_{jj}(x)} = \gamma_{ij}(x)$ for all $x\in U_j$,
    \end{itemize}
    such that $g_{ij} \circ g_j = g_i$. 

    We assume that the neighborhoods $V_i$ are small enough so that the projection $p: \EtSol{\R^n} \to \R^n$ restricted to each of the $V_i$ is a diffeomorphism. If we define $A_i \coloneqq p(V_i) \subset \R^n$, the $g_{ij}$ lift diffeomorphisms $h_{ij} : A_i \to A_j$. 

    We now construct the microbundle as in the proof of \cref{lem:corPGB-Mb}. That is, we define $E = \sqcup_i (U_i \times A_i) / \sim$. The foliation is then locally given by $U_i \times \{t\}$ for $t \in A_i$. The $\SR$-transverse structure is defined on each transversal $\{x\} \times A_i$ by the pushforward of the tautological solution $\tau$ on $\EtSol{\R^n}$ along $p|_{V_i}$. 
    
    For the other direction, we start with a representative of an $\SR$-microbundle. We let $\{U_i\}_{i \in I}$ be a finite open covering of $M$ such that the underlying foliated microbundle $E \to M$ trivializes over each $U_i$ as in \cref{lem:GMbAtlas}. Hence for each $i \in I$ there exists a diffeomorphism $\phi_i: V_i \to U_i \times \R^n$ where $V_i \subset E$ is an open neighborhood of $\iota(U_i)$ and $\phi$ maps the leaves of the foliation, restricted to $V_i$, to $U_i \times \{*\} $. Since the microbundle is endowed with a foliation, we know that the transition functions $\phi_{ij} : U_i \cap U_j \to \Diff_\loc(\R^n)$ are constant, which allows us to write $\phi_{ij} \in \Diff_\loc(\R^n)$. We already used this in \cref{cor:corGammaNMB-FolMB}. The $\SR$-transverse structure induces local solutions $F_i : \R^n \to \SR$ satisfying $(\phi_{ji})_* [F_i]_{\phi_i (\iota(x))} = [F_j]_{\phi_j (\iota(x))}$ for all $x \in U$. We now define the cocycle $\gamma_{ij} : U_i \cap U_j \to \GammaR$ by $\gamma_{ij}(x)=( [\phi_{ij}]_{\phi_j (\iota(x))} , [F_j]_{\phi_j (\iota(x))} )$. 
\end{proof}

\subsubsection{Space of \texorpdfstring{$\SR$}{R}-microbundles} \label{sec:spaceRMb}
Combining \cref{pro:corCocyclePrBundle,lem:corPGRB-Rmb} shows that we can associate a microbundle to every principal $\GammaR$-bundle. We recall that the smooth version of the Kister-Mazur theorem~\cite{kister1964microbundles,siebenmann1977stable} additionally states that every microbundle is equivalent (as a microbundle) to a vector bundle. Moreover, the vector bundle is unique up to isomorphism (of vector bundles). Hence we can speak of the vector bundle associated to a $\GammaR$-microbundle or principal $\GammaR$-bundle.

We interpret this more functorially as follows: There is a groupoid morphism 
\[ \GammaR \rightarrow \GL(n) \]
given by taking the differential of the germ of a diffeomorphism. By functoriality of the classifying space construction (recall \cref{sec:prelim2classspace}), this yields a map
\[ \eta: B\GammaR \rightarrow B\GL(n) \]
with codomain the classifying space of the underlying vector bundle. For each manifold $M$, this map specializes to
\[ \eta: \Maps(M,B\GammaR) \rightarrow \Maps(M,B\GL(n)). \]
The space on the right has multiple components, each corresponding to a concordance class of $\GL(n)$-bundles over $M$. We recall that concordance and isomorphism classes coincide for vector bundles.

If we are interested in $\SR$-microbundles with a fixed (up to isomorphism) underlying vector bundle $E$, we should take the homotopy fiber of the map $\eta$ over a map $\tau_E: M \to B\GL(n)$ defining the vector bundle $E$. We recall that this is the space consisting of pairs $(f,\gamma)$ with $f$ a map $M \to B\GammaR$, and $\gamma$ a homotopy in $\Maps(M, B\GL(n))$ between $\eta(f)$ and $\tau_E$. We denote the homotopy fiber $\hofib_{\tau_E}(\eta)$ by $\Maps_E(M,B\GammaR)$ and hence think of this space as the ``space of $\SR$-microbundle structures on $E$''. 

To illustrate this, we note that the map $\gamma$ above defines an isomorphism between $E$ and the vector bundle defined by $\eta(f)$, since homotopy classes and isomorphism classes coincide for vector bundles. Hence, an element in $\Maps_E(M,B\GammaR)$ defines an $\SR$-microbundle $(E',\SF,F)$ and an isomorphism $\tilde \gamma: E' \to E$. Using $\tilde \gamma$ we can pullback the $\SR$-foliation $(\SF,F)$ on $E'$ to an $\SR$-foliation on $E$. It follows moreover that: 
\begin{corollary}\label{cor:representabilityMicro}
    Any map $N \rightarrow \Maps_E(M,B\GammaR)$ can be represented by an $\SR$-microbundle on $N \times M$ with underlying vector bundle $\pr_N^*( E)$, where $\pr_N$ is the projection $M \times N \to N$.
\end{corollary} 

It is worth pointing out that $\Maps_E(M,B\GammaR)$ is in general not weakly equivalent to the fiber over $\tau_E$ in $\Maps(M,B\GammaR)$. We give an example below. 
\begin{example}
Consider $M = *$ and as transverse relation $\SR$ an orientation, which corresponds to choosing an orientation of the bundle $\R^n \to *$. Fix now a map $\tau: * \to B\GL(n)$. We endow the bundle $E$ on $*$ induced by $\tau$ with a positive or negative orientation, which we denote by $(E,+)$ and $(E,-)$ respectively. Choosing such an orientation corresponds to taking lifts $f_+,f_-: * \to B\GammaR$ of $\tau$. The bundles $(E,+)$ and $(E,-)$ are concordant as ends of the oriented and half-twisted strip $I \times \R^n$, which corresponds to a map $H : I \to B\GammaR$. However, the map $H$ composed with $\eta$ is not constant in $B\GL(n)$. Therefore, $f_+$ and $f_-$ are homotopic as elements of $\Maps_E(M,B\GammaR)$, but not as elements in the fiber of $\eta$ over $E$.
\end{example}

Note that the above example does not contradict \cref{cor:representabilityMicro}, since in the above example $H$ does not factor through $* \to B\Gamma_R$. Maps that \emph{do} factor like this have classically been an object of study.
\begin{remark} \label{rem:ClassFramedPB}
When $M$ is a point, there is a unique vector bundle $E$ over it. In this case, the space $\Maps_E(M,B\GammaR)$ is known as the classifying space $B\overline{\GammaR}$ of \emph{framed} principal $\GammaR$-bundles, because a map $N \to B\overline{\GammaR}$ corresponds to a principal $\GammaR$-bundle over $N$ with trivial underlying bundle. For the groupoid $\Gamma^n$ of diffeomorphisms of $\R^n$, they were studied by Mather~\cite{Mather}, Thurston~\cite{Th3}, and Haefliger~\cite{Haef1}, among others.
\end{remark}

\subsection{Summary} \label{ssec:summaryRBundles}
Analogously to \cref{ssec:summaryBundles}, we now summarize the correspondences between the various constructions involving $\GammaR$. All $\leftrightarrow$ indicate a one-to-one correspondence. This correspondence is up to isomorphism for principal bundles and up to equivalence in the case of cocycles:
\[
    \begin{tikzcd}[column sep={1em},row sep = tiny]
        \{ \GammaR \text{ - cocycles} \} \arrow[r,leftrightarrow] & \{  \text{principal $\GammaR$-bundles} \} \arrow[r,leftrightarrow] & \{ \GammaR \text{-microbundles} \} \arrow[r,leftrightarrow] & \{ \text{$\SR$-microbundles} \}.
    \end{tikzcd}
\]

\subsection{Tautological \texorpdfstring{$\GammaR$}{Γ\_R}-cocycle}
\label{sec:tautBundle}
In \cref{sec:tautsol} we already defined the tautological solution $\tau : \EtSol{N} \to \Psi$, where $N$ is an $n$-dimensional manifold and $\SR$ a differential relation of dimension $n$. Here we now also introduce the tautological $\GammaR$-\emph{cocycle} which is used in some of our upcoming proofs in \cref{sec:wrinklingHaefliger}.

\subsubsection{The tautological \texorpdfstring{$\GammaR$}{Γ\_R}-cocycle} \label{sec:tautGammaRCocycle}
The groupoid $\GammaRM{N}$ (from \cref{sec:MEGammaR}) is a groupoid with base space $\EtSol{N}$ and hence we can consider the unit $\GammaRM{N}$-cocycle over $\EtSol{N}$. Explicitly, this cocycle is given by $[s]_x\mapsto ([\id]_x,[s]_x)$. 

We obtain a \textbf{tautological $\GammaR$-cocycle} $\{\tautcoij\}_{i,j \in I}$ over $\EtSol{N}$ by using an atlas of $N$ to make the $\GammaR^N$-cocycle into a $\GammaR$-cocycle. That is, we fix an atlas $\{\phi_i: U_i \to \R^n\}_{i \in I}$ of $N$ with transition functions $\phi_{ij} = \phi_i \circ \phi_j^{-1}:\phi_j(U_i \cap U_j) \to \phi_i(U_i \cap U_j)$ and we define the opens $V_i = (p_b)^{-1}(U_i) \subset \EtSol{N}$. Then the cocycle $\{\tautcoij: V_i \cap V_j \to \GammaR \}_{i,j \in I}$ on $\EtSol{N}$ is defined by
\begin{align*}
    \tautcoij([s]_x) = \left( \left[ \phi_{ij} \right]_{\phi_j(x)}, \left( \phi_j \right)_* \left( \left[ s \right]_{x} \right) \right) 
\end{align*}
for all $[s]_x \in V_i \cap V_j$. Had we fixed a different atlas of $N$, we would have obtained a different, but equivalent cocycle. Hence without fixing an atlas of $N$, it is only the equivalence class of tautological $\GammaR$-cocycles that is defined. 

\begin{remark}
    Even though we have defined the tautological $\GammaR$-cocycles over $\EtSol{N}$, we do not mention a tautological $\GammaR$-microbundle or tautological principal $\GammaR$-bundle. The reason for this lies in the fact that the space $\EtSol{N}$ is neither Hausdorff, nor second-countable and that hence \cref{pro:corCocyclePrBundle,lem:corPGRB-Rmb} do not apply. In particular, we cannot use \cref{pro:MePrBundlesSame} and the Morita equivalence from \cref{lem:MEGammaRNn} to convert the unit principal $\GammaR^N$-bundle over $\SR(N)$ to a principal $\GammaR$-bundle over $\SR(N)$. However, for our purposes of \cref{sec:wrinklingHaefliger}, the equivalence class of tautological $\GammaR$-cocycles suffices.
\end{remark}

\subsubsection{Pullbacks of the tautological \texorpdfstring{$\GammaR$}{Γ\_R}-cocycle} \label{ssec:PullbackTautCocycle}
In \cref{ex:SolToBundle} we discussed how to obtain the $\SR$-microbundle $\exp(f)$ from a solution $f: N \to \SR(N)$  by pulling back the solution using the exponential map. There exists also an alternative way of obtaining this microbundle using the tautological $\GammaR$-cocycle over $\EtSol{N}$. This claim holds only up to equivalence, since the tangent bundle is, as microbundle, only defined up to equivalence.
\begin{corollary} \label{cor:SolToBundleTaut}  
    Let $f: N \to \Psi(N)$ be a solution and define the map $F : N \to \EtSol{N}$ by $x \mapsto [f]_x$. The $\SR$-microbundle corresponding to $F^*\{\tautcoij\}$ is equivalent to the $\SR$-microbundle $\exp(f)$.
\end{corollary}
\begin{proof}   
    When we pullback the tautological $\GammaR$-cocycle $ \{\tautcoij\} $ using $F$ we obtain the $\GammaR$-cocycle $\{ \delta_{ij} : U_i \cap U_j \to \GammaR \}$ over $N$ defined by 
    \begin{align*}
        \delta_{ij}(x) = \left( \left[ \phi_{ij} \right]_{\phi_j(x)}, \left( \phi_j \right)_* \left( \left[ f \right]_{x} \right) \right). 
    \end{align*}
    Hence the underlying $\Gamma$-cocycle sends $x \in N$ to $\left[ \phi_{ij} \right]_{\phi_j(x)}$, which defines the tangent bundle foliated by the fibers of the exponential map as microbundle (recall \cref{ex:MBtangentbundle}). Since the solutions on the fibers of the bundles agree, the $\SR$-microbundles are equivalent.
\end{proof}

In fact, we can also use other maps into $\EtSol{N}$ to pull back the tautological $\GammaR$-cocycle, resulting in an $\SR$-microbundle with underlying bundle $TN$.
\begin{corollary} \label{cor:SectionToBundleTaut}
    Let $G: N \to \EtSol{N}$ be a map such that $p_b \circ G: N \to N$ is homotopic to the identity. The $\SR$-microbundle corresponding to $G^*\{\tautcoij\}$ is an $\SR$-microbundle, whose underlying microbundle is equivalent to $TN$.
\end{corollary}
\begin{proof}
    When we pullback the tautological $\GammaR$-cocycle using $G$ we obtain the $\GammaR$-cocycle $\{ \delta_{ij} : U_i \cap U_j \to \GammaR \}$ over $N$ defined by 
    \begin{align*}
        \delta_{ij}(x) = \left(  \left[ \phi_{ij} \right]_{\phi_j(p_b \circ G(x))} , \left( \phi_j \right)_* \left( G\left(x\right) \right) \right). 
    \end{align*}
    Hence the underlying $\Gamma$-cocycle sends $x \in N$ to $\left[ \phi_{ij} \right]_{\phi_j(p_b \circ G(x))}$, which defines the pullback of $TN$ via $p_b\circ G$. Since $p_b \circ G$ is homotopic to the identity, it follows that the microbundle is equivalent to $TN$.
\end{proof}

\section{The formal analogue} \label{ssec:formalAnalogue}
We now introduce the formal counterparts of principal $\GammaR$-bundles. The reader should keep in mind that this passage is analogous to the passage from solutions of $\SR$ to formal solutions. In \cref{sec:formalGroupoid} we introduce the underlying groupoid and in \cref{sec:formalMb} the formal $\SR$-microbundles. This leads us to define the scanning map in \cref{ssec:scanningMap}, which we use to compare the genuine and formal versions of the discussed structures. We end in \cref{sec:tautBundleF} with the formal analogues of the tautological bundles from \cref{sec:tautBundle}.

\subsection{Groupoid of formal solutions} \label{sec:formalGroupoid}
We start by defining the formal analogue of $\GammaR$, where we consider formal solutions instead of germs of solutions. These formal solutions are acted upon by $J^\aorder \Gamma$ where $\aorder$ is the action order of the relation $\SR$. We recall from \cref{ssec:natural} that this is the minimal number $\aorder$ such that $J^\aorder\Gamma$ acts on $\SR$.  

\begin{definition} \label{def:GammaRf}
    The \textbf{formal classifying groupoid} $\GammaR^f$ associated to $\SR$ is the action groupoid $ J^\aorder \Gamma^n \ltimes \SR(\R^n)  \rightrightarrows \SR(\R^n)$.
\end{definition}

We observe that $\GammaR^f$ is a Lie groupoid that is finite dimensional, Hausdorff and second-countable, but not \'etale. Additionally we note that elements of $\GammaR^f$ are of the form $(j^\aorder_x(f),j^r_x (s)) $ with $j^\aorder_x(f) \in J^\aorder \Gamma^n$ and $j^r_x (s) \in \SR(\R^n)$.

\label{sec:MEGammaRf}
Similar to its genuine counterpart, we can also consider the action groupoid $\GammaR^{f,N} \coloneqq J^\aorder\Gamma^N \ltimes \SR(N)$, where $N$ is an $n$-manifold. We claim that also these two groupoids are Morita equivalent.
\begin{lemma} \label{lem:MEGammaRf}
    Let $N$ be an $n$-manifold. The groupoids $\GammaR^{f,N}$ and $\GammaR^f$ are Morita equivalent.
\end{lemma}
\begin{proof}
    We proceed as in the proof of \cref{lem:MEGammaRNn}. We claim that $J^\aorder\Gamma(N, \R^n) \times_{N} \SR(N)$ is a left principal $\GammaR^f$-bundle over $\SR(N)$ and a right principal $\GammaR^{f,N}$-bundle over $\SR(\R^n)$. The moment maps and the actions are defined analogous to the proof of \cref{lem:MEGammaRNn}. This makes $J^\aorder\Gamma(N, \R^n) \times_{N} \SR(N)$ into a $\GammaR^f$-$\GammaR^{f,N}$-bibundle.
\end{proof}

\subsection{Formal \texorpdfstring{$\SR$}{R}-microbundles}\label{sec:formalMb}
In the formal case, there is still a correspondence between $\GammaR^f$-cocycles and principal $\GammaR^f$-bundles. However, the definition of a $\SG$-microbundle does not make sense for $\SG = \GammaR^f$, because the groupoid is not \'etale (and hence also not effective). Therefore, we directly define the following formal analogue of an $\SR$-microbundle. 

\begin{definition}
	A \textbf{formal $\SR$-microbundle} $(E,F)$ is a microbundle $E$ with a family $F_x : \iota(x) \to \SR(E_x)$ of formal solutions varying continuously in $x\in M$ for the Whitney topology.
\end{definition}

An important difference between formal and genuine $\SR$-microbundles is that a formal $\SR$-microbundle does not come with a canonical foliation. On the other hand, each genuine $\SR$-microbundle canonically defines a formal $\SR$-microbundle. Hence we obtain the following:
\begin{lemma}
A genuine $\SR$-microbundle is a formal $\SR$-microbundle. 
\end{lemma}
\begin{proof}
    A genuine $\SR$-microbundle consists of a microbundle $E\to M$ with an $\SR$-foliation that is transverse to the fibers of $E$. Each fiber $E_x$ is therefore a transversal of the foliation and so we obtain using \cref{def:RTransStrTransversal} a family of solutions $f_x : E_x \to \SR$. By taking fiberwise jets along $\iota(M)$, we obtain the family of formal solutions $F_x : \iota(x) \to \SR(E_x)$ that varies continuously for the Whitney topology in $x \in M$.
\end{proof}

A formal solution $F$ on a manifold $M$ defines a formal $\SR$-microbundle $\exp(F)$ analogous to how a solution $f$ defines an $\SR$-microbundle (as discussed in \cref{ex:SolToBundle}):
\begin{example} \label{ex:SolToBundleFormal}
     Let $F$ be a formal solution on a manifold $M$. We use the exponential $\exp: TM \to M$ to pull back $F$ to $TM$ to obtain a family of formal solutions $\exp^*(F|_{\Op(x)}) : T_xM \to \SR$. This defines a formal $\SR$-microbundle which we denote by $\exp(F)$. We note that, even though in general this does not need to be the case, the formal $\SR$-microbundle $\exp(F)$ does come with a canonical foliation by the fibers of $\exp$.
\end{example}

We now prove the correspondence between $\GammaR^f$-cocycles and formal $\SR$-microbundles, the formal analogue of \cref{lem:corPGRB-Rmb}. Most of the work has already been carried out in \cref{sec:JrGammaToMb}, where we show that any principal $J^\aorder \Gamma$-bundle lifts to a microbundle. Hence below we only need to endow this microbundle with a family of formal solutions. 
Since the lift to a microbundle is only defined up to concordance, the correspondence below is as well.

\begin{lemma} \label{lem:corPGRB-fRmb}
	There is a one-to-one correspondence between concordance classes of $\GammaR^f$-cocycles and concordance classes of formal $\SR$-microbundles.
\end{lemma}
\begin{proof}	
    We define the maps $\pr_1: \GammaR^f \to J^\aorder \Gamma^n$ and $\pr_2: \GammaR^f \to \SR(\R^n)$. Let $\{\gamma_{ij}: U_i \cap U_j \to \GammaR^f\}$ be a cocycle representing a given principal $\GammaR^f$-bundle. We recall from \cref{lem:JrGPbToMb} that the $J^\aorder\Gamma$-cocycle $\pr_1 \circ \gamma_{ij}$ lifts to a smooth microbundle $E \to M$, which is unique up to concordance. The inclusion $\iota: M \to E$ is, locally over each $U_i$, given by $m \mapsto (m,\gamma_{ii}(m))$. Hence we define the family of formal solutions $F_x: \iota(x)\to \SR(E_x)$ for $x\in U_i$ as $x \mapsto \pr_2 \circ \gamma_{ii}(x)$. Since the $\gamma_{ij}$ form a cocycle, the families of formal solutions over each $U_i$ glue to a family over $M$.

    The other direction is analogous to the proof of \cref{lem:corPGRB-Rmb} with the adaptation that the transition functions are no longer constant and we have to take jets of the resulting $\Gamma^\Whitney$-cocycle.
\end{proof}

\subsection{The scanning map} \label{ssec:scanningMap}

Following the h-principle philosophy, we are interested in studying the \textbf{scanning map} 
\begin{equation*} \label{eq:scanningBGamma}
\scanR: \GammaR \rightarrow \GammaR^f,
\end{equation*}
which generalizes the usual inclusion of solutions into formal solutions. This map is defined by taking the $r$-jet of the solution. That is, given an element $([f]_x,[s]_x) \in \GammaR$ we map this to $(j^\aorder_x(f),j^r_x(s)) \in J^\aorder\Gamma^n \ltimes \SR(\R^n)$.

Due to the functoriality of the classifying space construction we also have a scanning map
\[ B\GammaR \rightarrow B\GammaR^f, \]
which is unique up to homotopy and which we still denote by $\scanR$ if needed. Similarly, we obtain classifying maps for the spaces of principal bundles and microbundles over a given manifold $M$:
\begin{equation} \label{eq:scanningPrincipal}
\Maps(M,B\GammaR) \rightarrow \Maps(M,B\GammaR^f),
\end{equation}
\begin{equation} \label{eq:scanningMicro}
\Maps_E(M,B\GammaR) \rightarrow \Maps_E(M,B\GammaR^f).
\end{equation}

Let us elaborate on the last item. We first recall from \cref{sec:spaceRMb} that $\Maps_E(M,B\GammaR)$ denotes the space of $\SR$-microbundles with underlying vector bundle $E \to M$. Observe that $\GammaR \rightarrow \GammaR^f$ factors the map $\GammaR \rightarrow \GL(n)$. Differently put, the associated microbundle (without $\SR$-foliation) depends only on first order formal data. This means that the map appearing in \cref{eq:scanningMicro} is the induced map on homotopy fibers of the map in \cref{eq:scanningPrincipal} over $\Maps(M,B\GL(n))$. Hence, in order to understand the connectivity of \cref{eq:scanningPrincipal} we just need to understand the connectivity of \cref{eq:scanningMicro}, or vice versa.

\subsection{Tautological \texorpdfstring{$\GammaR^f$}{Γ\^{}f\_R}-bundles}
\label{sec:tautBundleF}
In \cref{sec:tautBundle} we constructed the tautological $\GammaR$-cocycle over $\EtSol{N}$. We now proceed similarly for $\GammaR^f$ over $\SR(N)$. The main difference is however that $\SR(N)$ is a (Hausdorff and second-countable) manifold, implying that in this setting there is a correspondence between $\GammaR^f$-cocycles, principal $\GammaR^f$-bundles and formal $\SR$-microbundles over $\SR(N)$.

\subsubsection{The tautological principal \texorpdfstring{$\GammaR^f$}{Γ\^{}f\_R}-bundle}
The groupoid $\GammaR^{f,N}$ (from \cref{sec:MEGammaRf}) is a groupoid with base space $\SR(N)$. Hence we can consider the unit principal $\GammaR^{f,N}$-bundle over $\SR(N)$.
By the Morita equivalence of $\GammaR^{f,N}$ and $\GammaR^{f}$, we obtain as in \cref{pro:MePrBundlesSame} a $\GammaR^{f}$-bundle on $\SR(N)$. Since we started with a unit bundle, it is the bibundle $J^\aorder\Gamma(N, \R^n) \times_{N} \SR(N)$ from \cref{lem:MEGammaRf} itself that forms the $\GammaR^f$-bundle over $\SR(N)$. 

\begin{definition}\label{def:tautF}
    The \textbf{tautological principal $\GammaR^f$-bundle} $\tautPbF{N}$ over $\SR(N)$ is the principal $\GammaR^f$-bundle $J^\aorder\Gamma(N, \R^n) \times_{N} \SR(N) \to \SR(N)$. 
    We refer to the corresponding equivalence class of $\GammaR^f$-cocycles (as in \cref{pro:corCocyclePrBundle}) as the equivalence class of \textbf{tautological $\GammaR^f$-cocycles} and to the corresponding concordance class of formal $\SR$-microbundles (as in \cref{lem:corPGRB-fRmb}) as the concordance class of \textbf{tautological formal $\SR$-microbundles}. 
\end{definition}
When we fix a representative tautological formal $\SR$-microbundle, we refer to it as \emph{a} tautological formal $\SR$-microbundle and denote it by $\tautMbF{N} \to \SR(N)$.

\subsubsection{Tautological \texorpdfstring{$\GammaR^f$}{Γ\^{}f\_R}-cocycles and formal \texorpdfstring{$\SR$}{R}-microbundles}
We now provide a more explicit description of a tautological $\GammaR^f$-cocycle. To this end, we fix an atlas $\{\phi_i: U_i \to \R^n\}_{i \in I}$ of $N$ with transition functions $\phi_{ij} = \phi_i \circ \phi_j^{-1}:\phi_j(U_i \cap U_j) \to \phi_i(U_i \cap U_j)$, and define the opens $V_i = (p_b)^{-1}(U_i) \subset \SR(N)$.
\begin{lemma}
    A tautological $\GammaR^f$-cocycle $\{\tautcofij : V_i \cap V_j \to \GammaR^f \}_{i,j \in I}$ on $\SR(N)$ is defined by
\begin{align*}
    \tautcofij(j^r_x s) &= \left( j^\aorder_{\phi_j(x)} \phi_{ij}, \left( \phi_j \right)_* \left( j^r_x s \right) \right) 
\end{align*}
for all $j^r_x s \in \SR(N)$.
\end{lemma}
We recall from \cref{sec:JrGammaToMb} that a $J^\aorder\Gamma$-cocycle does not define a microbundle, but it may lift to a $\Gamma^\Whitney$-cocycles that does. We define the \textbf{tautological lift} of the tautological $\GammaR^f$-cocycle as the cocycle $\{ \tautcofliftij :  V_i \cap V_j \to \Gamma^\Whitney \ltimes \SR(\R^n) \}_{i,j \in I}$ defined by 
    \begin{align*}
        \tautcofliftij(j^r_x s) = \left(   [\phi_{ij}]_{\phi_j(x)} , \left( \phi_j \right)_* \left( j^r_x s \right)\right),
    \end{align*}
which indeed defines a microbundle by \cref{lem:PbToMbWhitney}. In fact, when inspecting the proof of \cref{lem:corPGRB-fRmb} we obtain the following.
\begin{lemma}
    A representative of a tautological formal $\SR$-microbundle is given by
    \begin{itemize}
        \item the vector bundle $p_b^* (TN)$, where $p_b$ is the projection $\SR \to N$,
        \item the inclusion $\iota: \SR(N) \to \tautMbF{N}$ given by the zero section, and
        \item the family of formal solutions $F_{j^r _x s}:\iota(j^r _x s) \to \SR(\Op(x))$ defined by $\iota(j^r _x s) \mapsto j^r _x s$.
    \end{itemize}
\end{lemma}

\subsubsection{Pullbacks of formal tautological bundles} \label{ssec:PullbackTautFormalCocycle}

The observations we made in \cref{ssec:PullbackTautCocycle} for solutions also hold for \emph{formal} solutions $F$. Hence we obtain the formal counterparts of \cref{cor:SolToBundleTaut,cor:SectionToBundleTaut}.

\begin{corollary} \label{cor:FormalSolToBundleTaut}
    Let $F: N \to \SR(N)$ be a formal solution. The formal $\SR$-microbundle $F^*( \tautMbF{N})$ is equivalent to the formal $\SR$-microbundle $\exp(F)$.
\end{corollary}
\begin{proof}       
    When we pullback the tautological lift of the tautological $\GammaR^f$-cocycle using $F$, we obtain the cocycle $\{ \delta_{ij} : U_i \cap U_j \to \Gamma^\Whitney \ltimes \SR(\R^n) \}$ defined by 
    \begin{align*}
        \delta_{ij}(x) = \left( [\phi_{ij}]_{\phi_j(x)},  \left( \phi_j \right)_* \left( F(x) \right)  \right). 
    \end{align*}
    Hence the underlying $\Gamma$-cocycle sends $x \in N$ to $\left[ \phi_{ij} \right]_{\phi_j(x)}$, which defines the tangent bundle as microbundle (recall \cref{ex:MBtangentbundle}). Since the families of formal solutions along the zero section agree, the formal $\SR$-microbundles are equivalent. 
\end{proof}

The proof of \cref{cor:FormalSectionToBundleTaut} is analogous to the proof of \cref{cor:SectionToBundleTaut} with the adaptation of using the tautological lift as we did in the proof of \cref{cor:FormalSolToBundleTaut}.
\begin{corollary} \label{cor:FormalSectionToBundleTaut}
    Let $G: N \to \SR(N)$ be a map such that $p_b\circ G:N\to N$ is homotopic to the identity. The formal $\SR$-microbundle $G^* (\tautMbF{N})$ has an underlying microbundle that is equivalent to $TN$.
\end{corollary}

\section{Wrinkling h-principles for \texorpdfstring{$\SR$}{R}-microbundles} \label{sec:wrinklingHaefliger}

We now prove the results stated in the introduction regarding $\SR$-microbundles, as well as their consequences regarding the space of principal $\GammaR$-bundles and the connectivity of the classifying space $B\GammaR$. Throughout this section, $\SR \subset J^r\Psi$ is an open and $\Diff$-invariant relation of dimension $n$.

\subsection{From formal sections to \texorpdfstring{$\SR$}{R}-microbundles}

In the case that $\dim(M)=n$, we recall from \cref{ex:SolToBundleFormal} that any formal solution $F: M \to \SR$ defines a formal $\SR$-microbundle $\exp(F)$. \Cref{thm:wrinklingMicro} states that every such bundle can be homotoped to produce an $\SR$-microbundle. The resulting $\SR$-microbundle is the pullback of the tautological $\SR$-microbundle on $\EtSol{M}$, via the map produced by \cref{thm:wrinklingEtale}. An illustration of the proof can be found in \cref{fig:Haefliger}.

\wrinklingMicro

\begin{figure}[h]
    \centering
    \includegraphics[width=0.9\textwidth]{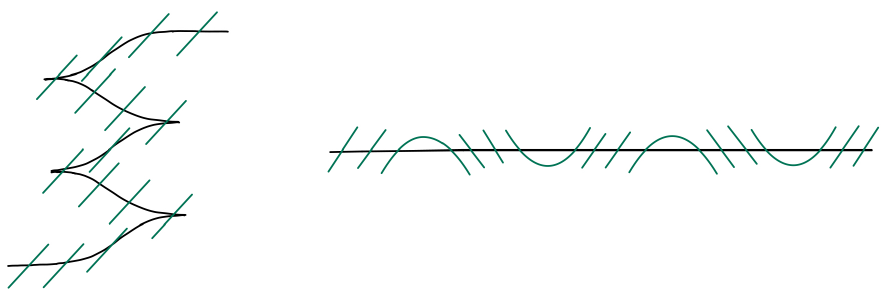} 
    \caption{A depiction of the proof of \cref{thm:wrinklingMicro}. On the left, we see the \'etale space $\EtSol{M}$ lying over $M$. The zig-zag shown is the map $G$ produced by \cref{thm:wrinklingEtale}. When restricted to the image of $G$, the tautological $\GammaR$-cocycle over $\EtSol{M}$ defines a tautological $\SR$-microbundle, which we imagine as a pullback of $TM$ foliated by the fibers of the exponential map. These fibers are drawn as the little diagonal segments in turquoise. We pullback this $\SR$-microbundle to $M$ itself, via the map $G$. Since $G$ is wrinkled, the pullback is not regular but instead exhibits wrinkled singularities, as depicted on the right. } \label{fig:Haefliger}
\end{figure}

\begin{proof}
    According to \cref{cor:wrinklingEtale}, there exists a homotopy of maps $H_t : M \to \SR$, a homotopy $\psi_t : M \to M$ and a wrinkled submersion $G: M \rightarrow \EtSol{M}$ such that $H_0 = F$, $\psi_0$ is the identity, $\psi_1$ is a wrinkled submersion, and the following diagram commutes
    \[
        \begin{tikzcd}
            & & \EtSol{M} \arrow[d,"p_r"] \\
            & & \SR \arrow[d,"p_b"] \\
            M \times \{1\} \arrow[r,hookrightarrow] \arrow[uurr,"G"] & M \times [0,1] \arrow[r,"\psi"] \arrow[ur,"H"] & M.
        \end{tikzcd}
    \]

    We now pullback the tautological formal $\SR$-microbundle on $\SR$ via $H$. This yields a formal $\SR$-microbundle on $M \times [0,1]$ starting at $\exp(F)$, as observed in \cref{cor:FormalSolToBundleTaut}. Since $G$ lifts $H_1$, the formal $\SR$-microbundle on $M \times \{1\}$ lifts to a genuine $\SR$-microbundle. The underlying microbundle is equivalent to $TM$ by \cref{cor:SectionToBundleTaut}. 
\end{proof}

Its parametric version, where we start with a family of formal solutions, follows by applying the arguments in the proof of \cref{thm:wrinklingMicro} parametrically. We recall that a $K$-family of microbundles $(E,\SF_k,F_k)$ is continuous for the Whitney topology, if the foliations $\SF_k$ form a distribution on the microbundle $E \times K \to M \times K$ and the solutions $F_k$ vary continuously for the Whitney topology. 

\wrinklingMicroParametric
\begin{proof}
    Since this theorem is the parametric and relative version of \cref{thm:wrinklingMicro}, its proof amounts to repeating the argument in the proof of \cref{thm:wrinklingMicro}, invoking instead of \cref{cor:wrinklingEtale}, its parametric and relative version \cref{cor:wrinklingEtaleParametric}. There is nonetheless one key observation to be made:

    \Cref{cor:wrinklingEtaleParametric} produces families in the \'etale space of solutions that are individually continuous for the \'etale topology and Whitney continuous in the parameter. This means that whenever we pullback the universal solution in $\EtSol{M}$, we obtain families of principal $\GammaR$-bundles that are continuous for the Whitney topology, but not necessarily for the \'etale one. This is however precisely what the statement claims.
\end{proof}

\subsection{Existence of \texorpdfstring{$\SR$}{R}-microbundles}
In this section we generalize \cref{thm:wrinklingMicro} by starting with a formal $\SR$-microbundle, instead of a formal solution. In particular we also allow for the relation $\SR\subset J^r\Psi$ to be of a different dimension $n$ than the dimension $m$ of the manifold $M$.

We recall that a formal $\SR$-microbundle $(E,F)$ does not come with a foliation $\SF$. However, due to work of Haefliger \cite{Haef2} on the connectivity of the map $B\Gamma \to B \GL(n)$, we know that the formal $\SR$-microbundle $(E,F)$ admits a foliation $\SF$ transverse to the fibers of $E$ with $\corank(\SF)=\rank(E)$, as long as $\dim(M) \leq n$. Compared to the proof of \cref{thm:wrinklingMicro}, the difference now is that $F$ does not need to be invariant under the holonomy of $\SF$, and that $(E \to M,\SF)$ need not be regular. It was, however, transversality with respect to the zero section that allowed us to apply our wrinkling from \cref{thm:wrinklingEtale}. Hence the first step of the proof below will be to achieve transversality. After that we argue again per simplex, using \cref{thm:wrinklingEtale} for the top dimensional simplices.

\HaefligerNonTangential
\begin{proof}
Let $(E,F)$ be a formal $\SR$-microbundle over $M$ and let $\SF$ be a foliation of $E$ transverse to its fibers with $\corank(\SF) =  \rank(E)$ \cite{Haef2}. We extend the family of formal solutions to a smooth family of formal solutions defined on the fibers $E_x$, and write $F_x:E_x \to \SR$. Using an improved version of Thurston's jiggling~\cite{FPTjiggling} we obtain a piecewise smooth section $\tilde{\iota}: M \rightarrow E$ with respect to a triangulation of $M$ such that $\tilde \iota$ is piecewise transverse to both $\SF$ and the fibers of $E$. Since one step for jiggling is to subdivide the simplices such that their diameter decreases uniformly, we can assume that the simplices are arbitrarily small. We then proceed as in the proof of \cref{prop:connectivityEtale}, arguing inductively on the dimension of the simplices:

\begin{figure}
    \centering
    \includegraphics[width=0.7\linewidth,page=6]{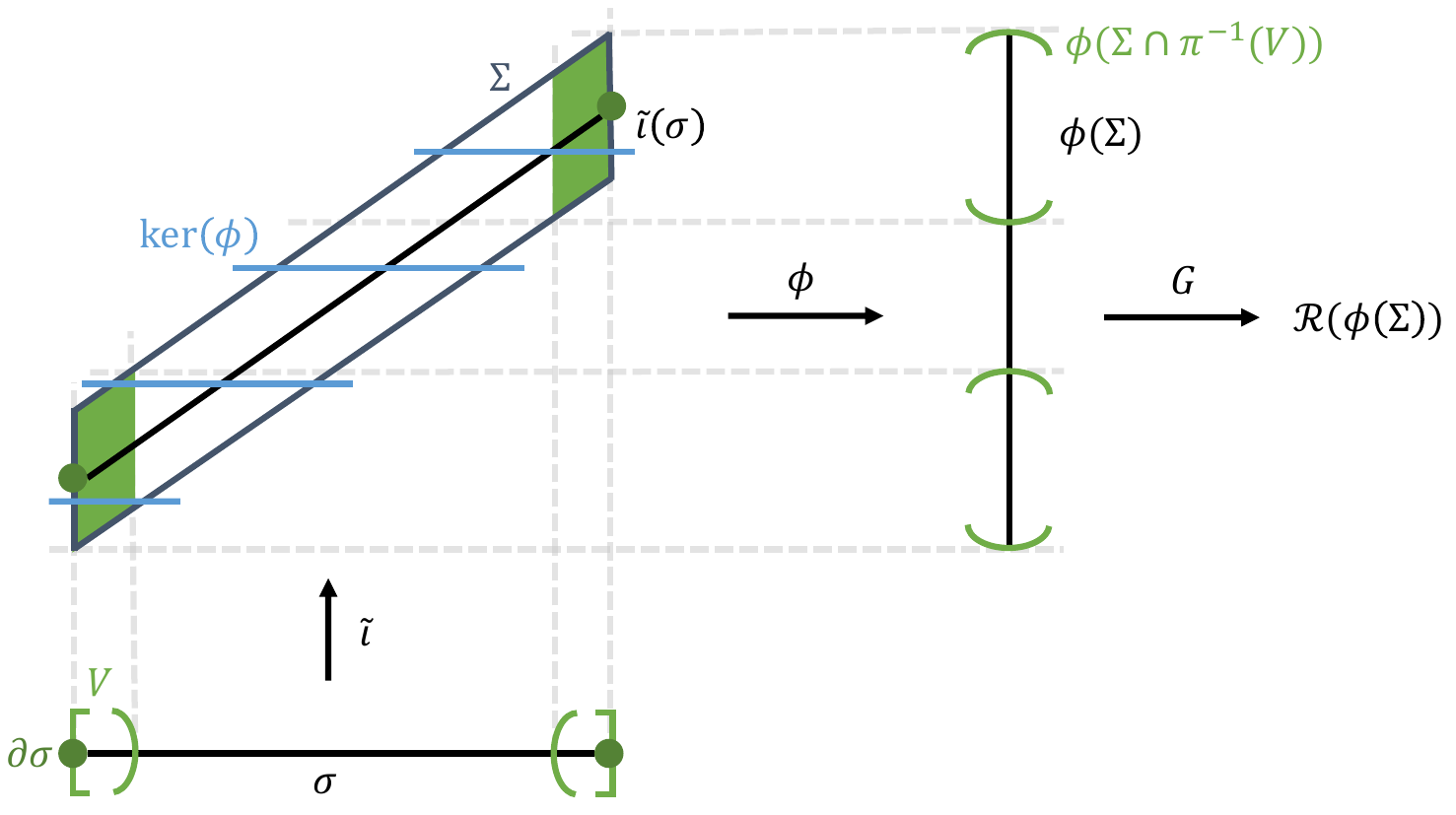}
    \caption{A sketch of the proof of \cref{prop:HaefligerNonTangential}. The green regions indicate where the bundle already lifts to an $\SR$-microbundle, and where $G$ is hence already holonomic.}
        \label{fig:existenceRmb}
\end{figure}

We first assume that we are given a simplex $\sigma \subset M$ of dimension strictly smaller than $n$, and that $\pi|_{\pi^{-1}(\sigma)}:E|_{\sigma}\to \sigma$ already lifts to an $\SR$-microbundle over a neighborhood $V \subset \sigma$ of $\partial \sigma$. We have illustrated this, and the upcoming part of the proof in \cref{fig:existenceRmb}. We thicken the image $\tilde\iota(\sigma)$ to a small open neighborhood $\Sigma$ in $\pi^{-1}(\sigma) \subset E$. We can assume that $\Sigma$ is contained in a foliation chart $\phi:\Op(\Sigma) \to \sigma \times \R^n$. Since $E$ already lifts to a microbundle over $V$ we know that the family $F_x : \Sigma \cap E_x \to \SR$, when restricted to $x\in V$, consists of holonomic sections and is invariant under holonomy. Hence the $F_x$ agree on the fibers of $\phi$. Because we have assumed $\sigma$ to be small, we can also find a formal solution $G: \phi(\Sigma) \to \SR$ such that
\begin{enumerate}[label=(\arabic*)]
    \item $G (\phi(x,y)) = \phi_* F_x(y)$ for all $(x,y) \in \iota(\sigma)$,
    \item \label{it:hol} $G|_{\phi(E_x)} = \phi_* F_x$ for $x\in V$, and
    \item $(\phi|_{\Sigma\cap E_x})^*G$ and $F_x$ are homotopic as a $\sigma$-family of formal solutions relative to $V$.
\end{enumerate}
We note that \cref{it:hol} implies in particular that $G$ is holonomic over $\phi(\Sigma \cap \pi^{-1}(V))$. Since $\SR$ is open and $\Diff$-invariant, and $\phi(\tilde\iota \sigma)$ has positive codimension, we know the $h$-principle holds \cite[Theorem 8.3.1]{CiElMi} so that we obtain a section $\tilde G : \phi(\Sigma) \to \SR$ such that $G$ and $\tilde G$ are homotopic relative to $\phi(\Sigma \cap \pi^{-1}(V))$, and such that $G$ is holonomic over a neighborhood $\Sigma'$ of $\phi(\tilde \iota(\sigma))$.

We now pullback $\tilde G$ to $\tilde \iota(\sigma)$ by defining $ \tilde F_x : \Sigma'\cap E_x \to \SR$ for $x\in \sigma $ as $\tilde F_x = (\phi|_{\Sigma'\cap E_x})^* G$. We define solutions $ f_x: E_x \to \Psi$, by requiring that $j^r f_x (y)=\tilde F_x(y) $ for $(x,y) \in \Sigma'$ and that the $f_x$ are constant on the fibers of $\phi$, implying that they are holonomy invariant. This defines an $\SR$-microbundle over $\sigma$, homotopic to the formal $\SR$-bundle we started with, relative to $\partial \sigma$. 

In the case that $\dim(M)<n$, we can reason as above for all simplices in $M$. If instead $\dim (M) =n$, we argue for a top-dimensional simplex $\sigma$ as follows. We again find a formal section $G:\phi(\Sigma) \to J^r(\Psi)$ as above, but now apply \cref{thm:wrinklingEtale} to it. This results in a function $\tilde G: \phi(\Sigma) \to \EtSol{\R^n}$ together with a homotopy $H_t : \phi(\Sigma) \to \SR$ of formal solutions relative to $\cup_{x \in V} \phi(\Sigma \cap E_x)$, starting at $G$ and ending at $p_r \circ \tilde G : \phi (\Sigma) \to \SR$.

We now pullback the tautological formal $\SR$-microbundle on $\SR$ by $H_t$, to obtain a formal $\SR$-microbundle over $M \times I$. Since $H_1$ lifts to $\tilde G$, the formal $\SR$-microbundle over $M\times\{1\}$ lifts to a genuine $\SR$-microbundle. Hence the microbundle over $M \times I$ defines a concordance (and hence homotopy) between the given formal $\SR$-microbundle over $M \times \{0\}$ and a genuine $\SR$-microbundle over $M\times \{1\}$, relative to $\partial(\sigma)$.

To conclude the proof, we note that we have defined a $\sigma$-family of solutions on $\Sigma$ that are invariant under the holonomy of $\SF$. Hence we smoothen $\tilde{\iota}|_\sigma$ to $\hat\iota:\sigma \to E|_\sigma$ such that $\hat{\iota}(\sigma) \subset \Sigma$. This defines an $\SR$-microbundle over the top-dimensional simplex $\sigma$ homotopic to the given formal $\SR$-microbundle. The smoothening will in general introduce singularities of the foliation with respect to $M$.
\end{proof}

Its Whitney parametric and relative version again follows by making slight changes to the proof:
\HaefligerParametricNonTangential
\begin{proof}
    We first deal with the relative case in just the domain. For this we just repeat the proof of \cref{prop:HaefligerNonTangential} but relative to a neighborhood $V$ of $M'$. That is, we triangulate relative to $V$, jiggle relative to $V$ and subsequently apply holonomic approximation and \cref{thm:wrinklingMicroParametric} relative to $V$.
    
    The parametric case we prove this analogous to \cref{prop:HaefligerNonTangential} by introducing parameters. The one subtlety is that we now additionally want triangulate the product $M \times K$ in a manner that is transverse to the fibers of $M \times K \rightarrow K$, which can also be achieved by an improved version of jiggling~\cite{FPTjiggling}. This implies that holonomic approximation can be applied parametrically in all smaller dimensional cells. \Cref{thm:wrinklingMicroParametric} is invoked to deal with the top-dimensional cells.
\end{proof}

\subsection{The h-principle for principal \texorpdfstring{$\GammaR$}{Γ\_R}-bundles with the \'etale topology}
The previous results allow us to study the connectivity of the map $B\GammaR \to B \GammaR^f$. More generally we can study the connectivity of the scanning map 
\begin{align*}
    \Maps(M, B\GammaR) &\to \Maps(M, B\GammaR^f)
\end{align*}
which was introduced in \cref{ssec:scanningMap}.

\label{sec:hprincipleMBEtale}
Since a $K$-family of maps $M \to B\GammaR$, that is continuous for the the \'etale topology, corresponds to a principal $\GammaR$-bundle over $M \times K$, the following follows almost directly from \cref{prop:HaefligerNonTangential}. 

\cocycleConnectivityEtale
\begin{proof}
	We first prove surjectivity in the $i$th homotopy group, with $i \leq n-m$. Represent a given homotopy class in $\Maps(M,B \GammaR^f)$ by a map $S^i \to \Maps(M,B \GammaR^f)$. The map corresponds to an element of $\Maps(S^i \times M, B\GammaR^f)$ and hence to a formal $\SR$-microbundle on $S^i \times M$. We then apply \cref{prop:HaefligerNonTangential} to find a homotopy to an $\SR$-microbundle over $S^i \times M$. This now corresponds to an element of $\Maps( S^i \times M, B\GammaR^f)$ and hence to a map $S^i \to \Maps(M,B \GammaR^f)$.
	
	The proof of injectivity for $i < n$ is similar. We take a $D^{i+1}$-family of maps $M \to B\GammaR^f$, that restricts to the boundary as a $S^i$-family of $M \to B\GammaR$-bundles. We now argue as above, but relative to $S^i \times M$, using again \cref{prop:HaefligerNonTangential}. 
\end{proof}

Together with the observation in \cref{ssec:scanningMap}, this is equivalent to:
\begin{corollary}
	Let $M$ be an $m$-dimensional manifold and let $\SR$ be an open and $\Diff$-invariant relation of dimension $n \geq m$. Let $E \rightarrow M$ be a vector bundle. The scanning map 
	\[ \Maps_E(M,B\GammaR) \rightarrow \Maps_E(M,B\GammaR^f) \]
	is $(n-m)$-connected.
\end{corollary}



\printbibliography
\end{document}